\documentclass[a4paper,12pt]{amsart}

\usepackage[english]{babel}
\usepackage[T1]{fontenc}
\usepackage{mathrsfs}
\usepackage{array}
\usepackage{makecell}

\usepackage{hyperref}
\usepackage{cleveref}

\usepackage[a4paper,margin=2.5cm]{geometry}
\usepackage[justification=centering]{caption}

\usepackage{enumerate}
\usepackage{tabularx} 
\usepackage{amsmath} 
\usepackage{graphicx}
\usepackage{amsmath,amscd,amsbsy,amssymb,latexsym,url,bm,amsthm,mathrsfs}
\usepackage{epsfig,graphicx,subfigure}
\usepackage{enumerate,balance,mathtools}
\usepackage{wrapfig}
\usepackage[vlined,boxed,commentsnumbered,linesnumbered,ruled]{algorithm2e}
\usepackage{amsfonts}
\usepackage{amsthm}
\usepackage{fancyhdr}
\usepackage{picinpar}
\usepackage{listings}
\usepackage{layout}
\usepackage[all]{xy}
\usepackage{indentfirst}
\usepackage{tikz-cd}
\usepackage{tikz}
\usepackage[all]{xy}
\usepackage{mathrsfs}
\usepackage{hyperref}
\hypersetup{
colorlinks=true,
linkcolor=[RGB]{0,0,204},
filecolor=magenta,      
urlcolor=cyan,
citecolor=[RGB]{0,204,0},
}
\usepackage{setspace}

\usepackage{amsmath,amsthm,amssymb,amsxtra,calligra,mathrsfs}
\usepackage{graphicx}
\usepackage{tikz}
\usepackage{tikz-cd} 
\usepackage{xcolor}
\usepackage{upgreek}
\usepackage{comment}
\usepackage{graphicx}
\usepackage{mathtools}
\usepackage{extarrows}
\usepackage{faktor}
\usepackage{mathrsfs,euscript}
\usepackage{comment}
\usepackage{bookmark}
\usepackage{mathrsfs}
\usepackage{multirow}

\newcommand{\wt}[1]{\widetilde{#1}}

\newcommand{\mb}[1]{\mathbb{#1}}

\newcommand{\ove}[1]{\overline{#1}}

\newcommand{\mtc}[1]{\mathcal{#1}}
\newcommand{\mtf}[1]{\mathfrak{#1}}
\newcommand{\mts}[1]{\mathscr{#1}}

\newcommand{\cQ}{\mathcal{Q}}

\DeclareMathOperator{\tw}{tw}

\DeclareMathOperator{\st}{st}
\DeclareMathOperator{\PGL}{PGL}

\DeclareMathOperator{\Spec}{Spec}
\DeclareMathOperator{\univ}{univ}
\DeclareMathOperator{\Proj}{Proj}

\DeclareMathOperator{\mult}{mult}

\DeclareMathOperator{\CY}{CY}
\DeclareMathOperator{\Supp}{Supp}

\DeclareMathOperator{\ord}{ord}

\DeclareMathOperator{\lct}{lct}

\DeclareMathOperator{\sst}{ss}
\DeclareMathOperator{\Id}{Id}
\DeclareMathOperator{\NS}{NS}

\DeclareMathOperator{\GIT}{GIT}

\DeclareMathOperator{\Aut}{Aut}

\DeclareMathOperator{\sh}{sh}

\DeclareMathOperator{\sm}{sm}

\DeclareMathOperator{\Hodge}{Hodge}

\newcommand{\bF}{\mathbb{F}}

\newcommand{\bG}{\mathbb{G}}

\newcommand{\sF}{\mathscr{F}}

\newcommand{\sslash}{\mathbin{\mkern-3mu/\mkern-6mu/\mkern-3mu}}

\newcommand{\sheafHom}{\mathscr{H}\text{\kern -3pt {\calligra\large om}}\,}

\DeclareMathOperator{\KSBA}{\mathrm{KSBA}}

\newcommand{\bV}{{\mathbb V}}

\newcommand{\bA}{{\mathbb A}}
\newcommand{\bZ}{{\mathbb Z}}

\newcommand{\bC}{{\mathbb C}}
\newcommand{\bT}{{\mathbb T}}
\newcommand{\bQ}{{\mathbb Q}}
\newcommand{\bP}{{\mathbb P}}
\newcommand{\bR}{{\mathbb R}}

\usepackage{longtable}
\allowdisplaybreaks
\usepackage{xcolor}

\newcommand{\sY}{{\mathscr{Y}}}
\newcommand{\sU}{{\mathscr{U}}}

\newcommand{\sP}{{\mathscr{P}}}

\newcommand{\sD}{{\mathscr{D}}}

\newcommand{\sZ}{{\mathscr{Z}}}
\newcommand{\sX}{{\mathscr{X}}}
\newcommand{\sC}{{\mathscr{C}}}
\newcommand{\sV}{{\mathscr{V}}}
\newcommand{\sB}{{\mathscr{B}}}

\newcommand{\cX}{{\mathcal{X}}}
\newcommand{\cA}{{\mathcal{A}}}
\newcommand{\cB}{{\mathcal{B}}}
\newcommand{\cY}{{\mathcal{Y}}}
\newcommand{\cG}{{\mathcal{G}}}
\newcommand{\cN}{{\mathcal{N}}}
\newcommand{\cO}{{\mathcal{O}}}
\newcommand{\cR}{{\mathcal{R}}}

\newcommand{\cU}{{\mathcal{U}}}
\newcommand{\cV}{{\mathcal{V}}}
\newcommand{\bmu}{\bm{\mu}}
\newcommand{\oH}{{\operatorname{H}}}
\definecolor{darkgreen}{rgb}{0.0, 0.5, 0.0}

\DeclareMathOperator{\can}{can}
\DeclareMathOperator{\norm}{n}
\DeclareMathOperator{\Exc}{Ex}
\DeclareMathOperator{\bfM}{{\bf{M}}}

\DeclareMathOperator{\sn}{sn}

\DeclareMathOperator{\Gal}{Gal}

\newcommand\spec{\text{\rm Spec}}

\newcommand\cD{{\mathcal{D}}}
\newcommand\cS{{\mathcal{S}}}

\newcommand\cE{{\mathcal{E}}}

\newcommand\cC{{\mathcal{C}}}
\newcommand\cM{{\mathcal{M}}}

\newtheorem{theorem}{Theorem}[section]

\newtheorem{lemma}[theorem]{Lemma}

\newtheorem{corollary}[theorem]{Corollary}
\newtheorem{prop}[theorem]{Proposition}

\newtheorem*{notation}{Notation}

\newtheorem{remark}[theorem]{Remark}

\newtheorem{setup}[theorem]{Setup}

\newcommand{\dm}{{\operatorname{DM}}}

\theoremstyle{definition}
\newtheorem{defn}[theorem]{Definition}

\newtheorem{example}[theorem]{Example}

\theoremstyle{remark}

\title{Moduli of surfaces fibered in log Calabi-Yau pairs}
\date{\today}

\author{Giovanni Inchiostro}
\address{C-138 Padelford, Box 354350, Seattle, WA 98195, USA}
\email{ginchios@uw.edu}

\author{Roberto Svaldi}
\address{Dipartimento di Matematica ``F. Enriques'', Via Saldini 50, Milano (MI) 20133, Italy}
\email{roberto.svaldi@unimi.it}

\author{Junyan Zhao}
\address{William E. Kirwan Hall, 4176 Campus Dr, College Park, MD 20742, USA}
\email{jzhao81@umd.edu}

\begin{document}

\begin{abstract}
We study the moduli spaces of surface pairs $(X,D)$ admitting a log Calabi--Yau fibration \((X,D) \rightarrow C.\) We develop a series of results on stable reduction and apply them to give an explicit description of the boundary of the KSBA compactification. Three interesting cases where our results apply are:
\begin{enumerate}
    \item divisors on \(\mathbb{P}^1 \times \mathbb{P}^1\) of bidegree \((2n,m)\); 
    \item K3 surfaces which map $2:1$ to $\bF_n$, with $X=\bF_n$ and $D$ the ramification locus;
    \item elliptic surfaces with either a section or a bisection.
\end{enumerate}
The main tools employed are stable quasimaps, the canonical bundle formula, and the minimal model program.
\end{abstract}

\maketitle

\tableofcontents

\section{Introduction}

\begin{comment}
    In this paper we use KSBA-stability to study a compactification of the moduli space of surface pairs $\pi\colon (X,\frac{1}{n}D)\to C$ which are fibered in log-canonical pairs of the form $(\bP^1,\frac{1}{n}(p_1+\ldots+p_{2n}))$, under the following two mild assumptions:
\begin{enumerate}
\item the divisor $D$ restricted to the generic fiber of $\pi$ consists of $2n$ distinct points, and
    \item the pair $(X,(\frac{1}{n}+\epsilon_1)D+\epsilon_2F)$ is KSBA-stable for every $0<\epsilon\ll\epsilon_2\ll 1$, where $F$ are the fibers of $\pi$ intersected by $D$ it less than $2n$ points.
\end{enumerate}
Two interesting examples are the case $X=\bP^1\times\bP^1$ with $D$ a generic $(2n,m)$-curve for $m\ge 4$, or $X=\bF_n$ with $\frac{1}{2}D\sim -K_{\bF_n}$ for $n\le 4$, which would be the quotient of a normal K3 surface with $D$ the ramification locus. We prove the following
\end{comment}

The study of moduli spaces and their compactifications is a central theme in modern algebraic geometry. Compactifying the moduli of varieties fibered in log Calabi–Yau (CY) pairs not only sheds light on the geometry of degenerations, but also connects with a range of moduli theories. Our goal in this paper is to describe the boundary of certain KSBA compactifications of surface pairs fibered in log Calabi–Yau curves, using methods that avoid running explicit steps of the minimal model program (MMP).

\medskip
While our results apply to a broader context, one can first focus on surface pairs
\[
\pi : \bigl(X, \tfrac{1}{n}D\bigr) \longrightarrow C
\]
fibered in log Calabi–Yau pairs of the form
\((\mathbb{P}^1, \tfrac{1}{n}(p_1+\cdots+p_{2n}))\),
under the following mild assumptions:
\begin{enumerate}
    \item the divisor $D$ restricts to $2n$ distinct points on a general fiber of $\pi$;
    \item the pair $\bigl(X, (\tfrac{1}{n}+\epsilon_1)D+\epsilon_2F\bigr)$ is KSBA-stable for every
    $0 < \epsilon_1 \ll \epsilon_2 \ll 1$, where $F$ denotes a fiber of $\pi$ meeting $D$ in fewer than $2n$ points.
\end{enumerate}
Three notable examples falling within this framework are:
\begin{itemize}
    \item divisors of type $(2n,m)$ on $\mathbb{P}^1 \times \mathbb{P}^1$ with $m \geq 4$;
    \item K3 surfaces that admit a $2\!:\!1$ map to a Hirzebruch surface $\bF_n$ for $n \leq 4$, with $D$ the ramification divisor; and
    \item elliptic surfaces with a section or a bisection.
\end{itemize}

\medskip
By the general KSBA moduli theory (cf.~\cite{Kol23}), there exists a \emph{proper} Deligne--Mumford stack $\mtc{M}^{\KSBA}_n(\epsilon_1,\epsilon_2)$ parametrizing KSBA-stable pairs, where a general member is of the form $\bigl(X, (\tfrac{1}{n}+\epsilon_1)D + \epsilon_2 F\bigr),$ as described in the preceding paragraph. Our first main result gives an explicit description of the KSBA-stable pairs parametrized by $\mtc{M}^{\KSBA}_n(\epsilon_1,\epsilon_2)$, with particular emphasis on those lying on the boundary.

\begin{theorem}\label{thm_intro_description_boundary}
    Let $(X_0,(\frac{1}{n} +\epsilon_1)D_0 + \epsilon_2F_0)$ be a KSBA-stable pair parametrized by $\mtc{M}^{\KSBA}_n(\epsilon_1,\epsilon_2)$. Then there exists a nodal curve $C_0$ such that the fibration structure $(X,(\frac{1}{n} +\epsilon_1)D + \epsilon_2F)\to C$ of a general pair in $\mtc{M}^{\KSBA}_n(\epsilon_1,\epsilon_2)$ extends to a fibration $X_0\to C_0$ with pure one-dimensional fibers. Moreover, for any irreducible component $G$ of $C_0$, the reduced structure of $X_0|_G$ is either:
    \begin{enumerate}
        \item \textup{(Proposition \ref{prop: geometry of S1})} the coarse space of a projective bundle $\bP_\cG(\cV)$ over a smooth orbifold curve $\cG$ with coarse space $G$.
        \item \textup{(Proposition \ref{prop: geometry of S2})} the coarse space of a weighted blow-up of a surface $\bP_\cG(\cV)$ as above.
        \item \textup{(Proposition \ref{prop: geometry of S3})} the coarse space of gluing a surface $\bP_\Sigma(\cV)$ ruled over a (possibly disconnected) smooth curve $\Sigma$, glued along an involution of $\Sigma\to\Sigma$; see Figure \ref{fig:Push-out diagram}.
    \end{enumerate}
    When $n=2$, the fibers of $(X_0,D_0)\to C_0$ are classified in Section \S\ref{sec:classification of fibers}.
    \end{theorem}

%The main difference between \Cref{thm_intro_description_boundary} and similar results in the literature is that to find the KSBA-stable limits we don't describe the explicit steps of an MMP. Rather, we use three ingredients:
%\begin{enumerate}
 %   \item a general result on MMP for log Calabi-Yau fibrations in any dimension, 
 %   \item the moduli space of stable quasimaps \cite{twisted_map_2}, and
 %   \item an auxiliary limit $(\sY,\frac{1}{n}\sD_\sY)\to \spec(R)$ of a surface pair as above, which we will refer to as the \textit{ruled model}.
%\end{enumerate}
%We now present these three tools more in detail, and explain how they are used for proving \Cref{thm_intro_description_boundary}.

\medskip
A key difference from previous approaches (e.g. \cite{DH,ascher_invariance_pluri,alexeev2023explicit}) is that our method does not rely on describing the explicit steps of an MMP for the surface pair. Instead, our proofs rest on three main tools:
\begin{enumerate}
    \item a general MMP result for log Calabi–Yau fibrations in arbitrary dimension;
    \item the theory of stable quasimaps to the moduli of points on $\mathbb{P}^1$ developed in \cite{twisted_map_2};
    \item the notion of a \emph{ruled model}, which provides a convenient birational model for controlling singular fibers.
\end{enumerate}
Together, these ingredients allow us to track KSBA-stable limits through quasimaps compactifications and to describe them explicitly via ruled models.

\begin{theorem}[MMP for log Calabi-Yau fibrations]\label{thm_intrp_MMP} Let $R$ be a DVR, and $\pi\colon \left(X,D\right)\to (Y,B_Y+\bfM)\rightarrow \Spec R$ be a morphism of klt generalized pairs. Assume that $$\textstyle \pi^*(K_Y+B_Y+\bfM) \ \sim_\bQ \ K_X +D$$ and that $D$ is $\pi$-ample. Then given an MMP with scaling for the generalized pair $(Y,B_Y+\bfM)$ one can make it follow by a sequence of birational contractions
    for $(X,D)$: 
    $$\begin{tikzcd}
	{\big(X,D\big)} & {\big(X^{(1)},D^{(1)}\big)} & \cdots & {\big(X^{(k)},D^{(k)}\big)} \\
	{\big(Y,B_Y+\bfM)} & {\big(Y^{(1)},B_Y^{(1)}+\bfM^{(1)}\big)} & \cdots & {\big(Y^{(k)},B_Y^{(k)}+\bfM^{(k)}\big)}
	\arrow["{q^{(0)}}", dotted, from=1-1, to=1-2]
	\arrow["{\pi}", from=1-1, to=2-1]
	\arrow["{q^{(1)}}", dotted, from=1-2, to=1-3]
	\arrow["{\pi^{(1)}}", dotted, from=1-2, to=2-2]
	\arrow["{q^{(k-1)}}", dotted, from=1-3, to=1-4]
	\arrow[dotted, from=1-3, to=2-3]
	\arrow["{\pi^{(k)}}", dotted, from=1-4, to=2-4]
	\arrow["{p^{(0)}}", dotted,from=2-1, to=2-2]
	\arrow["{p^{(1)}}", dotted,from=2-2, to=2-3]
	\arrow["{p^{(k-1)}}", dotted,from=2-3, to=2-4]
\end{tikzcd}.$$ Moreover, the following holds:
\begin{enumerate}
    \item the proper transform $D^{(i)}$ of $D$
    remains ample over $Y^{(i)}$;
    \item if $\pi$ has relative dimension one, then $\pi^{(k)}$ has relative dimension one and \[(\pi^{(k)})^*(K_{Y^{(k)}}+B_Y^{(k)}+\bfM^{(k)})\sim_\bQ K_{X^{(k)}} +D^{(k)};\]
    \item if $p^{(i)}$ are morphisms, then $\big(X^{(k)},(1+\epsilon)D^{(k)}\big)$ is a weak canonical model for $\big(X,(1+\epsilon)D\big)$ over $Y^{(k)}$, for $0<\epsilon \ll 1$. 
\end{enumerate}
In particular, if $K_{Y^{(k)}} +  B_Y^{(k)}+\bfM^{(i)}$ is ample and $(X,D)\to \spec R$ is locally stable with stable generic fiber, then the KSBA-stable limit of $\big(X_\eta,(1 +\epsilon)D_\eta\big)$ is independent of $0<\epsilon\ll 1$.
\end{theorem}

\subsection*{Stable quasimaps}

Our second key ingredient is the moduli space of \emph{stable quasimaps} introduced in \cite{twisted_map_2}. 

Recall that a point 
\(
[F] \in \bP\bigl(\oH^0(\bP^1,\cO_{\bP^1}(2n))\bigr)
\)
is GIT-semistable for the $\PGL_2$-action if and only if the corresponding pair 
\(
\big(\bP^1,\tfrac{1}{2n}(p_1+\cdots+p_{2n})\big)
\)
is semi log-canonical, where the $p_i$ are the zeros of $F$. Thus, specifying a fibration 
\(
(X,\tfrac{1}{n}D)\to C
\)
as above is equivalent to giving a morphism 
\[
\phi\colon C \ \longrightarrow \ \sP^{\GIT}_{n-1},
\] 
where $\sP^{\GIT}_{n-1}$ denotes the GIT moduli space of $2n$ points on $\bP^1$. 

An application of the main results of \cite{twisted_map_2} is the construction of a natural compactification of the space of such morphisms. In this compactification, the boundary objects parametrize morphisms 
\[
\cC \ \longrightarrow\  \sP_{n-1},
\]
where $\cC$ is a twisted curve and $\sP_{n-1}$ is a suitable enlargement of $\sP^{\GIT}_{n-1}$. 

\medskip
This has an important consequence for us. Given a family 
\[
\pi_\eta\colon (X_{\eta},\tfrac{1}{n}D_{\eta}) \ \longrightarrow \  C_{\eta}
\]
over the generic point $\eta$ of the spectrum of a DVR $R$, one can (after possibly replacing $\Spec R$ by a ramified cover) extend it to a diagram
\[
\xymatrix{
(X_{\eta},\tfrac{1}{n}D_{\eta}) \ar[r] \ar[d]_{\pi_\eta} & (\cX^{\tw},\tfrac{1}{n}\cD^{\tw}) \ar[d]_\pi \\
C_{\eta} \ar[r] \ar[d] & \cC^{\tw} \ar[d] \\
\{\eta\} \ar[r] & \Spec R
}
\]
satisfying:
\begin{enumerate}
    \item $\cC^{\tw}$ is a family of nodal curves,
    \item $\pi$ has relative dimension one,
    \item $K_{\cX^{\tw}} + \tfrac{1}{n}\cD^{\tw}$ is $\pi$-trivial and $\cD^{\tw}$ is $\pi$-ample, and
    \item the fibers of $\pi$ over smooth points of $\cC^{\tw}\to \Spec R$ are slc.
\end{enumerate}

\noindent In other words, if $\bfM$ denotes the moduli part of $\pi$, then the family
\[
\big(\cX^{\tw},\tfrac{1}{n}\cD^{\tw}\big) \ \longrightarrow \ (\cC^{\tw},\bfM^{\tw})
\]
is of the type covered by \Cref{thm_intrp_MMP}. Hence, to run an MMP for $(\cX^{\tw},\tfrac{1}{n}\cD^{\tw})$ it suffices to run an MMP for $(\cC^{\tw},\bfM^{\tw})$, which is much simpler since the latter is a surface. By \Cref{thm_intrp_MMP}, this produces a sequence of birational transformations of $(\cX^{\tw},\tfrac{1}{n}\cD^{\tw})$ ending in a weak canonical model.

\medskip
To control explicitly how this MMP modifies $(\cX^{\tw},\tfrac{1}{n}\cD^{\tw})$, we introduce the notion of a \emph{ruled model} (\Cref{defn:ruled model}). Roughly speaking, a locally stable family $(\cX,\tfrac{1}{n}\cD)\to B$ with $\cX$ normal admits a ruled model if it is crepant birational to a locally stable family
\[
(\cY,\tfrac{1}{n}\cD_\cY) \ \stackrel{\pi_\cY}{\longrightarrow} \ \cC \ \longrightarrow \ B
\]
such that the fibers of $\pi_\cY$ over smooth points of $\cC\to \Spec R$ are slc pairs 
\(
\big(\bP^1,\tfrac{1}{n}(p_1+\cdots+p_{2n})\big),
\)
and $\cC\to B$ is a family of nodal curves.

\medskip
As an illustration, consider $B=\Spec k$ and $Z=\bP^1\times \bP^1$. Let $X$ be the $(1,2k)$-weighted blow-up of $Z$ at a point, so that the exceptional divisor appears with multiplicity one along a fiber of the projection $\pi_1\colon \bP^1\times\bP^1\to \bP^1$. Then $X$ has an $A_{2k-1}$-singularity at a point $p$. If $D\subseteq X$ is a divisor such that $(X,\tfrac{1}{2}D)\to \bP^1$ is fibered in slc log Calabi–Yau pairs, a ruled model can be obtained by first extracting a divisor $Z\to X$ over $p$ so that $(Z,\tfrac{1}{2}D_Z)\to \bP^1$ still fibers in slc log Calabi–Yau pairs and $Z$ has two $A_{k-1}$ singularities; then contracting the proper transforms of the two irreducible components of the singular fiber of $X\to \bP^1$. The last three pictures in \Cref{fig:nodes of X} illustrate these birational transformations.
 \begin{figure}
    \centering
    \includegraphics[width=1.0\linewidth]{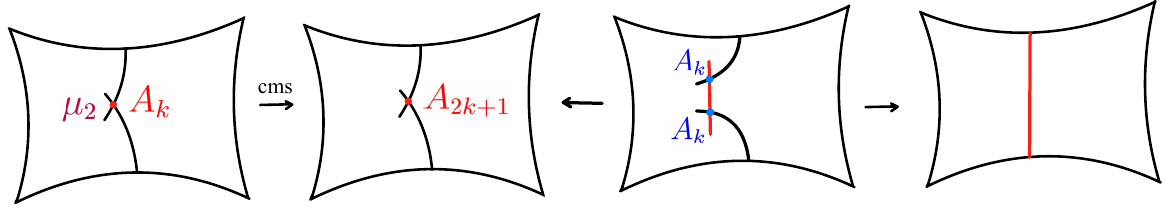}
    \caption{}
    \label{fig:nodes of X}
\end{figure}
\medskip
Our third main result is the following.

\begin{theorem}\label{thm_intro_ruled_model_exists}
    Every pair $(X^{(i)},\tfrac{1}{n}D^{(i)})$ constructed from 
    \[
    (\cX^{\tw},\tfrac{1}{n}\cD^{\tw}) \ \longrightarrow \ (\cC^{\tw},\bfM^{\tw})
    \]
    using \Cref{thm_intrp_MMP} admits a ruled model. 
\end{theorem}

The advantages of \Cref{thm_intro_ruled_model_exists} are threefold:
\begin{enumerate}
    \item ruled surfaces are more tractable geometrically,
    \item being crepant birational ensures control of the singularities of the fibers, and
    \item we can describe explicitly the birational transformations relating $(X^{(i)},\tfrac{1}{n}D^{(i)})$ to its ruled model. 
\end{enumerate}

In a forthcoming work, for surface pairs of the form $(X,\tfrac{1}{2}D)\to C$ with $D\to C$ admitting a section, we will use our results to recover and reprove \cite[Theorem~1.2]{inchiostro_elliptic}, \cite[Theorem~1.(a),(b)]{AB_del_pezzo}, and parts of \cite{AB_k3}.

\subsection*{Interaction with other moduli stacks}

In the last section, we use the explicit MMP of \Cref{thm_intrp_MMP} to construct a morphism from an appropriate modification of the quasimaps moduli space of $(1,4)$-curves in $\bP^1\times \bP^1$ to the GIT moduli space of $(1,4)$-curves in $\bP^1\times \bP^1$, and we compare these moduli spaces to the K-stability moduli space for $(1,4)$-curves.

\begin{theorem}
 Let $\cM^K_{(1,4)}(c)$, $\cM^{\GIT}_{(1,4)}$ and ${\cQ}_{(1,4)}$ be the K-moduli stack, GIT quotient stack, and the moduli stack of stable quasimaps of $(1,4)$-curves on $\bP^1\times\bP^1$, respectively.
 \begin{enumerate}
     \item There exists a natural isomorphism $\cM^K_{(1,4)}(c)\simeq \cM^{\GIT}_{(1,4)}$ for any $0<c<\frac{1}{2}$. In particular, there are no K-moduli wall-crossing $\cM^K_{(1,4)}(c)$.
     \item After possibly a normalization and a root stack for ${\cQ}_{(1,4)}$, there exists a natural stable reduction morphism $$\phi:\cQ_{(1,4)}\ \longrightarrow \ \cM^{\GIT}_{(1,4)},$$ which descends to a birational morphism $$\psi\colon \ove{Q}_{(1,4)} \ \longrightarrow \ \ove{M}^{\GIT}_{(1,4)}$$ between their good moduli spaces.
     \item The exceptional loci of $\psi$ (resp. $\phi$) are classified in Section \S \ref{sec:GIT and quasimap classification}.
     \item The explicit stable reduction of $\phi^{-1}\colon \cM^{\GIT}_{(1,4)}\dashrightarrow {\cQ}_{(1,4)}$ along a one-parameter family is displayed in Section \S \ref{sec:twisted stable reduction}.
 \end{enumerate}
\end{theorem}

\subsection*{Explicit stable reduction and canonical bundle formula}
In this subsection we interpret how we use \Cref{thm_intrp_MMP}, in terms of the canonical bundle formula. 

Recall that, given two locally stable families $(\cX,\cD)\to \spec R$ and $\cY\to \spec R$ over the spectrum of a DVR, with a morphism $\pi\colon (\cX,\cD)\to (\cY,\bfM)$ of generalized pairs such that $K_\cX+\cD\sim_\bQ\pi^*(K_\cY+\bfM)$, it is not true that the canonical bundle formula commutes with base change. In other terms, it is not true that the restriction of $\bfM$ to the special fiber of $\cY\to\spec R$, is the moduli part of the morphism $\pi$ restricted to the special fiber of $\cX_0$.
For example, one can take a generic $(4,3)$-curve $D$ in $\bP^1\times\bP^1$ degenerating to $V(x_0x_1(x_0-x_1)(x_0-2x_1)y_0y_1(y_0-y_1))$ along $\spec R$. The projection $(\bP^1_R\times \bP^1_R,\frac{1}{2}D)\to \bP^1_R$ has no boundary part generically, but on the special fiber there are three points which contribute to the boundary part with coefficient $\frac{1}{2}$; and dually the moduli part of the special fiber is trivial while it is not on the generic fiber.
In particular, if one has a fibration in log Calabi-Yau pairs $(X,D)\to Y\to \eta$ over the generic point $\eta $ of the spectrum of a DVR, and one completes this fibration to a fibration in log Calabi--Yau pairs $(\cX,\cD)\to \cY\to \spec R$, there is no reason to believe that the limit of the moduli part of the generic fiber will be the moduli part of the special fiber. 

It turns out that this phenomenon does not happen if one takes a limit in the quasimaps moduli space of \cite{twisted_map_2} instead. Indeed, the moduli part comes from a map from $Y$ to a compact moduli space of the fibers of $\pi$, and the quasimaps moduli spaces forces (essentially by definition) this map to extend.

Now, to take the stable limits as in \Cref{thm_intro_description_boundary}, our first step is to consider a limit coming from stable quasimaps $(\cX^{\tw},\frac{1}{n}\cD^{\tw})\to (\cC,\bfM)\to \spec R$. For this limit, the limit of the moduli part is the moduli part of the limit. On the other hand,
the $\bQ$-Cartier divisor\footnote{If the limit comes from quasimaps, then $\bfM$ is an actual $\bQ$-Cartier divisor rather than a \textbf{b}-divisor} $K_{\cC}+\bfM$ might not be ample over $\spec R$. The contractions which will make $K_{\cC}+\bfM$ ample are contractions which,
on the special fiber, \textit{replace a rational tail with boundary part, of the same degree as the degree of the moduli part on the contracted tail.} In other terms, as the degree of $K_{\cC^{(i)}}+\bfM^{(i)}$ remains constant, whenever a contraction of \Cref{thm_intrp_MMP}
contracts a tail $T$ of the special fiber, the point of the special fiber to which $T$ is contracted will be in the support of the boundary part for the special fiber, with coefficient $\deg(\bfM|_T)$; see \Cref{fig:example of MMP}.

\begin{figure}
\centering\includegraphics[width=0.5\linewidth]{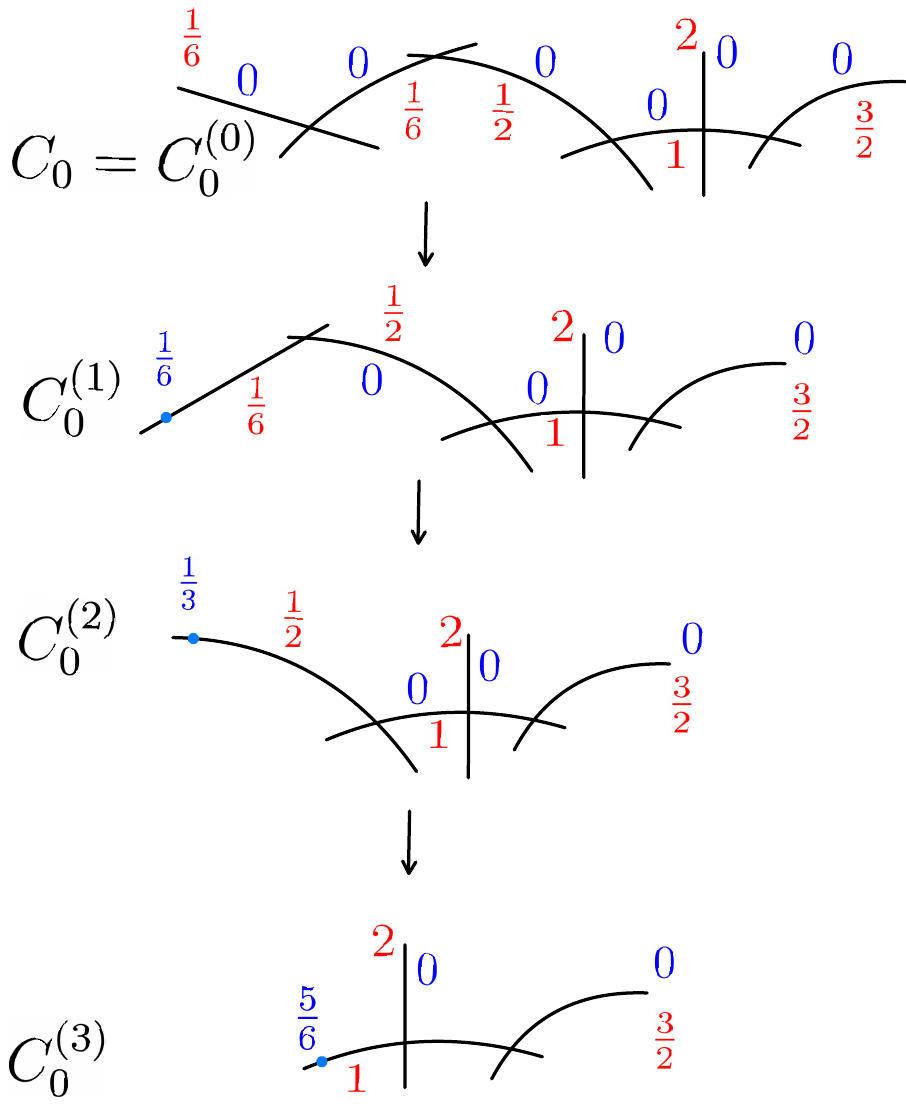}
    \caption{An example of MMP for $(C,B+\bfM)\rightarrow \Spec R$\\
    red = degree of moduli part; blue = degree of boundary part}
    \label{fig:example of MMP}
\end{figure}
Then we use the ruled model to control the fibers $X^{(i)}\to C^{(i)}$, and Section \S\ref{sec:classification of fibers} to classify the divisors on those fibers.

\subsection*{Prior and related works}
The previous works which are most relevant to the current paper are \cite{la2002explicit, ab_twisted, ascher_invariance_pluri,inchiostro_elliptic,AB_del_pezzo,AB_k3,BFHIMZ24}, where the authors use twisted stable maps of Abramovich and Vistoli to study the KSBA-compactification of the moduli space of elliptic surfaces with a section, with a with a special emphasis on elliptic
K3 surfaces in \cite{AB_k3} and rational elliptic surfaces in \cite{AB_del_pezzo}. One of the main challenges of the papers mentioned above is the explicit stable reduction algorithm to classify the objects on the boundary.
This relies on an explicit flip, called ``flip of La Nave''. Proving that, when running a specific MMP, the only type of flip occurring is the one of La Nave is the main tool used by the authors to classify the objects in the boundary, but it is arguably one of the most technical aspects in the papers above mentioned.

Now, if one takes the quotient by the hyperelliptic involution of an elliptic surface $(Y,\Sigma_Y)\to C$ with a section $\Sigma_Y$, one obtains a pair as the ones at the beginning of the introduction, of the form $\pi\colon (X,\frac{1}{2}(\Sigma+D'))\to C$ with $\Sigma$ is a section of $\pi$n and $\Sigma+D'$ the ramification locus.
Then the methods of this paper when $n=1$ and when $D = \Sigma + D'$ circumvent the need of running a specific MMP, hence removing many of the technicalities needed in the papers mentioned above. The applications of our methods to certain moduli of elliptic surfaces with either a section or a 2:1 section will be presented in a forthcoming work.

There is some other existing literature on KSBA-moduli spaces of pairs of the form $(Y,(c+\epsilon)D)$ where $K_Y+D\sim_\bQ0$, $D$ and $0\le c <1$; see \cite{ABE, aet, AE23,hacking,DH,devleming_P3}.

\subsection*{Outline of the paper}

In Section~\S\ref{section_prelim_on_qmaps} we review the basic notions used throughout the paper, with an emphasis on the moduli of stable quasimaps, and give a detailed study of stable quasimaps to log Calabi--Yau curves.  
Section~\S\ref{section_stable_reduction} develops general results on the MMP for varieties fibered in log Calabi--Yau pairs and applies them to fibrations arising from stable quasimaps.  
In Section~\S\ref{section_application_to_KSBA_moduli_when_epsilon_2_goes_to_0} we apply the results of Section~\S\ref{section_stable_reduction} to describe the boundary objects of the KSBA compactification for certain moduli spaces of surface pairs.  
Finally, Section~\S\ref{section_example_14_case} presents a concrete example of \((1,4)\)-curves on \(\mathbb{P}^1\times\mathbb{P}^1\) and compares the corresponding moduli of stable quasimaps with the \(K\)-moduli and GIT moduli stacks.

\subsection*{Table of notations}

For the reader’s convenience, we collect here the notations most frequently used throughout the paper.

\renewcommand{\arraystretch}{1.5}

\begin{center}
\begin{longtable}{| p{.18\textwidth} | p{.75\textwidth} |}
    \hline \textbf{Notation} & \textbf{Definition/Description}   \\ \hline
    \hline $\sP_n^{\GIT}$ & GIT semistable pairs $(\bP^1,p_1+\cdots+p_{2n+2})$    \\ \hline
      $\sP_n$   & enlargement of $\sP_n^{\GIT}$   \\ \hline
       $\sP_{n,\dm}$   & slc pairs $(C,(\frac{1}{n+1}+\epsilon)D)$   \\ \hline
        $\sP_{n}^{\CY}$   & moduli of boundary polarized log CY pairs $(\bP^1,\frac{1}{n+1}D_{2n+2})$     \\ \hline
        $P_{n}$   & good moduli space of $\sP_{n}$ (and ${\sP}_{n}^{\GIT}$)     \\ \hline
         $\bfM_n'$   & the Hodge line bundle on $\sP_{n}^{\CY}$     \\ \hline
         $\bfM_n$   & the pull-back of $\bfM_n'$ along $\sP_n\rightarrow \sP_{n}^{\CY}$    \\ \hline
          $(\bF_m;\mtf{e},\mtf{f})$   & the Hirzebruch surface $\bF_m$ with a section $\mtf{e}$ satisfying $(\mtf{e}^2)=-m$  and the fiber class $\mtf{f}$     \\ \hline
           $\cQ_{g,m;n,\beta}$   & closure of the locus parametrizing maps from smooth curves, in the stack of stable quasimaps from genus $g$ twisted $m$-pointed curves $\sC$ to $\sP_n$ of class $\beta$. When there is no ambiguity, we will remove the subscripts $g,m;n,\beta$    \\ \hline
            % $\ove{\cQ}$   & closure of the locus in $\cQ_{0,1,\beta}$ parametrizing quasimaps $\bP^1 \rightarrow \sP_1$ corresponding to a $(1,4)$-divisor in $\bP^1 \times \bP^1$\textcolor{red}{old, leaving it for now}    \\ \hline
               $\mtc{M}^{\GIT}_{(d_1,d_2)}$    & quotient stack of the space $|\mtc{O}_{\bP^1\times\bP^1}(d_1,d_2)|^{\sst}$ of GIT semistable $(d_1,d_2)$-curves by $\PGL(2)\times\PGL(2)$    \\ \hline
          $\ove{M}^{\GIT}_{(d_1,d_2)}$    & GIT quotient of the space $|\mtc{O}_{\bP^1\times\bP^1}(d_1,d_2)|$ of $(d_1,d_2)$-curves by $\PGL(2)\times\PGL(2)$   \\ \hline
          $\mtc{M}^{K}_{(d_1,d_2)}(c)$    & irreducible component of the K-moduli stack generically parametrizing $(\bP^1\times\bP^1,cC)$ for $(d_1,d_2)$-curves $C$    \\ \hline
          $\ove{M}^{K}_{(d_1,d_2)}(c)$    & good moduli space of $\mtc{M}^{K}_{(d_1,d_2)}(c)$    \\ \hline
          $\mtc{M}^{\KSBA}_n(\epsilon_1,\epsilon_2)$    & irreducible component of the KSBA-moduli stack generically parametrizing $\big(X,(\frac{1}{n}+\epsilon_1)D+\epsilon_2F\big)$, where $F$ are fibers of the log CY  fibration $(X,\frac{1}{n}D)\rightarrow C$    \\ \hline
          $\ove{M}^{\KSBA}(\epsilon_1,\epsilon_2)$    & coarse moduli space of $\mtc{M}^{\KSBA}_{(d_1,d_2)}(\epsilon_1,\epsilon_2)$    \\ \hline
        
\end{longtable}
    
\end{center}

\subsection*{Conventions} We adopt the following conventions throughout this paper.
\begin{enumerate}
    \item Every algebraic stack, unless otherwise stated, will be of finite type over $\bC$.
    \item A \emph{point} of a variety/scheme/stack is a $\bC$-point if there is no extra assumption. 
    \item In this paper, we widely use the language of log pairs from birational geometry. For the definition of singularities of log pairs, see \cite{KM98,Kol13}.
    \item We call a pair $(X,D)$ \emph{log Calabi-Yau} if $K_X+D\sim_{\bQ}0$. A pair $(X,D)$ is (\emph{semi-})\emph{log canonical Calabi-Yau} (in short, slc Calabi-Yau) if it is log Calabi-Yau and has (semi-)log canonical singularities.
    \item The \emph{coarse space} of a Deligne--Mumford (abbv. DM) stack $\mathcal{X}$ is the coarse moduli space $\mathcal{X}\rightarrow X$ (cf. \cite[Tag 04UX]{Sta18}). We remove the word \emph{moduli} as many stacks have no modular meaning.
    \item A \emph{stacky curve} is a proper, pure one-dimensional DM stack of finite type over $\bC$ which is generically a scheme.
    \item A \emph{twisted curve} is a special stacky curve; see \cite[Definition 4.1.2]{AV02} and \cite[Definition 1.2]{olsson2007log}. When, with the notations of \textit{loc. cit.}, we do not specify the gerbes $\Sigma_i$, we mean that $i=0$ (i.e. there are no gerbes). 
    \item A \textit{pointed twisted curve} is a twisted curve $\sC$ together with $m$ smooth distinct points $p_1,\ldots,p_m$.
    \item A \emph{fibration} $X\rightarrow Y$ is a surjective morphism between normal projective varieties with connected fibers, and $\dim X>\dim Y$.
    \item Throughout this paper, for any DVR $R$, we denote by $\eta$ (resp. $0$) the generic (resp. closed) point of $\Spec R$. If $f:X\rightarrow \Spec R$ is an object over $\Spec R$ (e.g. a family of pairs), then we denote by $X_\eta$ (resp. $X_{0}$) the generic (resp. closed) fiber of $f$.%\Gio{we should check we are consistent}
\end{enumerate}

\subsection*{Acknowledgments}
We thank Dori Bejleri, Stefano Filipazzi, Yuchen Liu and Luca Schaffler for helpful conversations. GI was supported by AMS-Simons travel grant. RS was supported by the ''Programma per giovani ricercatori Rita Levi Montalcini” of
MUR and by PSR 2022 – Linea 4 of the University of Milan. RS is a member of the GNSAGA
group of INDAM. JZ was supported by AMS-Simons travel grant.

\section{Stable quasimaps to moduli of pointed rational curves}\label{section_prelim_on_qmaps}

\subsection{Fibrations in log Calabi-Yau pairs}
We begin by recalling two moduli spaces which will be useful later.  
Let $n\geq 1$ be an integer. We denote the GIT moduli space of $2n+2$ unordered points on $\bP^1$ by $$\sP_n^{\GIT} :=\ \big[\bP(\oH^0(\cO_{\bP^1}(2n + 2)))^{ss}/\PGL_2\big].$$ Recall that every point of $\sP_n^{\GIT}$ represents a degree $2n+2$ divisor $D$ on $\bP^1$ such that $\mult_p(D)\leq n+1$ for any $p\in \bP^1$. Equivalently, this moduli space parametrizes log canonical Calabi-Yau pairs of the form $(\bP^1,\frac{1}{n + 1}D)$.

\begin{notation}
    \textup{We will denote the enlargement constructed in \cite[Appendix A.2.1]{twisted_map_2} by \[ \iota:\ \sP_n^{\GIT}\ \subseteq \ \sP_n\]
    and by $P_n$ the good moduli space of $\sP_n$ (or $\sP_n^{\GIT}$, see point (\ref{same_gms_for_wideP_and_P}) below).}
\end{notation}
Recall that the points of $\sP_n$ parametrize pairs $(\cC,\frac{1}{n+1}D)$ such that:

\begin{itemize}
\item $\cC$ is a twisted curve with coarse moduli space $C$ and $D\subseteq \cC$ are $2n+2$ smooth (schematic points of $\cC$);
\item the pair $(C,\frac{1}{n+1}D)$ is an slc log Calabi-Yau pair with $D$ ample;
    \item either $\cC\cong \bP^1$ or $\cC$ is a twisted curve, which is the nodal union of two $\bmu_2$-root stacks of $\bP^1$ at $\infty$, glued along the $\cB\bmu_2$ closed substacks, and each branch has $n$ points.
\end{itemize}

    \begin{figure}
    \centering
    \includegraphics[width=0.6\linewidth]{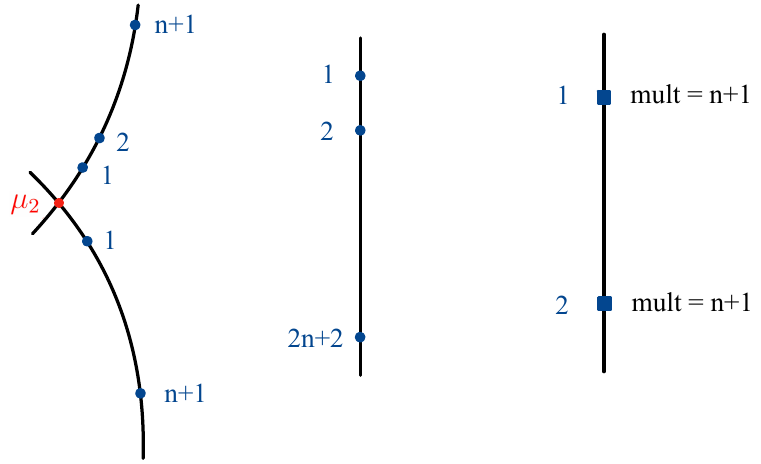}
    \caption{Three pairs parametrized by $\sP_n$}
    \label{fig:objects in stacks}
\end{figure}

\begin{figure}
    \centering
\includegraphics[width=0.6\linewidth]{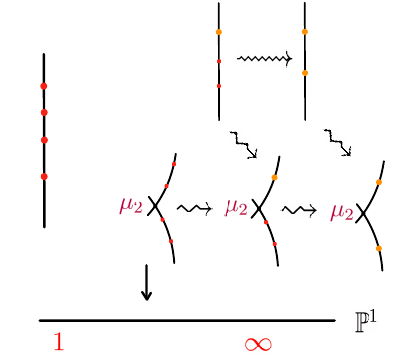}
    \caption{Points of the stack $\sP_{1}$ over its good moduli space $\bP^1$\\ red (resp. orange) point = multiplicity 1 (resp. 2); \\ squiggle arrow = $\bG_m$-equivariant degeneration }
    \label{fig:points sP1}
\end{figure}
We have the following relations between stacks.

\begin{enumerate}
\item (ref. \cite[Proposition A.7]{twisted_map_2}) the stack $\sP_n$ is a $\bmu_2$-root stack over the stack $\sP_{n}^{\CY}$ (ref. \cite[Definition 16.7]{ABBDILW23}), which parametrizes certain boundary polarized Calabi-Yau (abbreviated by bpCY) pairs. The root stack is along the divisor in $\sP_{n}^{\CY}$ parametrizing nodal curves. The morphism $\sP_n\to \sP_{n}^{\CY}$ sends a pair $(\cC,\frac{1}{n+1}D)$ consisting of a twisted curve and the divisor $D$ to the pair $(C,\frac{1}{n+1}D)$ where $\cC$ is replaced with its coarse moduli space. 
\item[(1$^{+}$)] We briefly describe the stack $\sP_{n}^{\CY}$ for the reader's convenience: it parametrizes slc pairs $(X,D)$ which admits a smoothing to $(\bP^1,\frac{1}{2n+2}(p_1+\cdots+p_{2n+2}))$, where $p_1,...,p_{2n+2}$ are distinct points, such that $-K_X$ is ample and $$(2n+2)\cdot (-K_X-D)\ \sim\  0 .$$
    \item\label{item_where_we_say_there_is_a_map_spn_to_tilde_pn} (ref. \cite[Theorem A.4(1)]{twisted_map_2}) The inclusion $\iota:\sP_n^{\GIT}\subseteq \sP_n$ admits a retraction, i.e. a morphism $$\pi:\ \sP_n \ \longrightarrow \ \sP_n^{\GIT}$$ such that $\pi\circ\iota=\Id$, which sends each point corresponding to a nodal curve to the point corresponding to the strictly polystable pair $$\big(\bP^1,(n+1)[0]+(n+1)[\infty]\big).$$
\item\label{same_gms_for_wideP_and_P} (ref. \cite[Corollary A.8]{twisted_map_2}) The stack $\sP_n$ has the same good moduli space as $\sP_n^{\GIT}$.
    \item The locus $\sP_{n,\dm}\subseteq \sP_n$ parametrizing pairs $(C,\frac{1}{n+1}D)$ such that there is an $\epsilon>0$ such that $\big(C,(\frac{1}{n+1}+\epsilon)D\big)$ is slc is an open and proper Deligne-Mumford substack.
\end{enumerate}
We summarize these relations in the following diagram
\[\xymatrix{\sP_{n,\dm} \ar@{^{(}->}[rr]^{\text{open}} & &\sP_n\ar[d]^{\text{root stack}}  \ar@/^1.0pc/@[black][dll]^-\pi\\
\ \sP_n^{\GIT} = \big[\bP(\oH^0(\cO_{\bP^1}(2n + 2)))^{ss}/\PGL_2\big] \ar@{^{(}->}[rr] \ar@/^0.5pc/@{^{(}->}[rru]^{\iota}\ar[d]_{\text{good moduli space}} & & \sP_{n}^{\CY} \ar[dll]^-{\text{good moduli space}} \\   \bP(\oH^0(\cO_{\bP^1}(2n + 2)))^{ss} \sslash\PGL_2 & &  }
\]

\subsubsection{Construction of the retraction $\pi\colon \sP_n\to \sP_n^{\GIT}$}\label{subsubsection_recalling_how_to_take_ruled_model_in_qmap_paper}
Let us provide more details on the item (2) above, more explicitly on how the retraction $\pi$ is constructed. For simplicity, we will first stick with one-parameter families which intersect $\sP_n^{\GIT}$. After, we recall a more global description that will be needed in this paper, and we refer the reader to \cite[Theorem A.4(1)]{twisted_map_2} for further details.

Let $C\to \sP_n$ be a morphism from a curve which sends the generic point of $C$ to $\sP_n^{\GIT}$. This corresponds to a surface pair $(\sX,\frac{1}{n+1}\sD)\to C$ whose nodal fibers have a twisted node, namely a balanced node of a twisted curve (see \cite[Definition 1.2]{olsson2007log}) with a $\bmu_2$-stabilizer, and whose generic fiber is smooth. Take the coarse space \[\begin{tikzcd}
	{(\sX,\frac{1}{n+1}\sD)} && {(X,\frac{1}{n+1}D)} \\
	& C.
	\arrow["{\textup{cms}}", from=1-1, to=1-3]
	\arrow[from=1-1, to=2-2]
	\arrow[from=1-3, to=2-2]
\end{tikzcd}\] As the nodes of $\sX$ have a $\bmu_2$-stabilizer, the corresponding points of $X$ will have $A_{2k+1}$-singularities. In particular, one can extract the divisors $Z\to X$ over each such singularity so that the resulting family is obtained by replacing each nodal fiber with a chain of 3 $\bP^1$s, with two $A_k$-singularities, and with the divisor not intersecting the exceptional divisor; see Figure \ref{fig:nodes of X}.
 
We can then contract $Z\to Y$ the two irreducible components of the fiber intersecting the proper transform of $D$, as in the picture above. The resulting family will be obtained by replacing each nodal fiber with a fiber isomorphic to $(\bP^1,(n+1)[0]+(n+1)[\infty])$. 

More generally, let $\sC\rightarrow \Spec R$ be a family of twisted curves with normal generic fiber, where $R$ is a DVR. Let $(\sX,\frac{1}{n+1}\sD)\to \sC$ be the family pulled back from $\sP_n^{\CY}$ along the composition $\sC\to \sP_n\to \sP_n^{\CY}$, where the first map is a stable quasimap sending the generic fiber to $\sP_n^{\GIT}$. Then for every nodal point $q$ of the fibers of $\psi\colon \sX\to \sC$, there is a smooth neighborhood of the pointed map \[(\sX,q) \ \longrightarrow \  (\sC,\psi(q)),\] which is isomorphic to a smooth neighborhood of the pointed map \[\big(\spec(A_\mtf{m}[x,y]/xy-f^2), \bV(\mtf{m},x,y)\big) \ \longrightarrow \ \big(\spec(A_{\mtf{m}}),\bV(\mtf{m})\big).\] Here, $\mtf{m}$ is the maximal ideal of the local ring $A_{\mtf{m}}:=\cO_{\sC,\psi(q) }$, and $f\in \mtf{m}$ vanishes precisely along the locus over which the family is singular. Then the birational transformation $\sX\dashrightarrow \sY$ corresponding to $\sC\to \sP_n\to \sP_n^{\GIT}$ is constructed by first performing a blow-up $\sZ\to \sX$, which in the chart above is the blow-up of $\bV(x,y,f)$, followed by a contraction $\sZ\to \sY$ which contracts the fiber components of reducible fibers of $\sZ\to \sC$, which are not the exceptional divisor for this blow-up.

\subsubsection{Moduli of stable quasimaps}
We begin by recalling the following definition from \cite{twisted_map_2}.
\begin{defn}\label{def_stable_qmap}
   Let $(\mts{C};p_1,\ldots,p_m)$ be a pointed twisted curve of genus $g$, and $\beta$ be the class of a curve in $\sP_n$. A representable morphism $\phi:\mts{C}\rightarrow \sP_n$, with the restricted universal family $\pi:(\sY,\frac{1}{n+1}\sD)\to \sC$, is called a \emph{stable quasimap} of class $\beta$ if
    \begin{itemize}
            \item[(Sing.)] the subset $\Delta$ of $\mathscr{C}$ consisting of points $p$ such that $(\sY_p,(\frac{1}{n+1}+\epsilon)\sD_p)$ does {\bf{not}} have semi-log-canonical singularities for any $0<\epsilon \ll 1$ is a finite (possibly empty) union of smooth points on $\sC\smallsetminus\{p_1,\ldots,p_m\}$,
            \item[(Stab.)] if $\cR\subseteq \sC$ is an irreducible component such that $\deg(\omega_\sC(p_1+\ldots+p_m) |_\cR)< 0$, then not all the fibers of $\pi|_\cR\colon(\sY|_\cR,\frac{1}{n+1}\sD|_\cR)\to \cR$ are $S$-equivalent; whereas if $\deg(\omega_\sC(p_1+\ldots+p_m) |_\cR)=0$, then not all fibers of $\pi|_\cR$ are isomorphic, and
            \item[(Num.)] the family $\pi$ is the pull-back of the universal family along a map $\sC\to \sP_n$ of class $\beta$ from a twisted curve of genus $g$.
        \end{itemize}
\end{defn}

The labels (Sing.), (Stab.) and (Num.) stand for singularity condition, stability condition, and numerical condition. The following follows from \cite[Theorem 1]{twisted_map_2}; see also \cite[Theorem 2]{twisted_map_2}.

\begin{theorem}There exists a proper Deligne-Mumford stack, denoted by $\cQ_{g,m;n,\beta}$, whose closed points parametrize stable quasimaps from a genus $g$ twisted $m$-pointed curve $\sC$ to $\sP_n$ of class $\beta$.
\end{theorem}
\begin{remark}\label{remark_S-equivalence}\textup{\begin{enumerate}
    \item When the subscripts $g,m;n,\beta$ are either clear or not relevant, we will suppress them and simply denote $\cQ_{g,m;n,\beta}$ by $\cQ$.
    \item Being S-equivalent is an equivalence relation among isomorphism class of pairs in $\sP_n$ \cite[\S 6.5]{ABBDILW23}, and the points of the good moduli space $P_n$ of $\sP_n$ are in bijection (via the morphism $\sP_n\to P_n$) to S-equivalence classes. In particular, a morphism $C\to \sP_n$ from a curve satisfies that the composition $C\to \sP_n\to P_n$ is finite if and only if there are two points $x_1,x_2\in C$ which map to pairs in $\sP_n$ which are not S-equivalent.
\end{enumerate}}
\end{remark}

\begin{defn}
    A stable $m$-pointed quasimap $\phi\colon (\sC,p_1,\ldots,p_m)\to \sP_n$ is called \emph{smoothable} if there is a DVR $R$ with a morphism $\spec R\to \cQ$ which sends the closed point $\xi$ to $[\phi]$, and the generic point $\eta$ to a point $[f\colon C\to \sP_n^{\GIT}\subseteq \sP_n]$, where $C$ is a smooth genus $g$ curve.\end{defn}

\begin{lemma}\label{lemma_non_ksba_locus_is_cartier}
    For any $n\ge 1$, there exists a Cartier divisor $\Delta_{\sP_n}\subseteq \sP_n$ supported at the set of points parametrizing pairs $(\cR;x_1,\ldots,x_{2n+2})$ where at least two of the $2n+2$ marked points $x_1,...,x_{2n+2}$ collide.
\end{lemma}

\begin{proof}
    %For any family of pointed stable quasimaps $(\sC,p_1,\ldots,p_m)\to \sP_n\times B$ over a base $B$, and consider the open $[\bP(\oH^0(\bP^1,\cO_{\bP^1}(2n+2)))^{ss}/\PGL_2]\subseteq \sP_{n}^{\CY}$, which we denote by $\sU$.
    
    Observe that the locus in $\bP(\oH^0(\bP^1,\cO_{\bP^1}(2n+2)))$ where two of the $2n+2$ points agree is a Cartier divisor; this is known as the discriminant divisor. This Cartier divisor is $\PGL_2$-equivariant, and thus it descends to a Cartier divisor on $\wt{\sP}_n$, denoted by $\Delta_{\wt{\sP}_n}$.
    
   As $\sP_{n}^{\CY}$ is smooth by \cite[Lemma 16.8 (b)]{ABBDILW23}, then the closure of $\Delta_{\wt{\sP}_n}$ in $\sP_{n}^{\CY}$, whose closed points correspond to pairs $(C,\frac{1}{n+1}D)$ where the support of $D$ consists of at most $2n+1$ points, is also a Cartier divisor $\Delta_{\sP_{n}^{\CY}}$. Now it suffices to take $\Delta_{\sP_n}$ to be the pull-back of $\Delta_{\sP_{n}^{\CY}}$ along the root stack $\sP_n\rightarrow \sP_{n}^{\CY}$
\end{proof}

\begin{remark}\label{rem:existence of discriminant divisor}\textup{From \Cref{lemma_non_ksba_locus_is_cartier}, for every $g,m$ and $\beta$, there is a Cartier divisor on the universal curve $\sC_{\univ}\to \cQ_{g,m;n,\beta}$ supported at the points $p\in \sC$ which map to $\Delta_{\sP_n}$. We will denote it by $\Delta_{g,m;n,\beta}$ if we want to emphasize the numerical invariant $g,m,n$ and $\beta$; otherwise we will simply denote it by $\Delta$.}
\end{remark}

\subsubsection{Hodge line bundles and moduli divisors}\label{subsubsection_moduli_part}
We summarize several results concerning the canonical bundle formula, which plays a central role in the study of lc-trivial fibrations.

Let $(\mts{X},\mts{D})\rightarrow T$ be a family of lc-trivial fibrations; see \cite[Definition 14.2]{ABBDILW23}. Then, up to replacing $T$ with a finite morphism $T'\to T$ and $X$ by the normalization of the main component of $X\times_T T'$, by \cite{Amb05} there is a $\bQ$-linear equivalence
$$K_X + D \sim_{\bQ} f^{*}(K_T + D_T +\bfM_T),$$
where $D_T$ and $\bfM_T$ are $\bQ$-divisors on $T$ defined as follows:
\begin{enumerate}
    \item $D_T$ is the \emph{discriminant divisor}, which is defined by $D_T := \sum_P (1-b_P )P$, where the sum runs through prime divisors $P \subseteq T$ and
$$b_P := \max\big\{\lambda \in \bQ \ |\  (X,D + \lambda f^*P)\textup{ is sub-lc over the generic point of }P\big\}.$$
    \item $\bfM_T := G -K_T - D_T$ is the \emph{moduli divisor} (class!), where $G$ is a $\bQ$-Cartier $\bQ$-divisor on $T$ such that $K_X + D\sim_{\bQ} f^{*}G$.
\end{enumerate}

\begin{defn}
    Let $f \colon (X,D) \rightarrow T$ be a family of boundary polarized CY pairs parametrized by $\sP_{n}^{\CY}$. We define the \emph{Hodge line bundle} associated to $f$ as
$$\lambda_{\Hodge,f} := f_*\omega^{[n+1]}
_{X/T}\big((n+1)D\big).$$
\end{defn}
By definition, this gives rise to the \emph{Hodge line bundle} on $\sP_{n}^{\CY}$, denoted by $\bfM_n'$, which satisfies the following:
\begin{enumerate}
    \item (functoriality) for any morphism $f\colon T\to \sP_{n}^{\CY}$ from a smooth scheme $T$ with the pull-back family of pairs $f\colon (X,D)\to T$, the line bundle $f^*(\bfM_n')$ is isomorphic to $\lambda_{\Hodge,f}$. 
    \item\label{fact:desent} (ref. \cite[\S 14.4]{ABBDILW23}) there is a positive integer $N$ such that $(\bfM_n')^{\otimes N}$ descends to an ample line bundle on $P_n$.
    \item (ref. \cite[Proposition 14.7 (3)]{ABBDILW23}) The Hodge line bundle associated to the family $f:(X,D)\rightarrow T$ agrees with the $(n+1)$ multiple of the moduli part in the canonical bundle formula for the morphism $\pi$.
\end{enumerate}

    We will denote by $\bfM_n$ the pull-back of $\bfM_n'$ via the root stack $\sP_n\to \sP_{n}^{\CY}$.

\begin{lemma}
    Let $\cR$ be the root stack of a smooth curve $R$ at finitely many points, and $\phi\colon\cR\to \sP_n$ be a morphism. Let $\pi\colon(\sY,\frac{1}{n+1}\sD)\to \cR$ be the pull-back of the universal family, and $\pi_R:(Y,\frac{1}{n+1}D)\rightarrow R$ be the induced morphism on the coarse spaces. Then the following are equivalent:
    \begin{enumerate}
        \item not all the fibers of $\pi\colon(\sY,\frac{1}{n+1}\sD)\to \cR$ are $S$-equivalent;
        \item the moduli part in the canonical bundle formula for $\pi_R$ has positive degree;
        \item the morphism $R\rightarrow P_n$ is non-constant.
    \end{enumerate}
\end{lemma}
\begin{proof}
    Consider a finite and generically \'etale morphism $C\to \cR$ from a smooth curve $C$, and let $(Y_C,\frac{1}{n+1}D_C)$ be the pull-back family. The moduli part in the canonical bundle formula associated to the family $(Y,\frac{1}{n+1}D)\to R$ is, by definition, the push-forward of the moduli part on $C$. In particular, it suffices to check the desired statement for $(Y_C,\frac{1}{n+1}D_C)\to C$. But in this case it follows from \Cref{remark_S-equivalence} and the fact (\ref{fact:desent}).
\end{proof}

\subsection{Four unordered points on $\bP^1$}
In this subsection, we collect some facts about $\sP_1$ which we will use for constructing explicit examples in \Cref{section_application_to_KSBA_moduli_when_epsilon_2_goes_to_0} and \Cref{section_example_14_case}.

The following result is well-known, and we report it for convenience.

\begin{lemma}\label{lemma_order_stabilizers}
    The following holds true for $\sP_{1,\dm}$; see Figure \ref{fig:points in sP1,dm}.
    \begin{enumerate}
        \item The $k$-point $x_3$ in $\sP_{1,\dm}$ corresponding $(\bP^1,[1]+[\zeta]+[\zeta^2]+[\infty])$, where $\zeta$ is a primitive third root of unity, has  $\Aut_{\sP_1}(x_3)\simeq (\bZ/2\bZ)^2\rtimes (\mb{Z}/3\bZ)$;
        \item The $k$-point $x_2$ on $\sP_{1,\dm}$ representing $(\bP^1,[1]+[i]+[-1]+[-i])$ has $\Aut_{\sP_{1,\dm}}(x_2) \simeq (\bZ/2\bZ)^2\rtimes (\mb{Z}/2\bZ)$.
        \item The point $x_{\infty}\in \sP_{1,\dm}$, representing the union $\mtc{C}$ of two root stacks of $\bP^1$'s along $\cB \bmu_2$ each having two distinct marked points, satisfies that \[\begin{tikzcd}
	1 & {\bZ/2\bZ} & {\Aut_{\sP_{1,\dm}}(x_{\infty})} & {(\bZ/2\bZ)^2\rtimes (\mb{Z}/2\bZ)} & 1
	\arrow[from=1-1, to=1-2]
	\arrow[from=1-2, to=1-3]
	\arrow["\phi", from=1-3, to=1-4]
	\arrow[from=1-4, to=1-5]
\end{tikzcd}\]
        where the image of $\bZ/2\bZ$ in $\Aut_{\sP_{1,\dm}}(x_\infty)$\footnote{ usually called the \textit{ghost automorphism}; see \cite[\S 1.5]{di2021polarized}} acts trivially on the coarse moduli space of the twisted curve $\mtc{C}$, and the map $\phi$ is the restriction of the automorphisms on $\cC$ to its coarse space.
        \item A general point $x_t$ of $\sP_{1,\dm}$ has $\bmu_2\times \bmu_2$ as its stabilizer.
    \end{enumerate}

\begin{figure}
    \centering
    \includegraphics[width=0.4\linewidth]{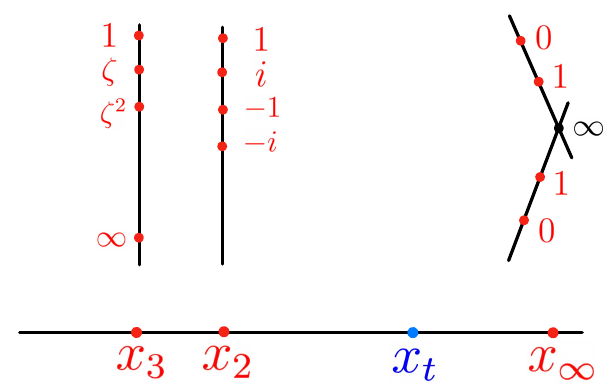}
    \caption{Points on the DM stack $\sP_{1,\dm}$}
    \label{fig:points in sP1,dm}
\end{figure}
\end{lemma}
%\begin{proof}Consider $\infty\in \bP^1$ the point in the good moduli space of $\sP_1$ which is not stable, and let $\cA^1:=(\bP^1\setminus \{\infty\})\times_{\bP^1}\sP_1$. The stabilizer of the generic point in $\cA^1$ is $\bmu_2\times \bmu_2$, so it suffices to prove that there is a point $x_2$ with $\bmu_4$ as a subgroup of $\Aut_{\cA^1}(x)$ and a point $x_3$ with $\bmu_3$ as a subgroup of the automorphisms.The first corresponds to the configuraiton of 4 points on the edges of a square, with the $\bmu_4$ subgroup is the rotation of 90 degrees, and the latter corresponds to the configuration where three points are on the edges of a regular triangle and the third one in the center. The generation of the $\bmu_3$ is again a rotation. \end{proof}

\begin{lemma}\label{lemma_if_C_is_a_scheme_on_the_pteimage_of_x_i_then_i_divides_the_order_of_f}
    Let $\cC$ be the root stack of a smooth curve $C$ at a point, with a map $\phi\colon\cC\to \sP^1$ inducing $f\colon C\to \bP^1$ on good moduli spaces. Let $x_2$ (resp. $x_3$) as in Lemma \ref{lemma_order_stabilizers}. 
    \begin{itemize}
        \item If $p\in \cC$ is a schematic point mapping to $x_2$, then $f$ ramifies of order $2k$ at $p$.
        \item If $p\in \cC$ is a schematic point mapping to $x_3$, then $f$ ramifies of order $3k$ at $p$.
    \end{itemize}
    In particular, if $\cC$ is a scheme along $\phi^{-1}(x_2)$ (resp. $\phi^{-1}(x_3)$), then $2|\deg(f)$ (resp. $3|\deg(f)$).
\end{lemma}
\begin{proof}
Let $\infty\in \bP^1$ be the point in the good moduli space of $\sP_1$ which is not stable, and let $\cA^1:=(\bP^1\setminus \{\infty\})\times_{\bP^1}\sP_1$. Then $\cA^1$ is a smooth Deligne-Mumford stack with coarse moduli space $\bA^1$, so from \cite[Theorem 1 and Remark 2]{geraschenko2017bottom} we can factor $\cA^1\to \bA^1$ as $$\cA^1\ \xrightarrow{\alpha} \ \cA^1/\!\!/G\ \xrightarrow{\beta} \ \bA^1,$$ where $\alpha$ is a gerbe and $\beta$ is a root stack. We can now use Lemma \ref{lemma_order_stabilizers} to control the stabilizers of $\cA^1/\!\!/G$: the stabilizer of the geometric generic point in $\cA^1$ is $\bmu_2\times \bmu_2$, and there is a point $x$ with $\bmu_4$ as a subgroup of $\Aut_{\cA^1}(x)$ (namely, $x_2$) and a point $y$ with $\bmu_3$ as a subgroup of the automorphisms (namely, $x_3$). In particular, one has \[
\Aut_{\cA^1/\!\!/G}(x_2)=\bmu_2\ \ \ \text{ and } \ \ \ \Aut_{\cA^1/\!\!/G}(x_3)=\bmu_3.
\]
The conclusion follows since an affine curve $D$ with a morphism $\xi:D\to \bA^1$ which factors via $D\to \cA^1/\!\!/G\to \bA^1$ is totally ramified at $\xi^{-1}(x_i)$ of order a multiple of $i$. 
\end{proof}
\begin{lemma}\label{lemma_M11_to_P1_is_injective}
    There is a morphism $\cM_{1,1}\to \sP_1$ that induces an injective map $\bA^1\to \bP^1$ on good moduli spaces.
\end{lemma}
\begin{proof}
    Consider the universal elliptic curve $\cC\to \cM_{1,1}$ and its quotient $\cD\to \cM_{1,1}$  by the hyperelliptic involution. If we denote by $\Delta\subseteq \cD$ the ramification divisor, then the pair $(\cD,\Delta)\to \cM_{1,1}$ induces a morphism $\cM_{1,1}\to \sP_1$. To check that it is injective on good moduli spaces, we use a test family. Consider the family of cubics
    \[
Y\ :=\ \bV\big(y^2z - x(x-z)(x-tz) \big) \ \subseteq\  \bA^1_{t}\times \bP^2_{x,y,z}
\]
over $\bA^1_t$ with the constant section $[0,1,0]$. Away from $t=0, 1$, this gives a morphism $\bA^1_t\setminus\{0,1\}\to \cM_{1,1}$, whose composition with the coarse moduli space $\cM_{1,1}\to M_{1,1}$ is generically 6:1 by \cite[Proposition 1.7]{silverman2009arithmetic}. The quotient by the hyperelliptic involution is $$\textstyle \big(\bA^1_t\times \bP^1_{x,z}, \ \frac{1}{2}\bV\big(xz(x-z)(x-tz)\big)\big).$$
Now, observe that one can give an ordering to the four markings, and that the corresponding morphism $\bA^1_t\setminus\{0,1\}\rightarrow \cM_{0,4}$ is injective.
From this, we conclude that also the morphism to $\bP^1$ is generically also 6:1 onto its image: this follows from the fact that given four ordered points on $\bP^1$ with cross ratio $k$, the possible cross ratios of all the permutations of these points are the six values $$\textstyle \left\{k,\ \frac{1}{k},\  1-k, \ \frac{1}{1-k},\  \frac{k}{k-1},\ \frac{k-1}{k}\right\}.$$ Therefore, the map between good moduli spaces $M_{1,1}\to \bP^1$ is injective.
\end{proof}

%\begin{notation}From \cite[Proposition 4.7]{BABWILD} there is a line bundle $\cM$ on $\sP_1$ such that for every morphism $\phi\colon S\to \sP_1$ from a smooth scheme $S$, if $f\colon (X,D)\to S$ is the corresponding family of $\bP^1$s, then $$f^*(2K_S + \phi^*\cM) \ =\  2K_X + \Delta.$$ One can check that the line bundle $\cM^{\otimes 12}$ descends to a line bundle on $\bP^1$. We call the class $\frac{1}{2}\cM\in \Pic(\sP_1)_{\bQ}$ the \emph{moduli part}.\Gio{do i need this notation}\end{notation}
The following results will regard the moduli part in a fibration coming from a morphism from a (twisted) curve to $\sP_1$.
\begin{lemma}\label{lemma_relation_degree_of_map_to_p1_and_moduli_part}
    Assume that $\phi\colon \cC\to \sP_1$ is a morphism from the root stack $\cC\to C$ of a smooth curve $C$, such that the composition $f\colon \cC\to \sP_1\xrightarrow{p} P_1$ is finite and such that $\phi^{-1}(\sP_{1,\dm})$ is a scheme in $\cC$. Then $\deg(f)$ is divisible by $6$. Moreover, if $\cC = C$ is a scheme and we denote by $g\colon (X,\frac{1}{2}D)\to C$ the family of pairs induced by $\pi\circ\phi\colon C\to \sP_1$, then $g$ has no boundary part and the moduli part is of degree $\frac{1}{12}\deg f$ on $C$. In particular, one has
    \[\textstyle
    K_X+\frac{1}{2}D\ \sim_\bQ \ g^*(K_C+\bfM)\ \ \ \  \text{ and }\  \ \ \deg \bfM \ =\ \frac{\deg(f)}{12}
    \]
\end{lemma}
\begin{proof}
    The first part follows from Lemma \ref{lemma_if_C_is_a_scheme_on_the_pteimage_of_x_i_then_i_divides_the_order_of_f}.
For the statement on the canonical bundle formula, observe that the boundary part is empty as all the fibers of $\pi\circ \phi$ are semi-log canonical. From \Cref{subsubsection_moduli_part}, to compute the moduli part it suffices to use a test family; we choose the test family given by the quotient by the hyperelliptic involution on a pencil of cubics in Weierstrass form. Choose $A_1, A_2, B_1, B_2$ be generic enough so that the pencil of genus one curves
\[
Y\ :=\ \bV\big( s(y^2z-x^3-A_1xz^2-B_1Z^3) + t(y^2z-x^3-A_2xz^2-B_2Z^3)\big)\ \subseteq\  \bP^2_{x,y,z}\times \bP^1_{s,t}
\]
is smooth. The second projection
$f\colon Y\to \bP^1$ gives an elliptic surface with $j$ map of degree $12$. Let $q\colon Y\to X$ be the quotient by the hyperelliptic involution,
with morphism $f'\colon X\to \bP^1$. This is the restriction to $Y$ of the quotient map $$\bP^2\times \bP^1\ \longrightarrow\  \bP(1,2,1)\times \bP^1,\ \ \ \ \ ([x,y,z],[s,t])\ \mapsto\  ([x,y^2,z],[s,t]).$$ For every point $p=[s,t]\in \bP^1$ we have \[(f')^{-1}(p) \ =\ \bV\big(YZ-X^3-\alpha XZ^2 -\beta Z^3\big)\ \subseteq\  \bP(1,2,1)\]for an appropriate choice of $\alpha$ and $\beta$, where $X,Z$ have weight $1$ and $Y$ has weight $2$.
In particular, one can check that the fibers of $f'$ are isomorphic to $\bP^1$. Observe that the ramification divisor $\Delta\subseteq X$ is such that $(X,\Delta)\to \bP^1$ induces a map $\psi:\bP^1\to \sP^1$. Namely, for every $p\in \bP^1$, the pair $((f')^{-1}(p),\frac{1}{2}\Delta|_{(f')^{-1}(p)})$ is log canonical. Applying the canonical bundle formula for minimal elliptic surfaces with a section, we have that \[\textstyle f^*(2K_{\bP^1}+\frac{1}{6}\deg j)\ \sim \ 2K_Y\ \sim \  q^*(2K_X+\Delta)\  =\  q^*(f')^*(2K_\bP^1 + \psi^*\cM)\] where $j\colon\bP^1\to \overline{M}_{1,1}$ is the $j$-map associated to $f:Y\rightarrow \bP^1$ and $\frac{1}{2}\cM$ is the moduli part. Taking the push-forward of the left-hand side and the right-hand side both to $\bP^1$ and using $f_*\cO_Y=\cO_{\bP^1}$, we have that $2K_{\bP^1}+\frac{1}{6}\deg j = 2K_\bP^1 + \psi^*\cM$. By Lemma \ref{lemma_M11_to_P1_is_injective}, one has that $\deg f=\deg j$, and thus the moduli part has degree $\frac{\deg(f)}{12}$ as desired.
\end{proof}

\begin{corollary}\label{cor:degree 6n}
    Assume that $C\subseteq \bP^1\times \bP^1$ is a smooth curve of class $(n,4)$, such that for every $x\in \bP^1$, the pair $\big(\pi^{-1}_1(x),\frac{1}{2}C|_{\pi^{-1}_1(x)}\big)$ is log canonical. Then the pair $(\bP^1\times \bP^1,\frac{1}{2}C)$ induces a morphism $\phi:\bP^1\to \sP_1^{\GIT}\xrightarrow{i} \sP_1$, such that the composition $f:\bP^1\to \sP_1\xrightarrow{p} P_1$ has degree $6n$.
\end{corollary}

\begin{proof}
   One can compute the canonical bundle formula explicitly, and the moduli part has degree $\frac{n}{2}$. The desired statement follows from \Cref{lemma_relation_degree_of_map_to_p1_and_moduli_part}.
\end{proof}

\begin{lemma}\label{lem: Hirzebruch}
    Let $\phi:\bP^1\to \sP_1$ be a morphism such that the composition $ \bP^1 \rightarrow   \sP_1  \rightarrow  P_1 $ is finite. Let $\pi:\sP_1\xrightarrow{\pi}\sP_1^{\GIT}$ be the retraction morphism, and $f\colon X\to \bP^1$ be the $\bP^1$-bundle induced by $\pi$. Then $X$ is isomorphic to a Hirzebruch surface $\mathbb{F}_n$ such that $6n\le \deg f$.
\end{lemma}
\begin{proof}
First observe that $X$ is a $\bP^1$-fibration over $\bP^1$, and hence $X\simeq \bF_n$ for some $n$. We may assume that $n\geq 1$. Let $\mtf{e}\subseteq X$ be the negative section, and $\mtf{f}$ be the class of a fiber of $X\to \bP^1$. One has that $K_X \sim -2\mtf{e} + (-n-2)\mtf{f}$. By Lemma \ref{lemma_relation_degree_of_map_to_p1_and_moduli_part}, one has
    \[\textstyle
    K_X + \frac{1}{2}D \ \sim \  f^*K_{\bP^1}+\frac{\deg(f)}{12}\mtf{f}\]  and thus \[ \textstyle D \ \sim \  -2K_X + 2f^*(K_{\bP^1})+\frac{\deg(f)}{6}\mtf{f}\ =\  4\mtf{e} + (2n + \frac{\deg(f)}{6})\mtf{f}.
    \]
    Assume by contradiction that $n>\frac{\deg(f)}{6}$. Then $(D.\mtf{e})<0$, and as both are effective divisors, we must have that $\mtf{e}\subseteq D$. However, if $n>\frac{\deg(f)}{6}$, we also have that $(D - \mtf{e}).\mtf{e}<0$, so $2\mtf{e}\subseteq D$. This is impossible: since $C\to \bP^1$ was finite, the generic fiber of $f|_{D}:D\to C$ consists of four distinct points.
\end{proof}
We end this section with the following lemma which will be used in Section \S\ref{sec:KSBA limit} to guarantee that certain minimal models are canonical models.

\begin{lemma}\label{lemma_on_tails_i_have_a_marking}
Let $\cR$ be the $\mu_k$-root stack of $\bP^1$ at a point $p$, and $\phi\colon \cR\to \sP_n$ be a morphism such that $\phi(p)\in \sP_{n,\dm}$ and that the induced morphism between coarse spaces $\bP^1\to P_n$ is finite. Assume that if there is $q\in \cR\smallsetminus\{p\}$ such that $\phi(q)$ represents a nodal curve, then the image of every point of $\cR$ represents a nodal curve. Then there is a point $x\in \cR\setminus\{p\}$ such that $\phi(p)$ represents a pair with divisor supported on less than $2n+2$ points.
\end{lemma}
\begin{proof}
    We argue by contradiction. Let $f\colon (\sX,\frac{1}{n+1}\sD)\rightarrow \cR$ be the pull-back along $\phi$ of the universal family over $\sP_n$. If all the fibers of $f$ had $2n+2$ distinct markings away from $p$, then the divisor $\sD$ would be \'etale over $\cR\smallsetminus\{p\}$. As the only \'etale cover of $\cR\smallsetminus\{p\}$ is the trivial one, then $\sD$ consists of $2n+2$ disjoint sections. Observe that there is a finite cover $\bP^1\to \cR$, which is \'etale over $p$: indeed, the $\mu_k$-root stack of $\bP^1$ at \textit{two} points is isomorphic to the quotient stack $[\bP^1/\bmu_k]$, where $\mu_k$ acts on $\bP^1$ by $$\xi*[a:b] \ :=\  [\xi a:b].$$ 
     Then $[\bP^1/\bmu_k]$ admits a map to the root stack of order $k$ at a single point $p$: there is a morphism $[\bP^1/\bmu_k]\to \cR$ which is the root stack at a point $q\in \cR\smallsetminus \{p\}$.

    Now pulling back the universal family of $\sP_n$ along $\bP^1\rightarrow \cR\rightarrow \sP_1$, one gets a surface pair $(X,\frac{1}{n+1}D)$ where $D$ consists of $2n+2$ distinct sections. Observe also that, as by assumption $\phi(p)\in\sP_{n,\dm}$, if the fiber over $p$ is $\bP^1$ then at least three of these $2n+2$ sections are disjoint, and if the fiber is nodal at least four are disjoint. We distinguish three cases.
    
    \textbf{Case 1}: the fiber over $p$ is $\bP^1$. Then by assumption $X\rightarrow \bP^1$ is a $\bP^1$-fibration. As $D$ consists of at least 3 disjoint sections, it has to be that $X\simeq\bP^1\times\bP^1$, and $D\sim \mtc{O}_{\bP^1\times\bP^1}(2n+2,0)$, i.e. $$(X,D)\ \simeq\  (\bP^1,p_1+\cdots+p_{2n+2})\times\bP^1,$$ and hence the moduli part is constant. But the moduli part on $\bP^1$ is the pull-back of the one on $\sC$, so the moduli part on $\sC$ has degree 0 which contradicts the finitedness of $C\to \sP_1\to \bP^1$.
    
    \textbf{Case 2}: the fiber over $p$ is nodal, and the generic fiber is smooth. Then we can contract one of the two branches, let $E$ be such branch. The new surface is a ruled surface with $2n+2$ sections, with at least three of them disjoint,
    so it is again $\bP^1\times\bP^1$. Moreover, if we denote by $F_1, \ldots,F_{n+1}$ the connected components of $D$ which intersect $E$, then the images of $F_i$ intersect once contracting $E$, while not intersecting the image of $D- F_1- \ldots-F_{n+1}$. This is not possible as two fibers of the same fibration in $\bP^1$ have the same numerical class.

    \textbf{Case 3}: the generic fiber is nodal. This is not possible as the nodal locus in $P_n$ is a single point, and by assumption $\cR\to \sP_n$ induces a finite morphism on good moduli space.
\end{proof}
\subsection{Addendum}
The following lemmas, familiar to experts, will be invoked in the foregoing discussion. We state them here separately, as they are of a different nature.

\begin{lemma}\label{lemma_root_stack_of_a_curve_has_no_brauer}
    Let $C$ be a smooth curve, and $\cC\rightarrow C$ be a sequence of root stacks so that $\cC$ is smooth. Then any $\bP^1$-fibration $f\colon \cS\to \cC$ over $\cC$ is the projectivization of a rank $2$ vector bundle. 
\end{lemma}

\begin{proof}
    Let $U\subseteq \cC$ be the open subscheme  restricting on which $\cC\rightarrow C$ is an isomorphism, and the morphism $f|_U$ admits a section $\sigma_U$. Since the composition $\cS\rightarrow \cC\rightarrow C$ is proper, then by \cite[Theorem 3.1]{BV24}, $\sigma_U$ can be extended to a (representable) section $\cC'\to \cS$, where $\cC'\rightarrow C$ is a (unique) root stack. As both $\cC'\to \cS$ and $\cS\to \cC$ are representable, then so is the composition $\cC'\to \cC$, and hence by Zariski's main theorem $\cC=\cC'$. In particular, $\cS\to \cC$ has a section $\Sigma$, so $\cS$ is the projectivization of $f_*\cO_{\cS}(\Sigma)$.
\end{proof}
\begin{lemma}\label{lemma_gms_of_seminormal_is_seminormal}
     The good moduli space of a seminormal algebraic stack is seminormal.
\end{lemma}
\begin{proof} The good moduli space $\cX\to X$ of an algebraic stack $\cX$ is, \'etale locally on $X$, of the form $[\spec A/G]\to \spec(A^G)$, and being seminormal can be checked \'etale locally. Therefore, it suffices to prove that $\spec A^G$ is seminormal for any seminormal affine scheme $\spec A$ with an action by a reductive group $G$. As $\spec A$ is reduced, then so is $\spec A^G$. Since $\Spec A$ is seminormal, then the quotient $\Spec A\rightarrow \Spec A^G$ factors through the seminormalization $\nu\colon \spec (A^G)^{\sn}\to\spec(A^G)$. By the universal property of categorical quotients, $\nu$ has a section $s\colon \spec A^G\to \spec(A^G)^{\sn}$, which has connected fibers as $\nu$ is a homeomorphism. By \cite[Lemma 2.3]{MR4705376}, there is an isomorphism
    \[
    s_*\cO_{\spec A^G} \ \cong  \ \cO_{\spec(A^G)^{\sn}}.
    \] 
    In particular $\spec(A^G)^{\sn}$ has the same underlying topological space as $\spec A^G$ with the isomorphic structure sheaf, thus $\spec(A^G)^{\sn}\cong \spec A^G$. 
\end{proof}

\begin{lemma}\label{lemma_flops_are_crepant}
    Let $\pi\colon (X,D)\to (Y,B_Y+\bfM)$ be a morphism of lc generalized pairs such that $K_X+D\sim_\bQ \pi^*(K_Y+B_Y+\bfM)$. Let $\tau^{\st}\colon Y\to (Y^{\st},L)$ be a morphism to a polarized normal variety such that \[(\tau^{\st})^*L  \ \cong\  \cO_Y(m(K_Y+B_Y+\bfM))\] for some $m>0$.
    Let $0<|\epsilon| \ll 1$ and $\tau\colon (X,(1+\epsilon)D)\to (X',(1+\epsilon)D')$ be a contraction of an extremal ray. Then there is a morphism $\pi'\colon X'\to Y^{st}$ and $m(K_{X'}+D')\sim_\bQ(\pi')^*L$. In particular, $\tau:(X,D)\rightarrow (X',D')$ is crepant.
\end{lemma}

\begin{proof}
First observe that any curve contracted by $\tau$ is also contracted by $\tau^{\st}\circ \pi$, as $L$ is ample, $|\epsilon|$ is small, and
\[\textstyle
K_X+(1+\epsilon)D \ \sim_\bQ\  \pi^*(K_Y+B_Y+\bfM) +\epsilon D \ \sim_{\bQ} \ \frac{1}{m}\pi^*(\tau^{st})^*L+\epsilon D .
\]
    Therefore, there is a morphism $\pi'\colon X'\to Y^{\st}$. As $\tau$ is a birational contraction, and \[\textstyle K_X+D\ \sim_\bQ\  \pi^*(K_X+B_Y+\bfM) \ \sim_\bQ \ \frac{1}{m}\pi^*(\tau^{st})^*L,\]
    then the two $\bQ$-Cartier divisors $K_{X'}+D'$ and $\frac{1}{m}(\tau^{\st})^*L$ agree on the normal variety $X'$.
\end{proof}

\section{Minimal models for fibered log Calabi-Yau pairs}\label{section_stable_reduction}

The main objective of this section is to establish the following two results.  
First, we show that if $\pi\colon (X,B_X) \to Y$ is a fibration whose fibers are log Calabi-Yau pairs, and if a step of the MMP 
$Y \dashrightarrow Y'$ for a generalized pair on \(Y\) (induced by \(\pi\)) is given, then this step can be followed by a corresponding step of the MMP 
\(X \dashrightarrow X'\) as stated in \Cref{prop_can_run_mmp_upstairs_to_follow_the_one_downstairs}.  

We then apply this to the case where the fibration arises from a stable quasimap \(\cC \to \sP_n\)
over the spectrum of a DVR \(R\). In this setting, we obtain a structural description of the morphism \(X' \to Y'\) in terms of ruled surfaces; see \Cref{thm_you_can_take_ruled_model_in_ksba_moduli}.

\subsection{An MMP for fibrations in log Calabi-Yau pairs}\label{section_mmp}

In this subsection, we establish certain results concerning fibrations in log Calabi--Yau pairs. Roughly speaking, we show that if 
\[
\pi\colon (X,B_X) \ \longrightarrow \ (Y,B+\mathbf{M})
\]
is a fibration of a log Calabi-Yau pair onto a generalized pair \((Y,B+\mathbf{M})\) satisfying 
\[
\pi^*(K_Y+B+\mathbf{M})  \ \sim_{\mathbb{Q}} \ K_X+B_X,
\]
then, under mild assumptions, any MMP for \((Y,B+\mathbf{M})\) can be accompanied by a birational modification for \((X,B_X)\) such that each step of the MMP for \((Y,B+\mathbf{M})\) still admits a fibration in log Calabi--Yau pairs given by a birational model of \((X,B_X)\); see \Cref{prop_can_run_mmp_upstairs_to_follow_the_one_downstairs}.

\begin{lemma}\label{lemma_ext_contraction_downstairs_induces_one_upstairs}
    Let $(Y,B+\bfM)$ be an lc generalized pair, let $p\colon(Y,B+\bfM) \longrightarrow (Y',B'+\bfM')$ be a birational contraction of a $(K_Y+B+\bfM)$-negative extremal ray with exceptional locus $\Exc(p)\subseteq Y$. Let $\pi\colon(X,B_X)\to Y$ be a morphism from an lc pair $(X,B_X)$ such that $$\pi^*(K_Y+B+\bfM) \ \sim_{\bQ}\ K_X + B_X,$$ and $\Xi\subseteq X$ be a divisor which maps surjectively to $\operatorname{Ex}(p)$.
    Let $\pi'\colon (X',B_{X'})\to (Y',B'+\bfM')$ be a weak canonical model of $(X,B_X)$ over $Y'$. Then $\Xi$ is contracted via the rational map $X\dashrightarrow X'$.
\[\begin{tikzcd}
	&& W \\
	\Xi & {(X,B_X)} && {(X',B_{X'})} \\
	{\Exc(p)} & {(Y,B+\bfM)} && {(Y',B'+\bfM')}
	\arrow["q"', from=1-3, to=2-2]
	\arrow["{q'}", from=1-3, to=2-4]
	\arrow["\subseteq"{description}, draw=none, from=2-1, to=2-2]
	\arrow[two heads, from=2-1, to=3-1]
	\arrow["{(p\circ \pi)\textup{-MMP}}", dashed, from=2-2, to=2-4]
	\arrow["\pi"', from=2-2, to=3-2]
	\arrow["{\pi'}", from=2-4, to=3-4]
	\arrow["\subseteq"{description}, draw=none, from=3-1, to=3-2]
	\arrow["p", from=3-2, to=3-4]
\end{tikzcd}\]
\end{lemma}

\begin{proof}
Assume by contradiction that $\Xi$ is not contracted. Let $q\colon W\to X$ and $q'\colon W\to X'$ be a common resolution of $X\dashrightarrow X'$. Then we have 
\[q^*(K_X+B_X)\ =\ (q')^*(K_{X'} + B_{X'} ) + E,\] where
$E$ is an effective $q'$-exceptional divisor which does not contain $\Xi_W$, the proper transform of $\Xi$ in $W$.
As $\Xi$ maps surjectively to the exceptional locus of $\pi$, and as $\Xi $ is not contracted, there is an integral curve $C\subseteq \Xi_W$ which:
\begin{enumerate}
    \item is not contained in $\Supp E$;
    \item maps finitely to $\operatorname{Ex}(p)$ and $X'$; and
    \item the class of $\ove{C}:=\pi_*q_*C$ is contained in the $p$-extremal ray, and thus $$(\ove{C}. \ K_Y+B+\bfM)\ <\  0.$$
\end{enumerate}
It follows that \begin{equation}\nonumber
\begin{split}
    0\ >\ \big(C.\ q^*(K_X+B_X)\big) \ & =\  \big(C.\ q'^*(K_{X'} + B_{X'} ) + E\big)\\ 
    & \ge \ \big(C.\ q'^*(K_{X'} + B_{X'})\big)\  =\  \big(q'_*C.\ (K_{X'} + B_{X'})\big),
\end{split}
\end{equation}
which contradicts the nefness of $K_{X'}+B_{X'}$.
\end{proof}

\begin{prop}\label{prop_can_run_mmp_upstairs_to_follow_the_one_downstairs}
    Let $(X,B_X)$ be a klt pair with a fibration $\pi\colon(X,B_X)\to (Y,B_Y+\bfM_Y)$ to a klt generalized pair such that $$K_X+B_X \ \sim_\bQ \ \pi^*(K_Y+B_Y+\bfM_Y) \text{ and }\pi_*\cO_X=\cO_Y.$$ Assume that
    \begin{enumerate}
        \item $B_X$ is $\pi$-ample and $\bQ$-Cartier, and
        \item there is 
        \begin{itemize}
           \item either a divisorial contraction $p:Y\rightarrow Y'$ of a $(K_Y+B_Y+\bfM_Y)$-extremal ray,
           \item or a flip $p\colon Y\dashrightarrow Y^+$ of a $(K_Y+B_Y+\bfM)$-extremal ray, which factors as \[\begin{tikzcd}
	Y && {Y^+} \\
	& {Y'}
	\arrow["p", dashed, from=1-1, to=1-3]
	\arrow["{p^{-}}"{description}, from=1-1, to=2-2]
	\arrow["{p^+}"{description}, from=1-3, to=2-2]
\end{tikzcd}.\]
        \end{itemize} 
    \end{enumerate}
     Then there exists a weak canonical model $(X,B_X)\dashrightarrow  (X',B_{X'})$ over $Y'$ such that
    \begin{itemize}
    \item if $Y\to Y'$ is divisorial, then one can make $B_{X'}$ ample over $Y'$;
    \item if $Y\dashrightarrow Y^{+}$ is a flip, then one can choose $(X',B_{X'})\to Y'$ such that it factors as $(X',B_{X'})\to Y^+\to Y'$ with $B_{X'}$ ample over $Y^+$;
    \item if $K_{Y'}+B_{Y'}+\bfM_{Y'}$ is ample (resp. nef), then the pair $(X',(1+\epsilon)B_{X'})$ is a (weak) canonical model for $(X,(1+\epsilon)B_X)$ for any $0<\epsilon \ll 1$.        
    \end{itemize}
\end{prop}
\begin{proof}
Let $(X,(1+\epsilon)B_X)\dashrightarrow (X',(1+\epsilon)B_{X'})$ be the relative canonical model over $Y'$, which by \cite[Theorem 5.59]{HK} does not depend on $\epsilon$ as long as $0<\epsilon \ll 1$. In particular, $(X,B_X)\dashrightarrow (X',B_{X'})$ is a relative weak canonical model over $Y'$. Then the canonical model of $(X,B_X)$ over $Y$ agrees with that of $(X',B_{X'})$ over $Y'$, which is isomorphic to 
\[\textstyle
\Proj_{Y'} \bigoplus_m p_*\pi_*\cO_X(md(K_X+B_{X})) \ = \  \Proj_{Y'} \bigoplus_m p_*\cO_Y(md(K_Y+B_{Y} + \bfM_Y))
\]
for $d$ positive and divisible enough.
However, the latter is the relative canonical model of $(Y,B_Y+\bfM_Y)$ over $Y'$, which is
\begin{enumerate}
    \item $Y'$ if $Y\to Y'$ is a divisorial contraction;
    \item $Y^+$ if $Y\dashrightarrow Y'$ is a flipping contraction.
\end{enumerate}
In particular, the relative canonical model $(X',B_{X'})\to X^{\can}$ is
\begin{enumerate}
    \item $X'\to Y'$ if $Y\to Y'$ is a divisorial contraction;
    \item $X'\to Y^+$ if $Y\dashrightarrow Y'$ is a flipping contraction.
\end{enumerate} As $X'\to X^{\can}$ is induced by $K_{X'}+B_{X'}$, then we have 
\[
K_{X'} + (1+\epsilon)B_{X'}  \ \sim_{\bQ,X^{\can}} \ \epsilon B_{X'}.
\]
In particular, we have that $B_{X'}$ is ample over $X^{\can}$.
\end{proof}

\begin{remark}\label{rmk_during_MMP_pull_back_of_lc_divisor_is_lc_divispr}\textup{
    With the assumptions of Proposition \ref{prop_can_run_mmp_upstairs_to_follow_the_one_downstairs}, from \Cref{lemma_ext_contraction_downstairs_induces_one_upstairs}, if $Y\to Y'$ is a divisorial contraction and all the divisors $\Xi \subseteq X$ mapping to $\Exc(p)$ map surjectively onto it, then \[(\pi')^*(K_{Y'}+\bfM_{Y'}+B_{Y'}) \ =\  K_{X'}+B_{X'}.\] Observe that the assumptions of \Cref{lemma_ext_contraction_downstairs_induces_one_upstairs} are automatic if $X\rightarrow Y$ is a fibration from a threefold to a surface. In this case, $X'\to Y'$ is also of pure relative dimension one. %However, this might not be true if $Y\to Y'$ is small;
}\end{remark}
We end this subsection by invoking the following remark, which relates KSBA-stable limits to the MMP.
\begin{remark}[\textup{cf. \cite[Corollary 4.57]{Kol23}}]
    \textup{If $R$ is a DVR and $(X,D)\to \spec R $ a locally stable family with a stable generic fiber, then the canonical model of $(X,D)$ over $\spec R $ will be the stable limit of the generic fiber of $(X,D)\to \spec R $.}
\end{remark}

\subsection{From quasimaps limit to KSBA-limit}\label{sec:KSBA limit}

In this subsection, we will use the moduli space $\cQ$ discussed in \Cref{section_prelim_on_qmaps} to construct certain weak canonical models that will allow us to control limits in the KSBA-moduli space of stable pairs. More specifically, consider a moduli space of surface pairs $(X,\frac{1}{n+1}D)$ admitting a fibration $$\textstyle (X,\frac{1}{n+1}D)\ \longrightarrow \  C$$ with fibers isomorphic to log canonical Calabi-Yau pairs of the form $(\bP^1,\frac{1}{n+1}\Delta)$.
Two examples which are worth keeping in mind are the following:
\begin{enumerate}
    \item the (non-compact) moduli space of $(\bP^1\times\bP^1,\frac{1}{n+1}D)$ where $D$ is a generic divisor of class $(2n,k)$ with $k> 2n$,
    \item the (non-compact) moduli space of Hirzebruch surface pairs $(\bF_d,\frac{1}{n+1}D)$, where $D$ the union of the negative section $\mtf{e}$ and a multisection not containing $\mtf{e}$ of the class $(2n-1)\mtf{e}+(2n-1)d\mtf{f}$ with $\mtf{f}$ the class of the fiber of $\bF_d\rightarrow \bP^1$. 
\end{enumerate}
When $n=1$, pairs $(X,\frac{1}{2}D)$ as in the second example appear as the quotient by the hyperelliptic involution of an elliptic surface $Y\to \bP^1$.

 The goal of this section is to establish a few general results on the KSBA-compactification of these moduli spaces, using the moduli space $\cQ$. On a first approximation, we will use the results of \Cref{section_mmp} on the total space of a family of surface pairs corresponding to a one-parameter family in $\cQ$.
 
 For doing so, the first step is to guarantee that such a one-parameter family has the correct singularities, when we replace the coefficient of the divisor from $\frac{1}{n+1}$ with $\frac{1}{n+1} +\epsilon$ (\Cref{lemma_qmap_limit_plus_divisor_has_lc_sing}), and we add to the divisor some of the fibers with small coefficient (\Cref{cor_singularities_remain_slc_if_increasing_coeff_on_d_and_adding_fibers}). 

\begin{lemma}\label{lemma_qmap_limit_plus_divisor_has_lc_sing}
    Let $R$ be a DVR, $\sC\to \spec R $ a family of twisted curves, and let $\pi\colon (\sX,c\sD)\to \sC$ a proper morphism of Deligne-Mumford stacks, with $(\sX,c\sD)$ with slc singularities where $0<c<1$ is a rational number. Assume that
\begin{enumerate}
    \item $\sD$ is a $\bQ$-Cartier divisor;
    \item\label{condition_on_lct} for every point $x\in \sC_\eta$ (resp. $x\in \sC_0$) contained in the smooth locus of the family of twisted curves $\sC\to \spec R$, \begin{center}$\operatorname{lct}(\sX_\eta,c\sD_\eta;\pi^{-1}(x))>0$ (resp. $\operatorname{lct}(\sX_0,c\sD_0;\pi^{-1}(x))>0$); and\end{center}
    \item there is an $0<\epsilon\ll 1$ such that $(\sX,(c+\epsilon)\sD)\to \sC$ is locally stable, except possibly along a proper closed subscheme $\Delta\subseteq \sC$ contained in the smooth locus of each fiber of $\sC\to\spec R $. 
\end{enumerate}
    Let $(X,c{D})$ be the coarse moduli space of $(\sX,c\sD)$. Then $K_X$ is $\bQ$-Cartier, and the pair $(X,(c+\epsilon){D}+ X_0)$ is slc, where $0<\epsilon\ll 1$. In particular, if the generic fiber $(X_\eta,cD_\eta )$ is klt and $\pi$ is localy stable, then also $(X,(c+\epsilon){D})$ is klt for $0<\epsilon \ll 1$. 
\end{lemma}
Condition (\ref{condition_on_lct}) is automatically satisfied if either $\pi$ is locally stable, or if it is a log Calabi-Yau fibration and the boundary part in the canonical bundle formula appears with coefficient less than 1 along $\sC_\eta$ and $\sC_0$.
\begin{proof}
It suffices to prove that $(\sX,(c+\epsilon)\sD+\sX_0)$ is slc. This follows as being slc is an \'etale local property, and the map $\sX\to X$ \'etale locally is the quotient by a finite group, so $(\sX,(c+\epsilon)\sD+\sX_0)$ being slc is equivalent to $(X,(c+\epsilon){D}+ X_0)$ being slc.

By adjunction, it suffices to prove that $(\sX_0,(c+\epsilon)\sD_0)$ is slc. By assumption and adjunction, the pair $(\sX_0,(c+\epsilon)\sD_0)$ is slc, except possibly along the fiber $\{F_1,\ldots,F_k\}$ of finitely many smooth points $\{x_1,\ldots,x_k\}\subseteq \sC$; we now show it is slc over each point $x_i$. 
    
    The pair $(\sX_0,(c+\epsilon)\sD_0)$ is not slc if and only if there is a divisor $E$ whose center is on the fiber $(\sX_0)_{x_j}$ for some $j$, such that the discrepancy $a(E;\sX_0,(c+\epsilon)\sD_0)<-1$.
    However, by condition \ref{condition_on_lct} one has $ a(E;\sX_0,c\sD_0 + \epsilon' F_j)\geq-1$ for $0<\epsilon'\ll 1$. As $F_j$ is a Cartier divisor which contains the center of $E$ on $\mts{X}_0$, then $\ord_E(F_j)\geq 1$ and consequently \[-1\ \le\  a(E;\sX_0,c\sD_0 +  \epsilon'F_j) \ =\  a(E;\sX_0,c\sD_0) - \epsilon'\mult_E(F_j) \ \leq \ a(E;\sX_0,c\sD_0) - \epsilon', \] and hence $a(E;\sX_0,c\sD_0)> -1$.  Therefore, for any $0<\epsilon \ll 1$ we have that $$a(E;\sX_0,(c+\epsilon)\sD_0 )\ >\ -1$$ as desired.

    The last part of the statement about klt follows from \cite[Theorem 2.80]{Kol23}: for locally stable families over a smooth curve, the log-canonical center of the total space is flat over the base. Therefore, if $(X_\eta,cD_\eta)$ is klt, then $(X_\eta,(c+\epsilon)D_\eta +\epsilon F_\eta)$ is also klt for $0<\epsilon\ll1$, and hence so is $(X,(c+\epsilon)D +\epsilon F)$.
    \end{proof}
\begin{corollary}\label{cor_singularities_remain_slc_if_increasing_coeff_on_d_and_adding_fibers}
    Let $\pi_\sX\colon (\sX,c\sD)\to \sC\to \spec R $ be as in \Cref{lemma_qmap_limit_plus_divisor_has_lc_sing} with $(\sX,c\sD)\to \sC$ locally stable, let $c_S\in (0,1]_{\bQ}$ and let $S\subseteq \sC$ be a Cartier divisor such that $(\sC, (c_S+\epsilon) S)$ is slc for some $0<\epsilon\ll 1$. Then if we denote by $\sF$ the fibers of $\pi_\sX$ along $S$, and by roman letters the coarse moduli spaces of each Deligne-Mumford stack, the pair 
    \[
    (X,(c+\epsilon)D + c_SF + X_0)
    \]
    is slc for $0<\epsilon \ll 1$. Moreover, if the generic fiber $(X_\eta,cD_\eta )$ is klt, then also $(X,(c+\epsilon){D} + c_S F)$ is klt for $0<\epsilon \ll 1$.
\end{corollary}
\[\begin{tikzcd}
	& F & {(X,D)} & {(X_{\eta},D_{\eta})} \\
	{\mathscr{F}} & {(\mathscr{X},\mathscr{D})} \\
	{S} & {\mathscr{C}} \\
	& {\Spec R} & \eta
	\arrow["\subseteq"{description}, draw=none, from=1-2, to=1-3]
	\arrow[from=1-3, to=4-2]
	\arrow["\supseteq"{description}, draw=none, from=1-4, to=1-3]
	\arrow[from=1-4, to=4-3]
	\arrow["{\textup{cms}}"{description}, from=2-1, to=1-2]
	\arrow["\subseteq"{description}, draw=none, from=2-1, to=2-2]
	\arrow[from=2-1, to=3-1]
	\arrow["{\textup{cms}}"{description}, from=2-2, to=1-3]
	\arrow["\pi"', from=2-2, to=3-2]
	\arrow["\subseteq"{description}, draw=none, from=3-2, to=3-1]
	\arrow[from=3-2, to=4-2]
	\arrow["\supseteq"{description}, draw=none, from=4-3, to=4-2]
\end{tikzcd}\]
\begin{proof}
    As the fibration $\pi_\sX$ is locally stable and the base $(\sC,(c_S +\epsilon')S + \sC_0)$ is slc, where $\sC_0\subseteq \sC$ is the central fiber of $\sC\to \spec R $, the pair $(X,cD + (c_S +\epsilon')F+X_0)$ is slc from \cite[Lemma 2.12]{patakfalvi2016fibered} and the fact that finite quotients of an slc pair is slc. As the coefficients of the pairs $(X,(c+\epsilon){D} + c_S F + X_0)$ are a convex combination of the coefficients of $(X,cD + (c_S +\epsilon')F+X_0)$ and $(X,(c+\epsilon''){D} + X_0)$, which are slc, also the pair $(X,(c+\epsilon)D + c_SF + X_0)$ is slc for $0<\epsilon \ll 1$. The moreover part follows as in \Cref{lemma_qmap_limit_plus_divisor_has_lc_sing}.
\end{proof}

The following lemma is an useful observation which we will use the describe some of the objects on the boundary of $\cQ$.
    
\begin{lemma}\label{lemma_nodal_locus_for_qmap_limit_is_union_of_irred_cp_of_special_fiber}
    Let $(\sX,\frac{1}{n+1}\sD)\xrightarrow{\pi} \sC\to \spec R $ be as in Lemma \ref{lemma_qmap_limit_plus_divisor_has_lc_sing}, and assume that for each point $p\in \sC_\eta$, the fiber $\pi^{-1}(p)$ is smooth. Then the sublocus in $\sC$ over which the $\pi$-fibers are strictly nodal consists of a finite union of irreducible components of the special fiber of $\sC\to \spec R $.
\end{lemma}
\begin{proof}
    From the structure of the versal deformation space of a nodal singularity, the nodal locus is of codimension one in $\sC$. As the fibers over $\sC_\eta$ are smooth, it must be a finite union of irreducible components of the special fiber of $\sC\to \spec R $.
\end{proof}
   
\subsubsection{The ruled-model}
From \Cref{section_prelim_on_qmaps} there is a retraction morphism $$\xi\colon \ \sP_n\ \longrightarrow \  \sP_n^{\GIT}$$ which, at the level of points, is given by $$\textstyle \big[(C,\frac{1}{n+1}D)\big]\ \ \mapsto \ \ \begin{cases}
    \big[(C,\frac{1}{n+1}D)\big] \ \  
 & \textup{if }C\simeq \bP^1 \\
     &  \\
    \big[(\bP^1,(n+1)[0] + (n+1)[\infty])\big] \ \  & \textup{if }C\textup{ is nodal}.
\end{cases}$$ We will use this morphism to define the \textit{ruled model} of a surface pair $(X,\frac{1}{n+1}D)\to C$ over a curve $C$. The ruled model will be the main tool which will allow us to control the KSBA-stable limits of certain ruled surfaces as the beginning of this section, without running explicitly any MMP.

\begin{defn}\label{defn:ruled model} (ref. Figure \ref{fig:illustration of ruled models})
    Let $C\to B$ be a family of nodal curves, and let $\pi\colon(X,\frac{1}{n+1}D)\to C$ be a projective morphism (not necessarily flat) such that the composition $(X,\frac{1}{n+1}D)\to B$ is a locally stable family of surface pairs. Assume that $U\subseteq C^{\sm}$ is a dense open subset such that the restriction morphism $(X_U,\frac{1}{n+1}D_U)\xrightarrow{\pi|_U} U$ is the coarse space of the family pulled back along a morphism $U\to \sP_n$. A family of pairs $\pi_\sY\colon(\sY,\frac{1}{n+1}D_\sY)\to \sC$ is called a \textit{ruled model} for $\pi$ if:
    \begin{enumerate}
        \item $\sC\to B$ is a family of twisted curves with coarse moduli space $C$ over $B$;
        %\item the coarse moduli space of $\sC\to B$ is $C\to B$; %and the morphism $\sC\rightarrow C$ is an isomorphism over $U$
        \item the fibers of $\pi_\sY$ are isomorphic to $\bP^1$;
        \item there is a neighborhood $V\subseteq C$ of each node of the fibers of $C\to B$ and a map $\mts{V}:=V\times_C\sC\to \sP_n$ such that $(X|_{V\cup U},\frac{1}{n+1}D|_{V\cup U})$ is the coarse space of the family coming from $\mathscr{V}\cup U\to \sP_n$, and $(\sY|_{U\cup \sV},(\frac{1}{n+1}D_{\sY})|_{U\cup \sV})$ comes from taking the composition ${U\cup \mathscr{V}}\to \sP_n\to \sP_n^{\GIT}$ where the map $\sP_n\to\sP_n^{\GIT}$ is recalled in \Cref{item_where_we_say_there_is_a_map_spn_to_tilde_pn},
        \item if $\sC_0$ is an irreducible component of a fiber of $\sC\rightarrow B$ with geometric generic point $\ove{\eta}$ such that
        $(X_{\ove{\eta}},(\frac{1}{n+1}+\epsilon)D_{\ove{\eta}})$ is not klt, then
        $$\textstyle (\sY_{\ove{\eta}},D_{\sY_{\ove{\eta}}})\ \simeq \ (\bP^1,(n+1)[0]+(n+1)[\infty]);$$
        \item if the generic fiber of $X\to B$ is smooth, then the coarse moduli space of $(\sY,\frac{1}{n+1}D_\sY)$ and $(X,\frac{1}{n+1}D)$ are crepant birational.
    \end{enumerate}
\end{defn}

\begin{example}\textup{
    We introduce this example to help the reader navigate how the fibers of $X\to C$ and $\sY\to \sC$ differ. If $B=\spec(k)$ is a point, a ruled model of $(X,\frac{1}{n+1}D)$ is obtained as follows.
\begin{enumerate}
\item Let $p\in C$ be a smooth point of $C$, and $\eta$ be the generic point of the irreducible component of $C$ containing $p$.
\begin{enumerate}
    \item[(a)] If $\big(X_{\bar{\eta}},(\frac{1}{n+1}+\epsilon)D_{\bar{\eta}}\big)$ is \emph{not} klt, then
        $(\sY_p,(D_\sY)|_p)\cong(\bP^1,n[0]+n[\infty])$.
        \item[(b)] If $\big(X_{\bar{\eta}},(\frac{1}{n+1}+\epsilon)D_{\bar{\eta}}\big)$ is klt, and $(X_p,\frac{1}{n+1}D_p)$ is the coarse space of a pair $\phi(p)\in \sP_n$, then
       the fiber of the corresponding ruled model is the point $\xi(\phi(p))$, where $\xi\colon\sP_n\to \sP_n^{\GIT}$ is the retraction morphism.
\end{enumerate}
    \item Let $p\in C$ be a node of $C$. Then the surface pair $(X,\frac{1}{n+1}D)$ in a neighborhood of $p$ is the coarse moduli space of the surface pair coming from a map $\sV\to \sP_n$, and the ruled model replaces the fibers with those coming from the composition $\sV\to \sP_n\to \sP_n^{\GIT}$.     
\end{enumerate}
}\end{example}

\begin{figure}
    \centering
    \includegraphics[width=0.7\linewidth]{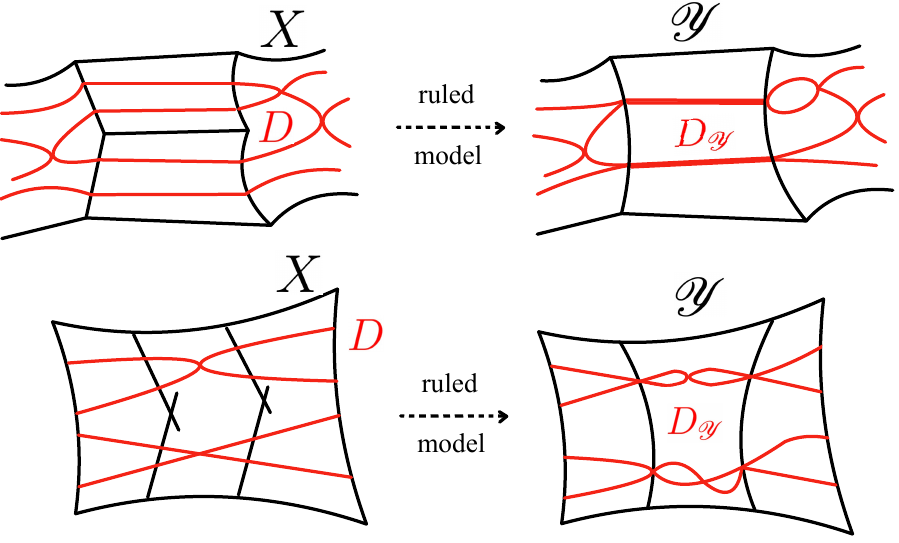}
    \caption{Two typical ruled models}
    \label{fig:illustration of ruled models}
\end{figure}

Another easy example of a surface pair admitting a ruled model is as follows.

\begin{example}\label{rmk_for_qmaps_there_is_a_ruled_model}\textup{
    Let $\pi\colon (\sX,\frac{1}{n+1}\sD)\to \sC$ coming from the pull-back along a (pointed) quasi-map $\phi\colon\sC\to \sP_n$, and $(X,\frac{1}{n+1}D)\to C$ be the coarse moduli spaces. Then one can post-compose $\phi$ with $\sP_n\to \sP_n^{\GIT}$, so that the corresponding family  $(\sY,\frac{1}{n+1}\sD_\sY)\to \sC$ is a ruled model $(X,\frac{1}{n+1}D)\to C$.}
\end{example}

The goal of this subsection is to prove that given a threefold pair $(X,\frac{1}{n+1}D)$ which is the coarse space of the threefold pair corresponding to a morphism $\spec R \to \cQ$, one can run a special MMP on $(X,\frac{1}{n+1}D)$, most of the steps of which are explicit. The setup will be the following.

\begin{setup}\label{setup_for_existence_of_ruled_model}
\textup{Let $R$ be a DVR, and let $\eta$ (resp. $0$) be the generic (resp. closed) point of $\spec R$. Let $$\textstyle (\sX,\frac{1}{n+1}\sD) \ \stackrel{\pi_\sX}{\longrightarrow} \ (\sC,x_1,\ldots,x_m) \ \longrightarrow \ \spec R $$ be a family of $m$-pointed stable quasimaps over $\spec R$. Assume that $\sC_\eta=C_\eta$ is smooth and that $\zeta_\eta\colon \sC_\eta\to \sP_n$ factors via $C_\eta\to \sP_n^{\GIT}\to \sP_n$ (i.e., the fibers of $\pi_\eta$ are smooth). We denote by $\pi_X\colon (X,D)\rightarrow C$ the morphism on coarse spaces induced by $\pi_{\sX}$.}

\[\begin{tikzcd}
	{(X,D)} & {(\mathscr{X},\mathscr{D})} & {(\mathscr{X}_{\eta},\mathscr{D}_{\eta})} \\
	& C & {\mathscr{C}} & {\mathscr{C}_{\eta}} & {\sP_n^{\GIT}} & {\mathscr{P}_n} \\
	&& {\Spec  R } & {\eta}
	\arrow["{\pi_X}"{description}, from=1-1, to=2-2]
    \arrow[ from=2-2, to=3-3]
	\arrow[from=1-2, to=1-1]
	\arrow["{\pi_{\mathscr{X}}}"{description}, from=1-2, to=2-3]
	\arrow["\supseteq"{description}, draw=none, from=1-3, to=1-2]
	\arrow["{\pi_{\eta}}"{description}, from=1-3, to=2-4]
	\arrow[from=2-3, to=2-2]
	\arrow["\zeta"{description}, shift left, bend left = 32pt, from=2-3, to=2-6]
	\arrow[from=2-3, to=3-3]
	\arrow["\supseteq"{description}, draw=none, from=2-4, to=2-3]
	\arrow[from=2-4, to=2-5]
	\arrow["{\zeta_{\eta}}"{description}, bend left = 28pt, from=2-4, to=2-6]
	\arrow["\xi", from=2-5, to=2-6]
	\arrow["\ni"{description}, draw=none, from=3-4, to=3-3]
    \arrow[ from=2-4, to=3-4]
\end{tikzcd}\]
\end{setup}
\noindent Let $S\subseteq \sC$ be a closed subscheme in the smooth locus of $\sC\to \spec R $ such that:
\begin{enumerate}
    \item it contains $x_1,\ldots,x_m$, and
    \item the pair $(\sC,(c+\epsilon')S)$ is log canonical for some $c\in \bQ\cap (0,1]$ and $0<\epsilon'\ll 1$.
\end{enumerate}
 Let $\sF:=\pi^{-1}_\sX(S)$. We finally assume that the pair $$\textstyle \big(\sX_\eta,(\frac{1}{n+1}+\epsilon)\sD + c\sF\big)$$ is KSBA-stable for $0<\epsilon\ll1$. Then denoting by $X^{(0)}$ (resp. $C^{(0)}$, $F^{(0)}$, $D^{(0)}$ and $S^{(0)}$) the coarse space of $\sX$ (resp. $\sC$, $\sF$, $\sD$ and $S$), we have that for $0<\epsilon \ll 1$  \begin{enumerate}
        \item the pair $\big(X^{(0)},(\frac{1}{n+1}+\epsilon)D^{(0)}+cF^{(0)}\big)$ is klt by Lemma \ref{cor_singularities_remain_slc_if_increasing_coeff_on_d_and_adding_fibers};
        \item the divisor $D^{(0)}$ is ample over $C^{(0)}$;
        \item up to a base change of $\spec R $, we can take a $\bQ$-divisor $H$ on $C^{(0)}$ such that that the pairs \[\textstyle \big(C^{(0)},c S^{(0)} + H^{(0)}\big) \  \ \ \ \text{ and } \ \ \ \big(X^{(0)},(\frac{1}{n+1}+\epsilon)D^{(0)}+cF^{(0)} + (\pi^{(0)})^{-1}(H^{(0)})\big)\] are KSBA-stable and klt, where $\pi^{(0)}:X^{(0)}\rightarrow C^{(0)}$ agrees with $\pi_X:X\rightarrow C$; and
        \item there is a semiample divisor $\bfM^{(0)}$ on $C^{(0)}$ so that $(C^{(0)},cS^{(0)}+\bfM^{(0)} + H^{(0)})$ is a generalized pair with $K_{C^{(0)}}+c S^{(0)}+\bfM^{(0)} + H^{(0)}$ ample and \[\textstyle (\pi^{(0)})^*(K_{C^{(0)}}+cS^{(0)}+\bfM^{(0)} + H^{(0)}) \  = \ K_{X^{(0)}}+\frac{1}{n+1}D^{(0)}+cF^{(0)} + (\pi^{(0)})^{-1}(H^{(0)}).\] 
    \end{enumerate}
    We are now in the setting of Subsection \ref{section_mmp}: running an MMP for $K_{C^{(0)}}+M^{(0)}+c S^{(0)}$ with scaling by $H^{(0)}$ can be followed by a sequence of birational contractions
    for $(X^{(0)},\frac{1}{n+1}D^{(0)}+cF^{(0)})$, so that the proper transform $D^{(i)}$ of $D^{(0)}$
    remains ample over $C^{(i)}$ and $(X^{(k)},(\frac{1}{n+1}+\epsilon)D^{(k)}+cF^{(k)})$ is a weak canonical model for $(X^{(0)},(\frac{1}{n+1}+\epsilon)D^{(0)}+cF^{(0)})$, for $0<\epsilon \ll 1$. The following diagram summarizes our situation
\begin{equation}\label{eq: diagram of MMP}   
\begin{tikzcd}
	{\big(X^{(0)},\frac{D^{(0)}}{n+1}+cF^{(0)}\big)} & {\big(X^{(1)},\frac{D^{(1)}}{n+1}+cF^{(1)}\big)} & \cdots & {\big(X^{(k)},\frac{D^{(k)}}{n+1}+cF^{(k)}\big)} \\
	{\big(C^{(0)},c S^{(0)}+\bfM^{(0)}\big)} & {\big(C^{(1)},c S^{(1)}+\bfM^{(1)}\big)} & \cdots & {\big(C^{(k)},c S^{(k)}+\bfM^{(k)}\big),}
	\arrow["{q^{(0)}}", dotted, from=1-1, to=1-2]
	\arrow["{\pi^{(0)}}", from=1-1, to=2-1]
	\arrow["{q^{(1)}}", dotted, from=1-2, to=1-3]
	\arrow["{\pi^{(1)}}", dotted, from=1-2, to=2-2]
	\arrow["{q^{(k-1)}}", dotted, from=1-3, to=1-4]
	\arrow[dotted, from=1-3, to=2-3]
	\arrow["{\pi^{(k)}}", dotted, from=1-4, to=2-4]
	\arrow["{p^{(0)}}", from=2-1, to=2-2]
	\arrow["{p^{(1)}}", from=2-2, to=2-3]
	\arrow["{p^{(k-1)}}", from=2-3, to=2-4]
\end{tikzcd}\end{equation} where the existence of morphisms $q^{(i)}$ and $\pi^{(i)}$ is guaranteed by Lemma \ref{prop_can_run_mmp_upstairs_to_follow_the_one_downstairs}, and $K_{C^{(k)}}+cS^{(k)}+\bfM^{(k)}$ is nef. We remark here that the bottom arrows are all solid as $C^{(0)}$ is a surface and hence no flips occur.

\begin{corollary}\label{cor_if_lc_div_on_the_curve_is_ample_its_pull_back_plus_epsilon+D_is_ample}
    With the same notations as above. If $K_{C^{(k)}} +  cS^{(k)}+\bfM^{(i)}$ is ample, then the KSBA-stable limit of $\big(X_\eta,(\frac{1}{n+1} +\epsilon)D_\eta+cF_\eta\big)$ is independent of $0<\epsilon\ll 1$.
\end{corollary}
\begin{proof}
    From Remark \ref{rmk_during_MMP_pull_back_of_lc_divisor_is_lc_divispr} we have that for every $i$,
    \[\textstyle (\pi^{(i)})^*(K_{C^{(i)}}+c S^{(i)}+\bfM^{(i)}) = K_{X^{(i)}}+\frac{1}{n+1}D^{(i)}+cF^{(i)}\] and  $D^{(i)}$ is $\pi^{(i)}$-ample by Proposition \ref{prop_can_run_mmp_upstairs_to_follow_the_one_downstairs}. In particular, if $K_{C^{(k)}} +  cS^{(k)}+\bfM^{(k)}$ is ample, then the pair $(X^{(k)},(\frac{1}{n+1} +\epsilon)D^{(k)}+cF^{(k)})$ is KSBA-stable for any $0<\epsilon\ll 1$, and it is the stable limit for $(X_\eta,(\frac{1}{n+1} +\epsilon)D_\eta+cF_\eta)$. So if $K_{C^{(k)}} +  cS^{(k)}+\bfM^{(k)}$ is ample, the stable limit of $(X,(\frac{1}{n+1} +\epsilon)D+cF)$ remains the same for every $0<\epsilon\ll1$.
\end{proof}
    
The next theorem describes each pair $(X^{(i)},\frac{1}{n+1}D^{(i)})$. Instead of studying the MMP for $(X^{(0)},\frac{1}{n+1}D^{(0)})$, what we use are ruled models $\gamma^{(i)}\colon (\sY^{(i)},\frac{1}{n+1}\sD_\sY^{(i)})\to \sC^{(i)}$, whose existence is justified in Theorem \ref{thm_you_can_take_ruled_model_in_ksba_moduli}. Denote by $g^{(i)}:(Y^{(i)},D_Y^{(i)})\rightarrow C^{(i)}$ the coarse spaces of $\gamma^{(i)}$. There could be some components of $C^{(i)}_{0}$ over which the geometric generic $g^{(i)}$-fiber is isomorphic to $$\textstyle \big(\bP^1,(n+1)[0]+(n+1)[\infty]\big).$$ We denote by $C^{(i)}_{0,1},...,C^{(i)}_{0,k}$ all these components, and by $E_1,...,E_k$ the preimages of them under $g^{(i)}$, as in Figure \ref{fig:ruled model}.

    \begin{theorem}[ref. Figure \ref{fig:ruled model}]\label{thm_you_can_take_ruled_model_in_ksba_moduli}
        With the same notations and assumptions as in \Cref{setup_for_existence_of_ruled_model}, the pair $(X^{(i)},\frac{1}{n+1}D^{(i)})$ admits a ruled model $(\sY^{(i)},\frac{1}{n+1}\sD_\sY^{(i)})\to \sC^{(i)}$. Moreover, the threefold $X^{(i)}$ is obtained from $Y^{(i)}$ by:
    \begin{enumerate}
        \item first extracting a sequence of divisors\footnote{see more details in the Proof of \Cref{prop: geometry of S3}} supported over the curves of $Y^{(i)}_{0}$ where $\frac{1}{n+1}D_Y^{(i)}$ has coefficient $1$ and which are contained in some irreducible components $E_1,\ldots,E_k$ of the central fiber of $Y\to \spec R$ via a morphism $Z^{(i)}\to Y^{(i)}$; and
        \item then contracting the proper transforms of $E_1,\ldots,E_k$ via a morphism $Z^{(i)}\to X^{(i)}$.
    \end{enumerate} 
   The pair $(X^{(i)},\frac{1}{n+1}D^{(i)})$ has discrepancy $0$ with respect to the divisors $E_j$ extracted in (1). Moreover, for every smooth point $p\in \sC^{(i)}_0$ of the central fiber $\sC^{(i)}_0\subseteq \sC^{(i)}$, the degree of the boundary part for $(\sY^{(i)},\frac{1}{n+1}\sD_\sY^{(i)})|_{\sY^{(i)}_0}\to \sC^{(i)}_0$ is strictly less than 1.
    \end{theorem}

\begin{figure}
    \centering
    \includegraphics[width=0.8\linewidth]{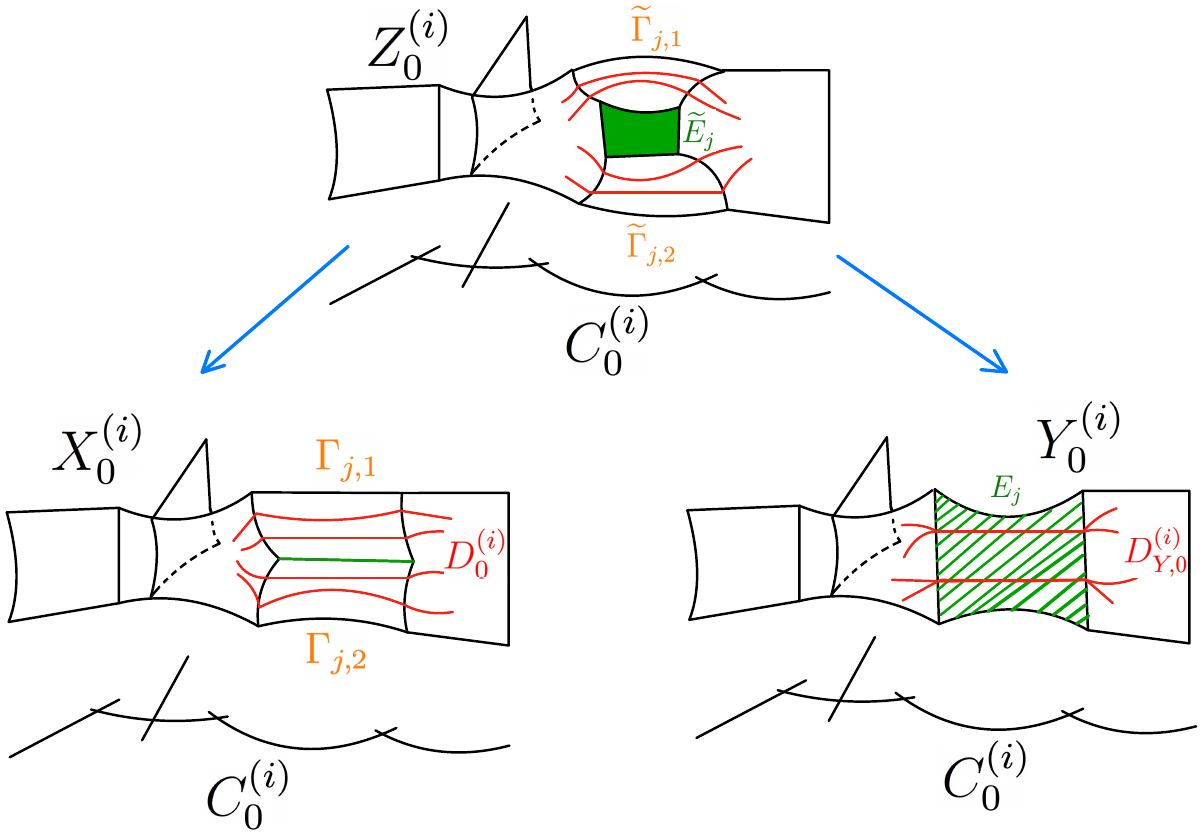}
    \caption{Interpolation with ruled models: fibers over $0\in \Spec R$. \\ The map $Z^{(i)}\to X^{(i)}$ extracts the divisor $\wt{E}_j$ and $Z^{(i)}\to X^{(i)}$ extracts $\widetilde{\Gamma}_{j,i}$.}
    \label{fig:ruled model}
\end{figure}

\begin{figure}
    \centering
    \includegraphics[width=0.7\linewidth]{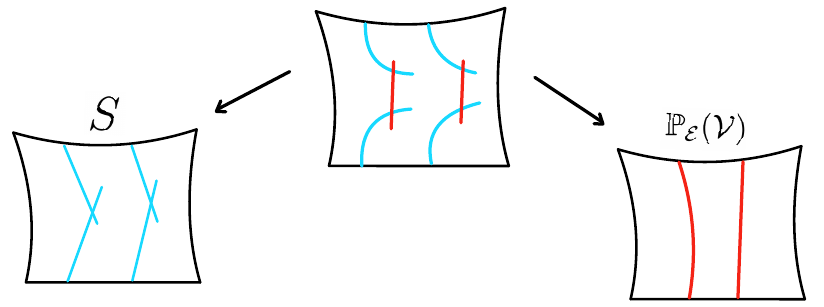}
    \caption{Interpolation with ruled models for normal surfaces}
    \label{fig:ruled model2}
\end{figure}
    
\begin{proof}
    We begin by observing that $p^{(i)}$ will only contract rational tails of the special fiber of $C^{(i)}\to \spec R $, as if $B\subseteq C^{(i)}$ is not a rational tail, one has $$0 \ \leq\  \deg(\omega_{C^{(i)}})|_B \ \le\  \deg(\omega_{C^{(i)}})|_B + (cS^{(i)}+\mathbf{M}^{(i)}).B.$$ 

    We prove the desired result by induction. As $c$ will play no role, to simplify the notations we assume that $c=0$. The initial morphism $\pi^{(0)}$ admits a ruled model by Example \ref{rmk_for_qmaps_there_is_a_ruled_model}. Assuming that there is a ruled model $$\textstyle (\sY^{(m-1)},\frac{1}{n+1}\sD^{(m-1)})\to \sC^{(m-1)}$$ for $\pi^{(m-1)}$, we now show that there is a ruled model for $\pi^{(m)}$. As each step of the MMP is the contraction of a rational tail of the special fiber of $C^{(i)}\to\spec R $, there is a \textit{smooth}
    point $x\in C^{(m)}$ so that $p^{(m-1)}$ is an isomorphism on $C^{(m)}\smallsetminus \{x\}$.
    By the induction hypothesis, it suffices to extend the ruled model over $C^{(m)}\smallsetminus \{x\}$ to the fiber over $x$. The open embedding $C^{(m)}\smallsetminus \{x\}\to C^{(m-1)}$ yields a twisted curve $\sC^{(m)}$ with coarse moduli space $C^{(m)}$, which admits an open embedding $\sC^{(m)}\smallsetminus\{x\}\to \sC^{(m-1)}$.
    The ruled model on $\sC^{(m)}\smallsetminus\{x\}$ gives a $\bP^1$-fibration, which is equivalent to a morphism $$\sC^{(m)}\smallsetminus \{x\} \ \longrightarrow \ \cB \PGL_2.$$ By \cite[Lemma 2.1]{twisted_map_2}, as $x$ is a smooth point, this extends to a map $\sC^{(m)}\to \cB \PGL_2$. In particular, there is a $\bP^1$-bundle $$r^{(m)} \ :\ \sY^{(m)}\ \longrightarrow \  \sC^{(m)}$$
    extending the one on $\sC^{(m)}\smallsetminus\{x\}$ coming from the ruled model over $\sC^{(m-1)}$. The divisor $\sD^{(m-1)}$ on $\sY^{(m-1)}$ gives a divisor on $\sY^{(m)}\smallsetminus \{x\}$ whose closure we denote by $\sD^{(m)}$. This is a Cartier divisor on a neighborhood of $(r^{(m)})^{-1}(x)$ as $\sY^{(m)}$ is smooth around $(r^{(m)})^{-1}(x)$. One can check that $(\sY^{(m)},\frac{1}{n+1}\sD^{(m)})\to \sC^m$ is indeed a ruled model. By induction it is a ruled model away from the fiber over $x$, and it is a $\bP^1$-bundle over $x$.
    So it is a ruled model. Indeed, point (3) of \Cref{defn:ruled model} is automatic as $\sY^{(m)}$ agrees with $\sY^{(m-1)}$ on the nodal locus of $\sC^{(m)}$ and the latter is a ruled model. Point (4) holds as it holds for $\sY^{(m-1)}$ and $\sY^{(m-1)}$ agrees with $\sY^{(m)}$ over the generic points of the special fiber of $\sC^{(m)}$. Finally point (5) follows from the canonical bundle formula: let $\bfM^{(m)}$ be the moduli part for $(X^{(m)},\frac{1}{n+1}D^{(m)})\to \sC^{(m)}$. Then it suffices to check that $K_{\sY^{(m)}}+\frac{1}{n+1}\sD^{(m)}$ and $\pi_{\sY^{(m)}}^*(K_{\sC^{(m)}}+\bfM^{(m)})$ agree in codimension one, which is true as it is true by indiction at the step $m-1$. 

    We prove now the moreover part. For $m=0$, it follows from how the retraction morphism $\sP_n\to \sP_n^{\GIT}$ is constructed in \cite[Theorem A.4(1)]{twisted_map_2}, which is recalled in \S\ref{subsubsection_recalling_how_to_take_ruled_model_in_qmap_paper}. In particular, one can go from $Y^{(0)}$ to $X^{(0)}$ by first performing a blowup $Z^{(0)}\to Y^{(0)}$ as in the statement of the theorem, followed by
    a contraction $Z^{(0)}\to X^{(0)}$ of the proper transform in $Z^{(0)}$ of the irreducible component of the special fiber of $Y^{(0)}$ which intersects the locus that was blown-up via $Z^{(0)}\to Y^{(0)}$ in codimension one, as in \Cref{fig:ruled model}.
    
    The resulting pair is $(Y^{(0)},\frac{1}{n+1}D_{Y}^{(0)})$: the coarse moduli space of the ruled model $\sY^{(0)}$. If we denote by $E_1,\ldots,E_k$ the divisors extracted by $Z^{(0)}\to X^{(0)}$ and by  $\Gamma_1,\ldots,\Gamma_{k'}$ the one extracted by $Z^{(0)}\to Y^{(0)}$, one can check that the pairs $(X^{(0)},\frac{1}{n+1}D^{(0)})$ and $(Y^{(0)},\frac{1}{n+1}D_{Y}^{(0)})$ are crepant birational. In particular, the divisors $E_i$ and $\Gamma_i$ have discrepancy 0 for both one of them. Again from how the map $\sP_n\to \sP_n^{\GIT}$ is constructed, for each divisor $E_j$ we can associate a set of at most two divisors $\{\Gamma_{j,1},\Gamma_{j,2}\}$, with possibly $\Gamma_{j,1}=\Gamma_{j,2}$\footnote{Indeed, $\Gamma_{j,1}$ and $\Gamma_{j,2}$ are the irreducible components of $X^{(0)}_0$ which map to a component $G\subseteq C^{(0)}_0$, such that the \textit{geometric} fiber of $\pi^{(0)}$ of the generic point $\eta_G$ of $G$ is non-normal. If the $\Gal\big(k(\overline{\eta}_G)/k(\eta_G)\big)$ acts by swapping the two irreducible components of $(\pi^{(0)})^{-1}(\overline{\eta}_G)$, then $(\pi^{(0)})^{-1}(\eta_G)$ has a single component.} such that $E_j$, $\Gamma_{j,1}$ and $\Gamma_{j,2}$ map to the same irreducible component of $C^{(0)}$.

    For the inductive step we proceed as follows. Let $G$ be an irreducible component of the central fiber of $C^{(i)}$ over which $C^{(i-1)}\to C^{(i)}$ is not an isomorphism. In other terms, the map  $C^{(i-1)}\to C^{(i)}$ contracts a tail $T$, and $G$ is the irreducible component to which it is attached. If $X^{(i-1)}$ and $Y^{(i-1)}$ are isomorphic over the generic point of $G$ there is nothing to prove: they are isomorphic over $G$ as $D^{(i)}$ and $D^{(i)}_Y$ are both ample over $C^{(i)}$, so $X^{(i-1)}$ and $Y^{(i-1)}$ are the projectivization of the same algebra in a neighbourhood of $G$. We now assume they are not isomorphic over the generic point of $G$. 
    
    The proof proceeds by first extracting the divisors $\{\Gamma_{j,\ell}\}_\ell$ via a morphism $Z^{(i)}\to Y^{(i)}$ in Step 1. We will then study and refine the space $Z^{(i)}$ in Steps 2 and 3. We will finally argue that there is a morphism $Z^{(i)}\to X^{(i)}$ which is as in the statement of the theorem in Steps 4 and 5. 
    
    \textbf{Step 1}. We extract the divisors $\{\Gamma_{j,\ell}\}_\ell$ which map to the irreducible component of $Y^{(i)}$ mapping surjectively to $G$ via $Z^{(i)}\to Y^{(i)}$, and we study $Z^{(i)}$.
    
    From \cite[Corollary 1]{moraga2020extracting} one can extract the divisors $\{\Gamma_{j,\ell}\}_\ell$ which map to the irreducible component of $Y^{(i)}$ mapping surjectively to $G$, via a morphism $Z^{(i)}\to Y^{(i)}$. The main observation for Step 1 and 2 will be that $(Y^{(0)},\frac{1}{n+1}D_Y^{(0)})$ and $(Y^{(i)},\frac{1}{n+1}D_Y^{(i)})$ are isomorphic over the generic point of $G$, and the fact that we can extract $\{\Gamma_{j,\ell}\}_\ell$ over $Y^{(0)}$ explicitly, as in \S \ref{subsubsection_recalling_how_to_take_ruled_model_in_qmap_paper}.

     Observe first that each irreducible component of $\Gamma_{j,\ell}$ maps surjectively to $G$, as in \Cref{fig:ruled model}.
     As the divisors $\Gamma_{j,\ell}$ have discrepancy 0 for $(Y^{(0)},\frac{1}{n+1}D_Y^{(0)})$ and as $(Y^{(0)},\frac{1}{n+1}D_Y^{(0)})$ and $(Y^{(i)},\frac{1}{n+1}D_Y^{(i)})$ are isomorphic over the generic point of $G$, we have:
     \begin{itemize}
         \item the divisors $\widetilde{\Gamma}_{j,\ell}$ have discrepancy 0 for $(Y^{(i)},\frac{1}{n+1}D_Y^{(i)})$ so the map 
     \[\textstyle
     \psi\colon \big(Z^{(i)},\frac{1}{n+1}D_Z^{(i)}\big)\to \big(Y^{(i)},\frac{1}{n+1}D_Y^{(i)}\big)
     \]
     is crepant birational, where $\frac{1}{n+1}D_Z^{(i)}$ (resp. $\widetilde{\Gamma}_{j,\ell}$) is the proper transform of $\frac{1}{n+1}D_Y^{(i)}$ (resp. $\Gamma_{j,\ell}$) and
     \item the (at most two) divisors $\{\widetilde{\Gamma}_{j,\ell}\}_\ell$ contracted by $Z^{(i)}\to Y^{(i)}$ have as centers curves in $D_Y^{(i)}$, so the non-empty fibers of $\widetilde{\Gamma}_{j,\ell}\to D_Y^{(i)}$ are one dimensional.
     \end{itemize}

     \textbf{Step 2}. We study the pair $(Z^{(i)},D_Z^{(i)})$ over the generic point $\eta_G\in G$. More explicitly, we prove that $D_Z^{(i)}$ is ample for $Z^{(i)}\to Y^{(i)}$ over $\eta_G$.
     
     Indeed, let $A:=\cO_{C^{(0)},\eta_G}$ and let $Z^{(0)}\to Y^{(0)}$ be the space as in the statement of the theorem, obtained using \S \ref{subsubsection_recalling_how_to_take_ruled_model_in_qmap_paper}. Observe that $A$ is a DVR. Let $(Z^{(0)}_A,\frac{1}{n+1}D_{Z,A}^{(0)})$ (resp. $(Z^{(i)}_A,\frac{1}{n+1}D_{Z,A}^{(i)})$) be the restriction of $(Z^{(0)},\frac{1}{n+1}D_{Z}^{(0)})$ (resp. $(Z^{(i)}_A,\frac{1}{n+1}D_{Z}^{(i)})$) to $\spec A$. Then the two pairs \[\textstyle\left(Z^{(0)}_A,\frac{1}{n+1}D_{Z,A}^{(0)}\right) \ \ \ \text{ and } \ \ \ \left(Z^{(i)}_A,\frac{1}{n+1}D_{Z,A}^{(i)}\right)\] are of dimension two, normal and isomorphic in codimension one as we extracted the same divisor from $(Y^{(0)},\frac{1}{n+1}D_Y^{(0)})|_{\spec A}\cong(Y^{(i)},\frac{1}{n+1}D_Y^{(i)})|_{\spec A}$. Then they are isomorphic. As the desired statement holds for $(Z^{(0)}_A,\frac{1}{n+1}D_{Z,A}^{(0)})$ by the explicit construction of $Z^{(0)}\to Y^{(0)}$, it holds for $(Z^{(i)}_A,\frac{1}{n+1}D_{Z,A}^{(i)})$.

     \textbf{Step 3}. We improve the pair $(Z^{(i)},D_Z^{(i)})$ over the nodal locus of $C^{(i)}_0$.

     Up to replacing $(Z^{(i)},(\frac{1}{n+1} +\epsilon)D_Z^{(i)})$ with its canonical model over $Y^{(i)}$, we can also assume that $D_Z^{(i)}$ is ample over $Y^{(i)}$. Observe that such a canonical model exists as $(Z^{(i)},(\frac{1}{n+1} +\epsilon)D_Z^{(i)})$ is klt, and it does not contract the divisors $\widetilde{\Gamma}_{j,\ell}$ from Step 2. We now describe the threefold $Z^{(i)}$ over the nodal locus of $C^{(i)}$.

 By \S \ref{subsubsection_recalling_how_to_take_ruled_model_in_qmap_paper}, there is $Z^{(0)}\to Y^{(0)}$ and $Z^{(0)}\to X^{(0)}$ as in the statement of the theorem. Let $U_X\subseteq X^{(0)}$ (resp. $U_Y\subseteq Y^{(0)}$) the locus where $X^{(0)}\cong X^{(i)}$ (resp. $Y^{(0)}\cong Y^{(i)}$). Observe that this locus includes the preimage of the nodal locus of $C^{(i)}$,
 as our MMP contracts only tails of $C^{(0)}$. Let $U_{Z}^{(0)}$ the preimage of $U_Y$ in $Z^{(0)}$, and $U_{Z}^{(i)}$ its preimage in $Z^{(i)}$. We summarize the notation
 \[
 X^{(0)}\supseteq U_X \subseteq X^{(i)},\text{ } Y^{(0)}\supseteq U_Y \subseteq Y^{(i)},\text{ }Z^{(0)}\supseteq U_Z^{(0)}, \text{ and } U_Z^{(i)} \subseteq Z^{(i)}.
 \]We now argue that also $U_Z^{(0)}\cong U_Z^{(i)}$.
 As we extracted the divisors $\Gamma_{j,\ell}$\[ (U_{Z}^{(i)}, D_Z^{(i)}|_{U_Z^{(i)}})\text{ agrees with }(U_{Z}^{(0)}, D_Z^{(0)}|_{U_Z^{(0)}})\text{ in codimension one.}\] Moreover, the divisors $D_Z^{(i)}|_{U_{Z}^{(i)}}$
 and $D_Z^{(0)}|_{U_{Z}^{(0)}}$ are both ample over $U_Y$: the former one as we took a canonical model at the beginning of this step, and the latter by explicit construction of \S \ref{subsubsection_recalling_how_to_take_ruled_model_in_qmap_paper}. In particular, as they are both $S_2$, they are isomorphic over $U_Y$ as they are the projectivization of the same graded algebra.

     \textbf{Step 4}. We argue that, if the boundary part for $(X^{(i)}_0,\frac{1}{n+1}D^{(i)}_0)\to C^{(i)}_0$ over a smooth point for $C^{(i)}_0$ is less than one, the pair $(Z^{(i)},(\frac{1}{n+1} +\epsilon)D_Z^{(i)} + Z^{(i)}_0)$ is lc.

     Observe that, from how $Z^{(i)}$ is constructed in Steps 1 to 3 and from \S \ref{subsubsection_recalling_how_to_take_ruled_model_in_qmap_paper}, the pair $(Z^{(i)},(\frac{1}{n+1} +\epsilon)D_Z^{(i)} + Z^{(i)}_0)$ is lc away from the fiber over finitely many smooth points of $C^{(i)}$; let $p$ be one of those points and let $W$ be the irreducible component of $C^{(i)}$ containing $p$.
     If the pair $(Z^{(i)},(\frac{1}{n+1} +\epsilon)D_Z^{(i)} + Z^{(i)}_0)$ was not lc at the fiber over $p$, as $(Z^{(i)},\frac{1}{n+1}D_Z^{(i)} + Z^{(i)}_0)$ and $(X^{(i)},\frac{1}{n+1}D^{(i)}+X^{(i)}_0)$ are crepant birational, there is an lc center for $(X^{(i)},\frac{1}{n+1} D^{(i)} + X^{(i)}_0)$ over $p$. So by adjunction \cite[Theorem 4.9]{Kol13}, there is an lc center for $(X^{(i)},\frac{1}{n+1} D^{(i)})|_W$. To rule out this possibility, it suffices to show that the boundary part for the canonical bundle formula for $(X^{(i)},\frac{1}{n+1} D^{(i)})|_W\to W$ is strictly less than one over $p$.

     \textbf{Step 5}. We prove that the boundary part for $(X^{(i)}_0,\frac{1}{n+1}D^{(i)}_0)\to C^{(i)}_0$ over a smooth point for $C^{(i)}_0$ is less than one.

     The pair $(X^{(i)},\frac{1}{n+1}D^{(i)})$ is constructed by running an MMP with scaling for the generalized pair $(C^{(0)},\bfM^{(0)})$ and using \S\ref{section_mmp}.
     This MMP contracts certain chains of rational curves which are tails of the central fiber of $C^{(0)}$. In particular, let \[p\in W\subseteq C^{(i)}\] be a point $p$ contained in an irreducible component $ W\subseteq C^{(i)}_0$ of the central fiber $C^{(i)}_0$ to which a chain of rational curves $R\subseteq C^{(0)}$ is contracted via $C^{(0)}\to C^{(i)}$.
     As the degree of $\bfM^{(i)}$ remains invariant through this MMP,
     \[
     \deg(\bfM^{(i)}|_{W}) =\deg(\bfM^{(0)}|_{W^{(0)}\cup T}) 
     \]
     where we denoted by $W^{(0)}\subseteq C^{(0)}$ the proper transform of $W$ in $C^{(0)}$, and $T$ the tails of rational curves of $C^{(i)}$ which contract to $W$ via $C^{(0)}\to C^{(i)}$. From how the family $(X^{(0)},\frac{1}{n+1}D^{(0)})$ is constructed (i.e. from a map to $\sP_n$), we have that $\bfM^{(0)}|_{W^{(0)}}$ agrees with the moduli part
     of $(X^{(0)},\frac{1}{n+1}D^{(0)})|_{W^{(0)}})$. Then from the canonical bundle formula, the degree of the {boundary} part $p$ for $(X^{(i)},\frac{1}{n+1}D^{(i)})|_W\to W$ agrees with $\deg(\bfM^{(0)}|_{ T_p})$ where $T_p$ are the chain of rational curves of $C^{(0)}$ mapping to $p$, as in \Cref{fig:example of MMP}.    
     Then observe that the MMP never contracts tails which are chain of rational curves where $\deg(\bfM^{(i-1)}|_{ T})=1$, as in this case $\deg(K_{C^{(i-1)}}+\bfM^{(i-1)}|_{ T})$ would be zero so $K_C{(i-1)}+\bfM^{(i-1)}|_{ T}$ would be already nef. So the boundary part in the canonical bundle formula for $(X^{(i)},\frac{1}{n+1}D^{(i)})|_G$ at $p$ is less than 1.

\textbf{End of the argument}. We prove that the pair $\left(Z^{(i)},(\frac{1}{n+1} +\epsilon)D_Z^{(i)}\right)$ is minimal over $C^{(i)}$ for $0<\epsilon \ll 1$, and we finish the argument.

The fibers of $Z^{(i)}\to C^{(i)}$ are one-dimensional, as the divisors extracted map to horizontal curves in $Y^{(i)}$, and 
\[\textstyle
K_{Z^{(i)}}+\big(\frac{1}{n+1} +\epsilon\big)\sim_{\bQ,C^{(i)}} \epsilon D_Z^{(i)}.
\]
Any effective divisor on a (possibly singular) curve is nef so our pair is minimal over $C^{(i)}$.

%As the divisors $\widetilde{\Gamma}_{j,\ell}$ map to $D_Y^{(i)}$ we have \[\textstyle \psi^*\left(K_{Y^{(i)}}+(\frac{1}{n+1} +\epsilon) D_Y^{(i)}\right)\sim_\bQ K_{Z^{(i)}}+\left(\frac{1}{n+1} +\epsilon\right)D_Z^{(i)} +\sum_\ell c_\ell(\epsilon)\widetilde{\Gamma}_{j,\ell} \] for every $0<\epsilon$ and for a certain $0<c_\ell(\epsilon)$ depending on $\ell$ and $0<\epsilon$. In particular, \[ \epsilon D_Z^{(i)}   \ \sim_{\bQ,C^{(i)}} \ \psi^*\big(\epsilon D_Y^{(i)}\big) -\sum_j c_j(\epsilon)\widetilde{\Gamma}_{j,\ell}. \]Then $K_{Z^{(i)}}+\left(\frac{1}{n+1} +\epsilon\right)D_Z^{(i)}$ is nef over $Y^{(i)}$ as $-\widetilde{\Gamma}_{j,\ell}$ is nef over $Y^{(i)}$ from the negativity lemma \cite[Lemma 3.39]{KM98} and $ D_Y^{(i)}$ is ample over $C^{(i)}$. Then the pair $\left(Z^{(i)},(\frac{1}{n+1} +\epsilon)D_Z^{(i)}\right)$ is minimal over $C^{(i)}$ for $0<\epsilon \ll 1$.
We proved that the pair
     \[\textstyle \big(Z^{(i)},(\frac{1}{n+1} +\epsilon)D_Z^{(i)} + Z^{(i)}_0\big)\]
     is log canonical and minimal over $C^{(i)}$, so one can take its canonical model, which will be a morphism rather than a rational map. As the KSBA-moduli space is separated, the morphism which takes the canonical model $(Z^{(i)},(\frac{1}{n+1} +\epsilon)D_Z^{(i)})$ over $C^{(i)}$ agrees with the pair $(X^{(i)},(\frac{1}{n+1} +\epsilon)D^{(i)})$,  and
         the canonical model $Z^{(i)}\to X^{(i)}$ contracts $\widetilde{E}_j$ as desired.
\end{proof}

 \begin{corollary}
     With the notations of Theorem \ref{thm_you_can_take_ruled_model_in_ksba_moduli}, if $G\subseteq C_0^{(i)}$ is an irreducible component of the special fiber $C_0^{(i)}$ of $C^{(i)}$ such that the generic fiber of $X^{(i)}|_G\to G$ is normal, then $(X^{(i)},\frac{1}{n+1}D^{(i)})$ and $(Y^{(i)},\frac{1}{n+1}D_Y^{(i)})$ are isomorphic over $G\setminus C_0^{(i)}$.
 \end{corollary}
 
 \begin{proof}
 This follows from part (i) in the moreover part of Theorem \ref{thm_you_can_take_ruled_model_in_ksba_moduli}, as every closed subscheme one blows up along maps to a nodal point of $C_0^{(i)}$. 
     %Indeed, let $x\in E$ be a smooth point of the special fiber of $C\to \spec R $, and let $U$ be as in the Definition \textcolor{red}{REF}. If $x\in U$ there is nothing to prove, as the fibers of the ruled model are $\bP^1$, so let's assume $x\not \in U$. By our assumptions, the generic fober of $X\to \spec R $ belongs to $U$ and the irreducible component of the special fiber containing $x$ intersects $U$, so $x$ is an irreducible component of $X\smallsetminus U$, so if $V=U\cup\{x\}$, we have that $V$ is open in $X$ and $x$ has codimension 2 in $V$. Since $X$ and $Y$ are $S_2$, and $D$ (resp. $D_Y$) is ample over $C$, we have
     %\begin{align*}
         %&X\times_C V = \Proj_{\cO_C}(\bigoplus_m\oH^0(X\times_C V,\cO_{X\times_C V}(mD))) =\\& = \Proj_{\cO_C}(\bigoplus_m \oH^0(X\times_C U,\cO_{X\times_C U}(mD)))= \Proj_{\cO_C}(\bigoplus_m
         %\oH^0(Y\times_C U,\cO_{Y\times_C U}(mD_Y))) =\\& = \Proj_{\cO_C}(\bigoplus_m \oH^0(Y\times_C V,\cO_{Y\times_C V}(mD_Y)))=Y\times_C V.\end{align*}
 \end{proof}

\begin{corollary}\label{cor_mmp_does_not_contract_divisors}
 Let $(X^{(k)},\frac{1}{n+1}D^{(k)}+cF^{(k)})\to\Spec R $ be as in Diagram (\ref{eq: diagram of MMP}). Assume 
 \begin{itemize}
     \item for each point $p\in C_\eta$, the fiber $\pi^{-1}(p)$ is smooth, and
     \item $K_{C^{(k)}} + M^{(k)} + c'S^{(k)}$ is ample for some $c'\le c$. 
 \end{itemize}
 Fix $0<\epsilon \ll 1$ such that $(X^{(k)},(\frac{1}{n+1}+\epsilon)D^{(k)} + cF^{(k)}\big)$ is KSBA-stable.
 Then there is a weak canonical model $$\textstyle \big(X^{(k)},(\frac{1}{n+1}+\epsilon)D^{(k)} + c'F^{(k)}\big) \ \dashrightarrow\ \big(X',(\frac{1}{n+1}+\epsilon)D'+c'F'\big)$$ over $\spec R $ such that $X^{(k)}\dashrightarrow X'$ is obtained by a sequence of flips. If $K_{C^{(k)}} + M^{(k)} + c'S^{(k)}$ is ample for every $0<c'\le c$, every flip as above is a flop for $K_{X^{(k)}}+\frac{1}{n+1}D^{(k)}$.
\end{corollary}
In other terms, as long as there is no divisorial contraction on the family of curves going from $c$ to $c'$, if one fixes $\epsilon$ and decreases the coefficient of $F^{(k)}$, taking the minimal model does not contract divisors on $X^{(k)}$.

\begin{proof} By \Cref{cor_if_lc_div_on_the_curve_is_ample_its_pull_back_plus_epsilon+D_is_ample}, for any fixed $c'$, the pair $\big(X^{(k)},(\frac{1}{n+1}+\alpha)D^{(k)} + c'F^{(k)}\big)$ is KSBA-stable for any $0<\alpha\le \alpha_0$, where $\alpha_0$ depends on $c'$. Therefore, one can recover $X^{(k)}$ by 
\begin{enumerate}
    \item first running an MMP  rescaling $c$ to $c'$ in $(X^{(k)},(\frac{1}{n+1}+\epsilon)D^{(k)}+cF^{(k)})$, which leads to the weak canonical model $X'$, and then
    \item running an MMP rescaling $\epsilon$ to $\alpha$, which gives back $X^{(k)}$.
\end{enumerate}
 Neither of these MMP contains a divisorial contraction, as the composition is $\Id_{X^{(k)}}$.

 The last sentence follows from \Cref{lemma_flops_are_crepant}, as there is a morphism $\tau^{st}\colon C^{(k)}\to C^{st}$ contracting all the curves where $K_{C^{(k)}} + M^{(k)}$ has degree 0. 
  %We can obtain a weak canonical model $(X',(\frac{1}{n+1}+\epsilon)D' + c'F')\to\Spec R $ by running an MMP on $(X,(\frac{1}{n+1}+\epsilon)D + c'F)\to\Spec R $ by reducing the weight $t$ on $(X,(\frac{1}{n+1}+\epsilon)D + tcF)$; at each step of this MMP we have a pair $(X_{(i)},(\frac{1}{n+1}+\epsilon)D_{(i)}+tcF_{(i)})$. But by Remark \ref{cor_if_lc_div_on_the_curve_is_ample_its_pull_back_plus_epsilon+D_is_ample}, the stable limit $(X,(\frac{1}{n+1}+\epsilon')D+tcF)\to\Spec R $ does not depend on $\epsilon'$ as long as $0< \epsilon' \ll 1$, so we can obtain $(X,(\frac{1}{n+1}+\epsilon)D+c'F)\to\Spec R $ by taking the canonical model of $(X_{(i)},(\frac{1}{n+1}+\epsilon')D_{(i)}+c'F_{(i)})$ and replacing $\epsilon'$ with $\epsilon$ and $c$. In particular, there is a rational contraction $X_{(i)}\dashrightarrow X$. So $X\dashrightarrow X_{(i)}$ cannot contract a divisor.
\end{proof}

\section{Applications to explicit KSBA moduli spaces}\label{section_application_to_KSBA_moduli_when_epsilon_2_goes_to_0}

In this section, we apply the results of Section \S \ref{section_stable_reduction} to describe the objects on the boundary of the KSBA compactification of certain moduli spaces of surface pairs. 

Given a normal and irreducible moduli stack $\cU$ of surface pairs $(X,D)$ admitting a non-isotrivial fibration $(X,D)\to C$ over a smooth curve $C$, satisfying that
\begin{itemize}
    \item the fibers of $X\to C$ are $\bP^1$s;
    \item the restriction of $D$ on each fiber $\bP^1$ is $\sum n_ip_i$, where $p_i$ are distinct points, $0\leq n_i\leq n+1$ are integers such that $\sum n_i=2n+2$;
    \item\label{point_2} the generic fiber of $(X,\frac{1}{n+1}D)\to C$ is klt, i.e. $n_i\leq n$; and
    \item the pair $\big(X,(\frac{1}{n+1}+\epsilon_1)D + (\vec{c}+\epsilon_2) F\big)$ is KSBA-stable for any $0<\epsilon_1\ll \epsilon_2\ll 1$, where $F$ are some general fibers of $X\to C$ marked with a weight vector $\vec{c}$ with $0\le \vec{c}\le 1$.
\end{itemize}
One can keep in mind the following two examples:
\begin{enumerate}
    \item the moduli space parametrizing $(\bP^1\times\bP^1,\frac{1}{n}C)$ with $C$ a general $(2n,m)$-curve, with $2n\le m$ and $F=\emptyset$,
    \item  the moduli space parametrizing Hirzebruch surfaces $X$ which are quotients of elliptic K3 surfaces by an involution, with $D$ the ramification locus.
\end{enumerate}
%One can also add the data of marked fibers of $\pi$; we will not treat this case as the content is essentially unchanged, but the notations would become heavier. A concrete example for this case will be for $(1,4)$-curves in $\bP^1\times \bP^1$ treated in Section \S \ref{section_example_14_case}.
As the coefficient vector $\vec{c}$ does not play any significant role in what follows, we will suppress it.
For any $\epsilon_1,\epsilon_2>0$ such that $\big(X,(\frac{1}{n+1}+\epsilon_1)D + \epsilon_2 F\big)$ is KSBA-stable for a general point $[(X,D)]\in \cU(\bC)$, by possibly shrinking $\cU$, one has a morphism $$\Phi_{\epsilon_1,\epsilon_2}\colon \cU \ \longrightarrow \  \cM^{\KSBA}_n(\epsilon_1,\epsilon_2)$$ to the KSBA-moduli stack with fixed numerical invariants and marked coefficients $(\epsilon_1,\epsilon_2)$. We denote by $\ove{\cU}^{\KSBA}(\epsilon_1,\epsilon_2)$ the closure of its scheme-theoretic image.

One can collect all admissible coefficients $(\epsilon_1,\epsilon_2)$ together to get a $2$-dimensional domain. By \cite{ascher2023wall,meng2023mmp}, up to passing to the normalization of $\ove{\cU}^{\KSBA}(\epsilon_1,\epsilon_2)$, there are wall-crossing morphisms so that one can compare the normalizations of different compactifications; see \Cref{Figure:wall-crossing}.

In Section \S\ref{subsection_epsilon_2_non_zero}, we will apply the results in Section \S\ref{section_stable_reduction} to study the wall crossing along the red path from the purple point to the pink point in \Cref{Figure:wall-crossing}.

\begin{figure}
    \centering
\includegraphics[width=.25\linewidth]{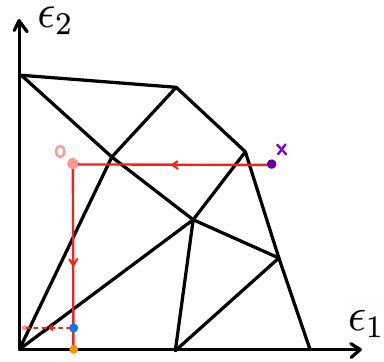}
    \caption{Wall crossing of KSBA moduli stacks $\mtc{M}^{\KSBA}_n(\epsilon_1,\epsilon_2)$}
    \label{Figure:wall-crossing}
\end{figure}

\subsection{KSBA-moduli space with a marked multi-section}\label{subsection_epsilon_2_non_zero}

Let $$\textstyle (\sX,\sD)\ \stackrel{\pi}{\longrightarrow}\  \sC \ \longrightarrow \ \cU$$ be the universal family over $\cU$. By \Cref{rem:existence of discriminant divisor}, there is an intrinsic Cartier divisor $\Delta\subseteq \sC$. whose support consists of points $p$ such that fibers of $\pi|_{\sD}:\sD\rightarrow \sC$ over $p$ are supported by less than $2n+2$ points, i.e. where at least two points collide. We denote by $\Delta_u$ the restriction of $\Delta$ on the fiber over $u\in \cU$. Then by \Cref{cor_singularities_remain_slc_if_increasing_coeff_on_d_and_adding_fibers} the family
 \[\textstyle
 \left(\sX,(\frac{1}{n+1}+\epsilon_1)\sD + \epsilon_2\pi^{-1}(\Delta)\right)\ \longrightarrow \ \cU
 \]
 is a family of KSBA-stable pairs for any $0<\epsilon_1\ll \epsilon_2\ll 1$. The following result, in particular, verifies that $$\ove{\cU}^{\KSBA}\ := \ \ove{\cU}^{\KSBA}(\epsilon_1,\epsilon_2)$$ does not depend on the choice of $0<\epsilon_1\ll \epsilon_2\ll 1$, which justifies our notation.
 \begin{theorem}\label{thm_section_4_in_ksba_limit_we_have_ruled_model}
     Keep the notations as above. Assume one of the following holds: 
     \begin{itemize}
         \item either $n=1$; or
         \item there exists a point $u\in \cU(\bC)$ such that the all fibers of $\big(\sX_u,(\frac{1}{n+1}+\epsilon_1)\sD_u\big)\to \sC_u$ over $\Delta_u$ are slc.
     \end{itemize} Then the surface pairs $(X,(\frac{1}{n+1}+\epsilon_1)D+\epsilon_2F)$ parametrized by $\ove{\cU}^{\KSBA}$ 
\begin{enumerate}
\item admit a ruled model (ref. \Cref{defn:ruled model}), and
\item do not depend on the choice of $\epsilon_i$, as long as $0<\epsilon_1\ll \epsilon_2\ll 1$.
\end{enumerate} 
 \end{theorem}
 \begin{proof} 

The results hold for the pair represented by the geometric generic point $\eta$ of $\ove{\cU}^{\KSBA}$. It suffices to prove the theorem for any stable limit of $\eta$ over $\Spec R$ for a DVR $R$. We first construct a limit using moduli of stable quasimaps, where the discriminant divisor $\Delta$ extends to the smooth locus.
 
 % We want to reduce to the situations studied in Section \S\ref{section_stable_reduction}. To this end, up to replacing $\cU$ by an \'etale cover, we need to guarantee that every pair $(\cX_u,(\frac{1}{n+1}+\epsilon_1)\cD_u + \epsilon_2\pi_u^{-1}(\Delta_u))$ admits a limit $(\sY,(\frac{1}{n+1}+\epsilon_1)\cD_\sY + \epsilon_2(\pi_u')^{-1}(\Delta_u'))$ which is still fibered in log Calabi-Yau pairs over a nodal curve $\cC'_u$, and such that $\Delta_u'$ is the pull-back of a divisor over a nodal curve, supported on the smooth locus of $\cC'_u$. If this is the case, then the results of \Cref{section_stable_reduction} go through and we can find an explicit weak canonical model as in \Cref{sec:KSBA limit}.

If $n=1$, consider the morphism $\cU\to \cQ_{g,0;1,\beta}$, and take an arbitrary limit of $\sC_{\eta}$ in $\cQ_{g,0;1,\beta}$, denoted by $\sC_u$. By \Cref{lemma_non_ksba_locus_is_cartier}, the limit of $\Delta_{\eta}$ over any one-parameter family is supported at those points whose images to $\sP_1$ represent pairs with divisor supported at most three points. When $n=1$, those points are not in $\sP_{1,\dm}$, and hence the support of the limit of $\Delta_{\eta}$ is contained in the smooth locus of $\sC_u$ by the stable quasimap condition (Sing.).

When $n>1$, this argument does not work, as the stable quasimap condition does not guarantee that the limit of $\Delta_\eta$ is contained in the smooth locus of $\sC_u$, since there are curves with two points colliding in $\sP_{n,\dm}$ for $n>1$. Instead, we use \textit{pointed} stable quasimaps as follows. 
     By our assumptions, there is an \'etale morphism $\cU'\to \cU$ with $\cU'\neq\emptyset$ such that:
 \begin{itemize}
     \item the cardinality of the points of $\Delta_{\cU'}$ is constant,
     \item the support of $\Delta_{\cU'}$ consists of disjoint sections $\sigma_1,\ldots,\sigma_m$, and
     \item the fibers of $(\sX,(\frac{1}{n+1}+\epsilon)\sD)|_{\cU'}\to \sC|_{\cU'}$ are slc over $\sigma_i$ for every $i$.
 \end{itemize}
These data yield a morphism $$\Phi^{(q)}\colon \cU' \ \longrightarrow \  \cQ_{g,m;n,\beta}$$ for a certain choice of $\beta$. Then by definition, the support of the limit of $\Delta_{\eta}$ in $\cQ_{g,m;n,\beta}$ is contained in the smooth locus, as desired.

Now, let $R$ be a DVR as above, \[\textstyle \left(\sY,\sD_\sY\right) \ \longrightarrow \ (\sC,x_1,\ldots,x_m)\ \longrightarrow \ \spec R \] be the stable quasimap limit construct as above, where $m=0$ if $n=1$, and let \[\textstyle \left(Y,D_Y\right)\ \stackrel{\pi}{\longrightarrow} \ (C,x_1,\ldots,x_m)\ \longrightarrow \ \spec R\] be the coarse spaces. Denote by $\eta$ and $u$ the generic point and the closed point of $\Spec R$ respectively as above. As the limit of $\pi^{-1}(\Delta_{u})$, which we denote by $F$, will not intersect the double locus
of the central fiber of $Y\to \spec R $, then the pair $(Y,(\frac{1}{n+1}+\epsilon_1)D_Y+\epsilon_2 F)$ satisfied the assumptions of \Cref{setup_for_existence_of_ruled_model}. Therefore, we can take a relative weak canonical model of $\big(Y,(\frac{1}{n+1}+\epsilon_1)D_Y+\epsilon_2 F\big)$ over $\Spec R$ as in Section \S\ref{sec:KSBA limit}. Then we may assume claim that $\big(Y,(\frac{1}{n+1}+\epsilon_1)D_Y+\epsilon_2 F\big)\rightarrow \Spec R$ is already a weak canonical model, and we claim that it is a canonical model for $0<\epsilon_1\ll\epsilon_2\ll 1$.

Let $(C,\epsilon_2\Delta'+\bfM)$ be the generalized pair defined by the canonical bundle formula for $(Y,\frac{1}{n+1}D_Y+\epsilon_2 F)$. By \Cref{cor_if_lc_div_on_the_curve_is_ample_its_pull_back_plus_epsilon+D_is_ample}, it suffices to check that $K_{C}+\epsilon_2\Delta'+\bfM$ is ample. Notice that it is already nef, and the only possible components of $C_u$, the restriction of $K_{C}+\epsilon_2\Delta'+\bfM$ to which is not ample, are 
\begin{enumerate}
    \item either a rational bridge $R$, or
    \item a rational tail $T$.
\end{enumerate}
From the stability condition (Stab.) for stable quasimaps, a rational bridge $R$ must contain a point in the closure of $\Delta_\eta$; this addresses case (1) as $\epsilon_2>0$. By \Cref{lemma_on_tails_i_have_a_marking}, $\Delta$ intersects every rational tail at the beginning of the MMP, and hence the same holds at the end of the MMP. As $K_{C}+\epsilon_2\Delta'+\bfM$ is nef for any $0<\epsilon_2\ll 1$ and $\Delta'$ intersects every rational bridge, then the restriction of $K_{C}+\epsilon_2\Delta'+\bfM$ must be ample for every $0<\epsilon_2\ll 1$ as desired.
 \end{proof}

Fix $g$ and $\beta$ as above. Assume the following condition holds: for every point in $\cQ_{g,0;1,\beta}$ parametrizing maps $C\to \sP_1$ from a smooth curve such that for every surface pair $(X,\frac{1}{2}D)\to C$ corresponding to a point in $\cU$, there exist $0<\epsilon_1\ll \epsilon_2$ such that $(X,(\frac{1}{2}+\epsilon_1)D+\epsilon_2 \pi^{-1}(\Delta_C))$ is KSBA-stable. Then the following holds.

\begin{corollary}\label{cor_there_is_a_map_normalization_qmaps_to_ksba} 
There exists a choice of $0<\epsilon_1\ll \epsilon_2\ll 1$, and a natural morphism \[\Phi_{\epsilon_1,\epsilon_2}\colon\cQ_{g,0;1,\beta}^{\norm} \ \longrightarrow \ \cM^{\KSBA}_2(\epsilon_1,\epsilon_2)\] from the normalization of $\cQ_{g,0;1,\beta}$, which generically is given by $$\textstyle \big[(X,\frac{1}{2}D)\to C\big] \ \mapsto \ \big[(X,(\frac{1}{2}+\epsilon_1)D+\epsilon_2 \pi^{-1}(\Delta_C))\big].$$  The surface pairs in the image of $\Phi_{\epsilon_1,\epsilon_2}$ do not depend on $\epsilon_i$ as long as $0<\epsilon_1\ll \epsilon_2\ll 1$, and they admit a ruled model.
\end{corollary}

\begin{proof}
    This follows immediately from wall-crossings for KSBA-stable pairs of \cite{ascher2023wall, meng2023mmp}, \Cref{thm_section_4_in_ksba_limit_we_have_ruled_model}, together with the fact that one can pick $\epsilon_1$ and $\epsilon_2$ uniformly due to the boundedness of $\cQ_{g,0;1,\beta}$.
\end{proof}
\subsection{Irreducible components of KSBA-stable limits}\label{subsection_irred_cpt_of_KSBA_limits}

The next three propositions are established under the following setup. Consider  
\[
\textstyle \big(X^{(i)},\frac{1}{n+1}D^{(i)}+cF^{(i)}\big) \ \stackrel{\pi^{(i)}}{\longrightarrow} \ C^{(i)} \ \longrightarrow \ \Spec R
\]
as in Diagram~(\ref{eq: diagram of MMP}). Let \(G\) be an irreducible component of \(C^{(i)}_0\), and let \(S := (\pi^{(i)})^{-1}(G)\) with the reduced scheme structure. The propositions that follow analyze the geometry of \(S\) in this context.

\begin{prop}\label{prop: geometry of S1}
    If $X^{(i)}$ agrees with the coarse moduli space of its ruled model in a neighborhood of $S$, then there is a stacky curve $\cG$ with coarse space $G$, and a vector bundle $\cV$ on $\cG$ such that $S$ is the coarse moduli space of $\bP_{\cG}(\cV)$. 
\end{prop}

\begin{prop}[ref. Figure \ref{fig:ruled model2}]\label{prop: geometry of S2}
    Suppose $S$ is normal but $X^{(i)}$ does not agree with its ruled model in a neighborhood of $S$. Then there is a stacky curve $\cG$ with coarse moduli space $G$, and a vector bundle $\cV$ on $\cG$ such that $(S,\frac{1}{n+1}D^{(i)}|_S)$ is crepant birational to the coarse space of $(\bP_{\cG}(\cV),\frac{1}{n+1}\cD)$. Moreover, $(S,\frac{1}{n+1}D^{(i)}|_S)$ is obtained by
    \begin{enumerate}
        \item (weighted) blowing up $\bP_\cG(\cV)$ along zero-dimensional subschemes, whose support map to the nodes $n_1,...,n_{\ell}$ of $C^{(i)}_0$ on $G$, then
        \item contracting proper transform of fibers over $n_1,...,n_{\ell}$, and taking coarse space. 
    \end{enumerate}
\end{prop}

\begin{remark}\textup{
  If $G$ is a rational tail or a rational bridge in Proposition \ref{prop: geometry of S1} or Proposition \ref{prop: geometry of S2}, then we can assume that $\cV=\cO_{\cG}\oplus L$ for a line bundle $L$ of non-positive degree. Indeed, by \cite[Theorem 2.4]{martens2012variations}, any vector bundle on a root stack of $\bP^1$ at at most two points is the direct sum of two line bundles $L_1,L_2$ with $\deg(L_1)\le \deg(L_2)$; one can twist $L_1\oplus L_2$ by $L_2^{-1}$.}
\end{remark}

\begin{proof}[Proof of Proposition \ref{prop: geometry of S1} and Proposition \ref{prop: geometry of S2}]
   Both propositions follow from Lemma \ref{lemma_root_stack_of_a_curve_has_no_brauer} and Theorem \ref{thm_you_can_take_ruled_model_in_ksba_moduli}.
\end{proof}

\begin{prop}[ref. Figure \ref{fig:Push-out diagram}]\label{prop: geometry of S3}
    Suppose that $S$ is not normal. Then it is isomorphic to the coarse space of the push-out $\cS$ of the following diagram
$$\begin{tikzcd}
\Sigma \arrow[r, "i"] \arrow[d, "2:1"'] & \mathbb{P}_{\Sigma}(\cV) \arrow[d] \\
\cG \arrow[r] & \cS
\end{tikzcd}, $$ where
\begin{itemize}
    \item $\cG$ and $\Sigma$ are stacky curves,
    \item $\Sigma\to \cG$ is a finite morphism of degree 2,
    \item $\cV$ is a rank 2 vector bundle on $\Sigma$, and
    \item $i$ is a section of $\bP_{\Sigma}(\bV)\rightarrow \Sigma$.
\end{itemize}
Let $F\subseteq S$ be the sublocus along which $S$ is glued to its nearby components in $X^{(i)}_0$. If $(S,\frac{1}{n+1}D|_S+F)$ is log Calabi-Yau, then each component of the coarse space of $\Sigma$ is rational.
\end{prop}

\begin{figure}
    \centering
\includegraphics[width=.5\linewidth]{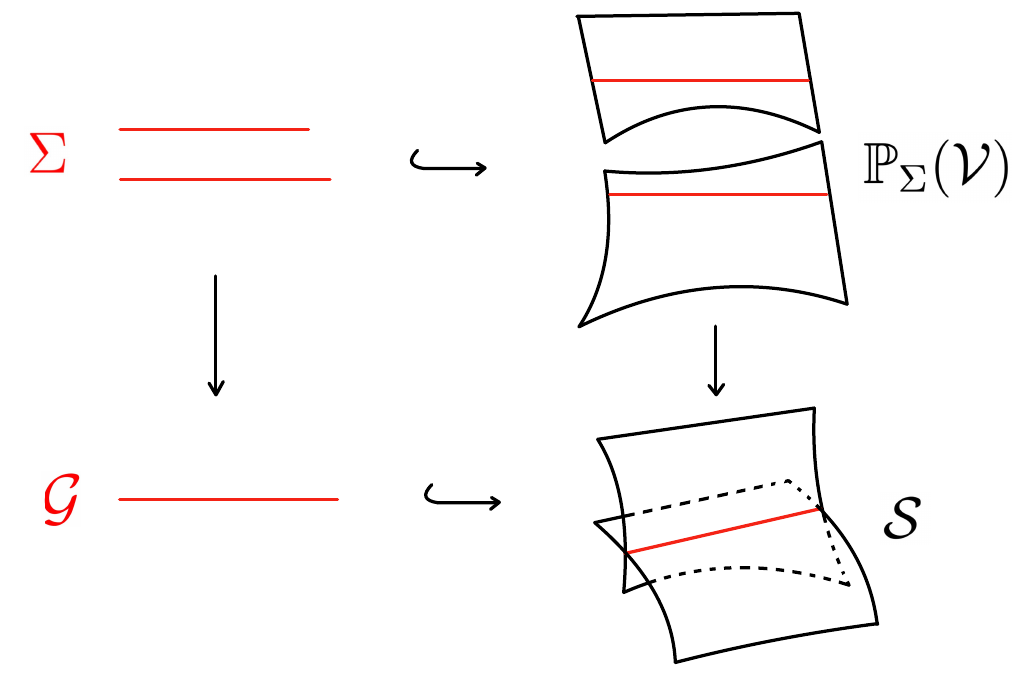}
    \caption{Push-out diagram defining $\cS$}
    \label{fig:Push-out diagram}
\end{figure}    
\begin{proof} 
  %  First, we argue that up to replacing $\sY$ with another Deligne-Mumford stack, the locus we perform a weighted blow-up on as in \Cref{thm_you_can_take_ruled_model_in_ksba_moduli} is a complete intersection. For doing so, we will first replace $\sY$ with a smooth Deligne-Mumford stack $\sY'$, and then find the explicit ideal sheaf we perform a weighted blowup over. Then the exceptional divisor of this weighted blow-up will be a projective bundle $\bP_\Sigma(\cV)$. Finally we argue that the surface $S$ will be the coarse space of the pushout.
    
    Consider the ruled model $\sY\to \sC$ of $X\rightarrow C$. As $\sC$ has $A_{n-1}$-singularities, then one can resolve them by taking their \emph{canonical smooth covering stack} $\sC'\to \sC$, which \'{e}tale locally replaces $\Spec k[\![x,y,t]\!]/(xy-t^n)$ with $[\Spec(k[\![u,v]\!])/\bmu_n]$, with the action being $$\zeta *u\ :=\ \zeta u, \ \ \ \ \zeta* v\ := \ \zeta^{-1}v.$$ 
    Observe that $\sC'\to \sC$ is an isomorphism on the locus where $\sC\to \spec R$ is smooth.
    It is straightforward to check that $\sC'\to \Spec R $ is also a twisted curve, with the same coarse moduli space as $\sC$. The pull-back of the universal family
    $(\sY,\sD_\sY)\to \sC$ along $\sC'\rightarrow \sC$ will be a $\bP^1$-bundle over $\sC'$, so it is smooth. To simplify the notation, we will still denote $\sC'$ by $\sC$ and $\sY\times_{\sC}\sC'$ by $\sY$; this will cause no confusion. Recall also that $G$ is an irreducible component of the central fiber of $C$; we will denote by $\cG$ the corresponding irreducible component in $\sC$.

    We now extract the divisors $\Gamma_{\ell,i}$ of \Cref{thm_you_can_take_ruled_model_in_ksba_moduli} in several steps. 
    By Theorem \ref{thm_you_can_take_ruled_model_in_ksba_moduli}, there is a curve $\Xi\hookrightarrow\sY\to \sC$, supported on $\sD_{\sY}$, which is a 2:1 multisection of the morphism    
    \[
    \cE:=\cG\times_\sC\sY \ \longrightarrow \ \cG,
    \] and such that $S$ is obtained by extracting a divisor over this multisection (the divisors $\widetilde{\Gamma}_{\ell,i}$ in \Cref{fig:ruled model2}), and contracting the proper transform of $\cE$. These transformations are performed as follows, where we will denote by $E$ the coarse moduli space of $\cE$. 

    \textbf{Step 1}. First we show that $\Xi$ is smooth.
    
    Indeed, away from the fibers over a finite set of smooth closed points $\{x_1,\ldots,x_r\}\subseteq \sC$, the surface pair $(\sY,\frac{1}{n+1}\sD_\sY)\to \sC$ comes from a morphism $\sC\to \sP_n^{\GIT}$.
    The curve $\cG$ maps to the polystable point $(\bP^1,(n+1)([0]+[\infty]))$ of $\sP_n^{\GIT}$, so the fibers of $\Xi\to \cG$ consist of two reduced points away from $\{x_1,\ldots,x_r\}$. We now argue that $\Xi$ is also smooth over $x_i$. Indeed, the curve $\sC$ is a scheme around $x_i$, and the two pairs $(X,\frac{1}{n+1}D)$ and $(Y,\frac{1}{n+1}D_Y)$ are crepant birational over the family of curves $C$. So if we denote by $\sigma_i$ a section of $C\to \spec R$ through $x_i$ (which exists up to replacing $\spec R$ with an \'etale cover) and by $F_i$ (resp. $F_{Y,i}$) its preimage in $X$ (resp. $Y$) also the two pairs $$\textstyle (X,\frac{1}{n+1}D+X_0+\epsilon F_i), \ \ \ \ \ \ (Y,\frac{1}{n+1}D_Y+Y_0+\epsilon F_{Y,i})$$ are crepant birational. By the last part of \Cref{thm_you_can_take_ruled_model_in_ksba_moduli}, for $0<\epsilon \ll 1$ the pair $(X,\frac{1}{n+1}D+X_0+\epsilon F_i)$ is lc so also $(Y,\frac{1}{n+1}D_Y+Y_0+\epsilon F_{Y,i})$ is lc. Then $(E,\frac{1}{n+1}(D_Y)|_{E}+\epsilon (F_{Y,i})_{E})$
    is lc, and as $\Xi$ has coefficient one in $\frac{1}{n+1}(D_Y)|_{E}$ it cannot be singular over $x_i$. So $\Xi$ is smooth as desired.

    Then $\Xi$ is a local complete intersection in $\sY$. We can blow it up to get $\sY_1\to \sY$. Let then $\sD_{\sY,1}$ be the proper transform of $\sD_\sY$, and $\cE_1$ the exceptional divisor (which could be two irreducible components if $\Xi$ is disconnected) with coarse space $E_1$. Let finally $\Xi_1$ be the intersection $\sD_{\sY,1}\cap \cE_1$.
    
    \textbf{Step 2}. We describe the blow-up $\sY_1\to \sY$ explicitly.
    
    Over the geometric generic point $\eta_G$ of $G$, we replaced \[\textstyle (\bP^1_{\eta_G},\frac{1}{n+1}((n+1)[0]_{\eta_G}+(n+1)[\infty]_{\eta_G})\] with a chain of three rational curves, and the proper
    transform of the divisor has is supported at the two tails.
    Moreover, from how the ruled model is constructed generically over $\sC$ (i.e. via the morphism $\sP_n\to \sP_n^{\GIT}$ as in \S \ref{subsubsection_recalling_how_to_take_ruled_model_in_qmap_paper}), the geometric fiber $(\sY_1,\frac{1}{2}\sD_{\sY,1})|_{\eta_G}\to \eta_G$ over the geometric generic point $\eta_G$ of $G$, is a chain of three $\bP^1$s with $2n+2$ points marked with coefficient $\frac{1}{n+1}$. Moreover, it has a point of coefficient one if and only if it has two points of coefficient one.
In other terms, 
\begin{center}either  $ \lfloor (\frac{1}{n+1}\sD_{\sY,1})|_{\eta_G} \rfloor =0$ or $\lfloor (\frac{1}{n+1}\sD_{\sY,1})|_{\eta_G} \rfloor  =   (\frac{1}{n+1}\sD_{\sY,1})|_{\eta_G}$.\end{center}  
    This follows as the geometric fiber of $X\to C$ over $\eta_G$ has an $A_{k}$-singularity \textit{for $k$ odd} (see \Cref{fig:nodes of X}).
    In particular there are two cases:
    \begin{enumerate}
        \item\label{point_1_in_a_proof} $\Xi_1$ has still an horizontal irreducible component of coefficient one, or
        \item\label{point_2_in_a_proof} over $\Xi$ the central fiber $(\sD_{\sY,1})_0$ of $\sD_{\sY,1}$ has no horizontal irreducible component of coefficient one.
    \end{enumerate}

\textbf{Step 3}. We now extract the proper transforms of $\Gamma_{\ell,i}$.

In case (\ref{point_1_in_a_proof}), observe that that $\Xi_1$ is smooth. Indeed, the blow-up $(\sY_1,\frac{1}{n+1}\sD_{\sY,1})\to (\sY,\frac{1}{n+1}\sD_\sY)$ is again crepant birational, and $\sY_1$ is again smooth. So we can repeat the argument above. Then we blow-up $\sY_2\to \sY_1$ along $\Xi_1$, and iterate this process until we arrive at case (\ref{point_2_in_a_proof})
\[
(\Xi_m\subseteq \sY_m)\to (\Xi_{m-1}\subseteq \sY_{m-1})\to \ldots \to (\Xi_1\subseteq \sY_1)\to (\Xi\subseteq \sY)
\]
So after a sequence of blow-ups, we extract the divisors which in the proof of \Cref{thm_you_can_take_ruled_model_in_ksba_moduli} and in \Cref{fig:ruled model2} we denoted by $\widetilde{\Gamma}_{\ell,i}$: the proper transforms of the irreducible components of $X$ mapping over the coarse space of $\cG\subseteq \sC$. From how $\widetilde{\Gamma}_{\ell,i}$ are constructed, they are the exceptional divisor of the blow-up of a
smooth stacky threefold over a smooth stacky curve. We will denote such a smooth stacky curve by $\Sigma$. So $\widetilde{\Gamma}_{\ell,i}\cong \bP_\Sigma(\cV)$ for a vector bundle $\cV$ on $\Sigma.$

There are maps $\bP_\Sigma(\cV)\to X$ and $\Sigma\to \cG\to X$ which map $\bP_\Sigma(\cV)$ to the irreducible components $\Gamma_{\ell,i}$ and $\Sigma$ to the double locus of the central fiber of $X\to \spec R$ over $G$.
This follows as in the proof of \Cref{thm_you_can_take_ruled_model_in_ksba_moduli}: proceeding as in \textit{loc. cit.} we observe that the pair $(Y_m,(\frac{1}{n+1}+\epsilon )D_m)$ is slc, and its canonical model $(Y_m,(\frac{1}{n+1}+\epsilon )D_m)\to (Y^c,(\frac{1}{n+1}+\epsilon)D^c)$
   over $Y$ contracts all but the last exceptional divisors we extracted. As in \Cref{thm_you_can_take_ruled_model_in_ksba_moduli}, one can study the structure of $Y^c$ over the nodal locus of $C$, and check that the pair $(Y^c,(\frac{1}{n+1}+\epsilon)D^c+Y^c_0)$ is slc using the canonical bundle formula. From separatedness of the KSBA moduli space, the canonical model for $(Y^c,(\frac{1}{n+1}+\epsilon)D^c+Y^c_0)$ is $X$, and there is a morphism $Y^c\to X$. This induces maps $\bP_\Sigma(\cV)\to S$ and $\Sigma\to \cG\to S$ where $S$ are the irreducible components of $X$ over $G$. 

\textbf{Step 4}. We claim that $\Gamma_{1,i}\cup\Gamma_{2,i}$ agrees with the coarse space of a pushout diagram, as in the statement.
   
   By \cite[Theorem 1.8]{alper2024artin}, one can take $\cS$ to be the push-out of the diagram as in the statement of the corollary, which admits a morphism $\cS\to S$ by the universal property of pushouts. Then $\cS$ is seminormal, as the pushout of seminormal Deligne-Mumford stacks is seminormal from the universal property of morphisms between seminormal Deligne-Mumford stacks.
   As $\cS$ is proper by \cite[Theorem A.4(viii)]{Rydh}, it admits a coarse moduli space $S'$. Then there is a morphism $\phi\colon S'\rightarrow S$ from the universal property of coarse moduli spaces.
   %The existence of this morphism is as in the proof of \Cref{thm_intro_ruled_model_exists}: proceeding as in \textit{loc. cit.} we observe that the pair $(Y_m,(\frac{1}{n+1}+\epsilon )D_m)$ is slc, and its canonical model $(Y_m,(\frac{1}{n+1}+\epsilon )D_m)\to (Y^c,(\frac{1}{n+1}+\epsilon)D^c)$ over $Y$ contracts all but the last exceptional divisors we extracted. As in \Cref{thm_intro_ruled_model_exists}, one can study the structure of $Y^c$ over the nodal locus of $C$, and check that the pair $(Y^c,(\frac{1}{n+1}+\epsilon)D^c+Y^c_0)$ is slc using the canonical bundle formula. From separatedness of the KSBA moduli space, the canonical model for $(Y^c,(\frac{1}{n+1}+\epsilon)D^c+Y^c_0)$ is $X$, and there is a morphism $Y^c\to X$.
    As the fibers of $S'\to S$ are connected and from how $S'$ is constructed, the map $S'\to S$ does not contract any curve. In particular, it induces a bijection on closed points, and $S$ is seminormal as it is an irreducible component of an slc pair. Since $S'$ is seminormal by \Cref{lemma_gms_of_seminormal_is_seminormal}, then $\phi$ is an isomorphism by the universal property of seminormalizations.

   The last paragraph of the proposition follows from the canonical bundle formula.
   Denote by $P$ the coarse moduli space of $\bP_\Sigma(\cV)$, and by $\sigma$ the coarse space of $\Sigma$. Then the pair $(P,\sigma + \frac{1}{n+1}D|_P +F|_P)$ is the normalization of $(S,\frac{1}{n+1}D|_S+F|_S)$. In particular, there is a fibration $(P,\sigma + \frac{1}{n+1}D|_P+F|_P)\to G$ in log Calabi-Yau pairs, and the pair $(P,\sigma + \frac{1}{n+1}D|_P+F|_P)$ is Calabi-Yau: from the canonical bundle formula $G$ is rational. 
\end{proof}
\subsubsection{Non-slc fibers and boundary part}\label{sec:classification of fibers}
We now classify the non-slc fibers of $$\textstyle\left(X,\big(\frac{1}{2}+\epsilon_1)D+\epsilon_2F\right) \ \longrightarrow \ C$$
over the smooth locus of $C$, for a fibration
parametrized by $\cM^{\KSBA}_2(\epsilon_1,\epsilon_2)$.
Coupled with Section \S \ref{subsection_irred_cpt_of_KSBA_limits}, this gives a complete characterization of the surface pairs on the boundary of $\cM^{\KSBA}_2(\epsilon_1,\epsilon_2)$, completing the proof of \Cref{thm_intro_description_boundary}.

   Let $R$ be a DVR and let $D$ be a divisor on $\bP^1_R$ such that the pair $(\bP^1_R,\frac{1}{2}D)$ is klt and fibered in log Calabi-Yau over $\spec R$. From \Cref{thm_you_can_take_ruled_model_in_ksba_moduli} the boundary part in the canonical bundle formula appears with coefficient $0<b<1$ over the closed point $0\in\spec R$. Then $\Supp(D|_{F})$ is at most two points, where $F=F_0$ is the fiber over $0\in \Spec R$. We first display the cases when $D$ has unibranched singularities; the branched singularities are the combination of unibranched singularities.

\begin{prop}
  Suppose that the support of $D\cap \pi^{-1}(p)$ consists of two points $x_1$ and $x_2$. Then up to swapping the indices, the local intersection number $(D.F)$ is $1$ at $x_2$ and $3$ at $x_2$. Moreover, the following holds.
    \begin{enumerate}
        \item If $D$ is smooth at $x_2$, then $D\to \spec R$ is ramified of ramification index $3$ and $b=\frac{1}{6}$.
        \item If $D$ is singular at $x_3$ of multiplicity $2$, then the analytic local defining polynomials for $(D,F)$ are $(y^2-x^3,y)$, and $b=\frac{1}{3}$.
        \item If $D$ is singular at $x_3$ of multiplicity $3$, then the analytic local defining polynomials for $(D,F)$ are $(y^3-x^m,x)$, where $m=4$ or $5$; in these two cases, one has $b=\frac{2}{3}$ and $b=\frac{5}{6}$ respectively.
    \end{enumerate}
\end{prop}

\begin{proof}
   By assumption, one has that $(D.F)_{x_1}+(D.F)_{x_2}=4$. If $(D.F)_{x_1}=2$, then $D$ is either smooth at $x_1$, or $x_1$ is a double point of $D$; in both cases $(\bP^1_{R},F+D)$ is log canonical at $x_1$ (and similarly at $x_2$) by inversion of adjunction. Thus it has to be that $(D.F)_{x_1}=1$ and $(D.F)_{x_2}=3$, up to possibly swapping the indices.

   If $D$ is smooth at $x_2$, then the analytic local defining polynomial of $(D,F)$ is $(y-x^3,y)$, in which case one has that $\lct(\bP^1_R,\frac{1}{2}D;F)=\frac{5}{6}$, and hence $b=\frac{1}{6}$.

   If $\mult_{x_2}D=2$, then analytic locally, $D$ is defined by the polynomial $y^2-x^m$. As $(D.F)_{x_2}=3$, then it has to be that $m=3$ and the (analytic) defining polynomial of $F$ is $y$. In this case, one has that $\lct(\bP^1_R,\frac{1}{2}D;F)=\frac{2}{3}$, and hence $b=\frac{1}{3}$. 

   Now assume that $\mult_{x_2}D=3$. Since $D$ is assumed to be unibranched, then analytic locally, the defining polynomial of $D$ is $y^3-g(x,y)$, where $g(x,y)$ is a polynomial with every monomial having degree at least $4$. Take a blow-up of $\bP^1_R$ at $x_2$ and denote by $E$ (resp. $\wt{D}$ and $\wt{F}$) the exceptional divisor (resp. proper transform of $D$ and $F$). Since $D$ is unibranched at $x$, then $\Supp(\wt{D}\cap E)$ is a single point, denoted by $p$. 
   \begin{enumerate}
       \item If $\mult_y\wt{D}\geq3$, then $(\bP^1_R,\frac{1}{2}D)$ is not klt: the log discrepancy with respect to the exceptional divisor of blow-up at $p$ is at most $0$. 
       \item If $\mult_p\wt{D}=2$, then analytic locally at $y$, $\wt{D}$ is defined by polynomial $u^2-v^m$. Since $$(\wt{D}.\wt{F}+E)\ = \ (\wt{D}.\wt{F}+E)_p \ =\ (\wt{D}.E)_p \ = \ (D.F)_{x_2}=3,$$ then it has to be the case that $m=3$ and $E$ is defined by the polynomial $u$, analytic locally at $p$. Therefore, $(D.F)$ has local polynomial $(y^3-x^5,x)$, and $b=\frac{5}{6}$.
       \item The similar argument shows that if $\mult_p\wt{D}=1$, then $(D.F)$ has local polynomial $(y^3-x^4,x)$, and $b=\frac{2}{3}$.
   \end{enumerate}
\end{proof}

Analogously, one has the following.

\begin{prop}
  Suppose that the support of $D\cap \pi^{-1}(p)$ consists of a single point $x$. Then $\mult_x(D)=1$ or $3$, and the following holds.
    \begin{enumerate}
        \item If $D$ is smooth at $x$, then $D\to \spec R$ is ramified of ramification index $4$ and $b=\frac{1}{4}$.
        \item If $D$ is singular at $x$ (of multiplicity $3$), then the analytic local defining polynomials for $(D,F)$ are $(y^3-x^4,y)$, and one has $b=\frac{3}{4}$.
    \end{enumerate}
\end{prop}

\renewcommand{\arraystretch}{1.5}
\begin{longtable}{| p{.10\textwidth} | p{.18\textwidth} | p{.20\textwidth} | p{.07\textwidth} | p{.07\textwidth} |}

\hline
\textbf{Index} & \textbf{Multiplicity} & \textbf{Local equation} & \textbf{lct} & \textbf{b} \\ \hline
\endfirsthead
\hline
\textbf{Index} & \textbf{Multiplicity} & \textbf{Local equation} & \textbf{lct} & \textbf{b} \\ \hline
\endhead
$3$ & $1$ & $(y-x^3,y)$ & $\tfrac{5}{6}$ & $\tfrac{1}{6}$ \\ \hline
$3$ & $2$ & $(y^2-x^3,y)$ & $\tfrac{2}{3}$ & $\tfrac{1}{3}$ \\ \hline
$3$ & $3$ & $(x^3-y^4,y)$ & $\tfrac{1}{3}$ & $\tfrac{2}{3}$ \\ \hline
$3$ & $3$ & $(x^3-y^5,y)$ & $\tfrac{1}{6}$ & $\tfrac{5}{6}$ \\ \hline
$4$ & $1$ & $(y-x^4,y)$ & $\tfrac{3}{4}$ & $\tfrac{1}{4}$ \\ \hline
$4$ & $3$ & $(y^3-x^4,y)$ & $\tfrac{1}{4}$ & $\tfrac{3}{4}$ \\ \hline
\caption{Boundary part: unibranched case} 
\end{longtable}

\renewcommand{\arraystretch}{1.5}
\begin{longtable}{| p{.10\textwidth} | p{.18\textwidth} | p{.30\textwidth} | p{.07\textwidth} | p{.07\textwidth} |}

\hline
\textbf{Index} & \textbf{Multiplicity} & \textbf{Local equation} & \textbf{lct} & \textbf{b} \\ \hline
\endfirsthead
\hline
\textbf{Index} & \textbf{Multiplicity} & \textbf{Local equation} & \textbf{lct} & \textbf{b} \\ \hline
\endhead
$3$ & $2$ & $\big(x(y-x^2),y\big)$ & $\tfrac{3}{4}$ & $\tfrac{1}{4}$ \\ \hline
$3$ & $3$ & $\big(x(x-ay)(x-by),y\big)$ & $\tfrac{1}{2}$ & $\tfrac{1}{2}$ \\ \hline
$3$ & $3$ & $\big(x(x^2-y^3),y\big)$ & $\tfrac{1}{4}$ & $\tfrac{3}{4}$ \\ \hline
$3$ & $3$ & $\big((x+ay)(x^2-y^m),y\big)$ & $\tfrac{1}{2}$ & $\tfrac{1}{2}$ \\ \hline
$4$ & $2$ & $\big(x(y-x^3),y\big)$ & $\tfrac{2}{3}$ & $\tfrac{1}{3}$ \\ \hline
$4$ & $2$ & $(y^2-x^4,y)$ & $\tfrac{1}{2}$ & $\tfrac{1}{2}$ \\ \hline
$4$ & $3$ & $\big((y-x^2)(y-ax)(y-bx),y\big)$ & $\tfrac{1}{2}$ & $\tfrac{1}{2}$ \\ \hline
$4$ & $3$ & $\big(x(y^2-x^3),y\big)$ & $\tfrac{1}{3}$ & $\tfrac{2}{3}$ \\ \hline
\caption{Boundary part: non-unibranched case} 
\end{longtable}
\begin{remark}
    The contribution for the boundary part on the normalization of a surface as in \Cref{prop: geometry of S3} is $\frac{1}{2}$. This happens if either the divisor is tangent to the horizontal double locus, or if the divisor has two transverse branches which both intersect transversally the double locus.
\end{remark}

\begin{remark}
    We remark that in this subsection we study the horizontal part of the divisor. It could (and in fact, in some cases it does) happen however that the divisor has vertical components. For example, a smooth $(4,4)$ curve in $\bP^1\times\bP^1$ can degenerate to the sum of four fibers for each of the two fibrations of $\bP^1\times \bP^1$. 
\end{remark}

\section{GIT quotients, K-moduli and quasimaps moduli spaces}\label{section_example_14_case}

The goal of this section is to compare the GIT, K-moduli, and quasimap compactification of the moduli stack (resp. space) of smooth (1,4)-curves in $\bP^1\times \bP^1$. We will show that the K-moduli is isomorphic to the GIT moduli, and there exists a birational morphism from the moduli of quasimaps to the GIT moduli.

\subsection{Quasimap and GIT compactifications of $(1,4)$-divisors}\label{sec:GIT and quasimap classification} We begin with describing the objects parametrized by the quasimap moduli stack and the GIT moduli, respectively. 

Let $\beta$ be a class in $\sP_1$ such that $\cQ_{0,1,\beta}$ generically parametrizes maps $\bP^1\to \sP_1$, which induces a family $$\textstyle p_1\ :\  (\bP^1\times\bP^1,\frac{1}{2}D) \ \longrightarrow\  \bP^1,$$ where $D$ is a $(1,4)$-divisor, i.e. $\mtc{O}_{\bP^1\times\bP^1}(D)\simeq p_1^*\mtc{O}_{\bP^1}(1)\otimes p_2^*\mtc{O}_{\bP^1}(4)$. 
\begin{notation}
    \textup{We denote by $\overline{\cQ}$ the closure of the locus in $\cQ_{0,1,\beta}$ parametrizing maps $\bP^1\to \sP_1$ corresponding to a divisor of class $(1,4)$ in $\bP^1\times \bP^1$.}
\end{notation}

\begin{prop}\label{prop:classfication_of_quasimaps_1_4}
   A quasimap $[\phi\colon \cC\to \sP_1]\in \overline{\cQ}$ satisfies one of the following \textup{(ref. Figure (\ref{fig:Classification of quasimaps}))}:
    \begin{enumerate}
        \item $\cC\cong \bP^1$, and the map $\phi$ induces a family $\bP^1\times\bP^1\to \bP^1$ with divisor $(1,4)$, or
        \item $\cC$ is a nodal union of two irreducible components $\cC_4\cup \cC_2$ with a $\bmu_3$ stabilizer as a node, and the map $\phi$ on good moduli spaces has degree 4 (resp. 2) on the coarse moduli space of $\cC_4$ (resp. $\cC_2$), or
        \item $\cC$ is a nodal union of two irreducible components $\cC_{3,1}\cup \cD_{3,2}$ with a $\bmu_2$ or $\bmu_4$ stabilizer as a node, and the map $\phi$ on good moduli spaces has degree 3 on the coarse moduli space of $\cC_{3,1}$ and $\cD_{3,2}$, or
        \item $\cC$ is a nodal union of four irreducible components $\cC_{2,1}\cup \cC_{2,2}\cup \cC_{2,3}\cup \cD$ with a $\bmu_3$ stabilizer at each node, and the map $\phi$ on good moduli spaces has degree 2 on the coarse moduli spaces of $\cC_{2,i}$ for every $i$ and contracts $\cD$, or
         \item $\cC$ is a nodal union of three irreducible components $\cC_{2,1}\cup \cC_{2,2}\cup \cC_{2,3}$ with a $\bmu_3$ stabilizer at each node, and the map $\phi$ on good moduli spaces has degree 2 on the coarse moduli spaces of $\cC_{2,i}$ for every $i$; or
           \item $\cC$ is a nodal union of three irreducible components $\cC_{2}\cup_p \cC_{1}\cup_q \cC_{3}$ with a $\bmu_3$ stabilizer at $p$ and a $\bmu_4$ stabilizer at $q$, and the map $\phi$ on good moduli spaces has degree $i$ on the coarse moduli spaces of $\cC_{i}$ for every $i$.
    \end{enumerate}
    \begin{figure}
    \centering
\includegraphics[width=.7\linewidth]{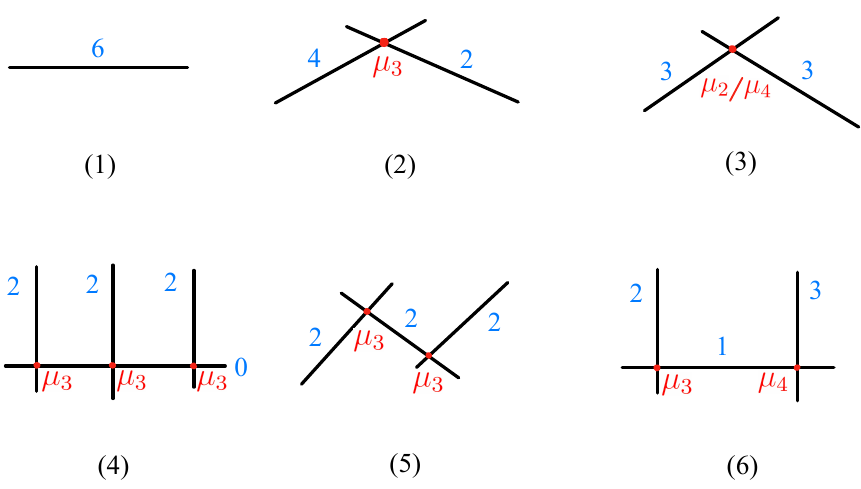}
    \caption{Classification of quasimaps \\
    red = stabilizer; blue = degree}
    \label{fig:Classification of quasimaps}  
\end{figure}
\end{prop}
\begin{proof}  Assume first that $\cC$ is smooth, i.e. $\mtc{C}\simeq \bP^1$. By \Cref{lem: Hirzebruch}, the pair $(X,\frac{1}{2}D)$ satisfies that either $X\simeq \bP^1\times \bP^1$ or $X\simeq \bF_1$. However, $\bP^1\times\bP^1$ cannot degenerate to $\bF_1$: a ruling $f_1$ of $\bP^1\times \bP^1$ would degenerate to a fiber of $\bF_1\to \bP^1$, and the other ruling $f_2$ of $\bP^1\times \bP^1$ would degenerate to a curve class in $F_2$ of the form $\mtf{e} + a \mtf{f}$; but $(\mtf{e} + a \mtf{f})^2 = -1 +2a\neq  0$.

We now assume that $\cC$ is not smooth. 
Let $\cD\subseteq \cC$ be an (irreducible) tail of $\cC$, i.e. it corresponds to a leaf in the dual graph of $\cC$. Since $\phi$ is stable, there must be another irreducible component $\cD'$ which is a leaf in the dual graph. Denote by $D$ (resp. $D'$) the coarse moduli space of $\cD$ (resp. $\cD'$). As $\phi$ is stable, the map $D\to \bP^1$ is finite, and $\cD$ and $\cD'$ have exactly one node, $n$ and $n'$. We know that $\deg(f|_{D})+\deg(f|_{D'})\le 6$ and $1\le\min( \deg(f|_{D'}),\deg(f|_{D'}))$, so there are only a few possible cases. We may assume that $\deg (f|_{D})\leq 3$.
\begin{enumerate}[(i)]
    \item $\deg(f|_{D})=1$ is impossible: as $\cC$ has at most one twisted note on $\cD$, then $\mtc{D}$ at most one stacky point, which contradicts Lemma \ref{lemma_if_C_is_a_scheme_on_the_pteimage_of_x_i_then_i_divides_the_order_of_f}.
    \item If $\deg(f|_{D})=2$, then Lemma \ref{lemma_if_C_is_a_scheme_on_the_pteimage_of_x_i_then_i_divides_the_order_of_f} implies that $\cD$ has a node over $x_3$, and the map $D\to \bP^1$ is totally ramified over $x_3$. From the stability condition, one can check that there is at most one irreducible component of $\cC$ over which the degree of $f$ is zero, and this component must be attached to the point on $\cC$ mapping to $x_3$. Thus, the only possible cases are those listed as (2), (4), (5) and (6).
    \item If $\deg(f|_{D})=3$, then Lemma \ref{lemma_if_C_is_a_scheme_on_the_pteimage_of_x_i_then_i_divides_the_order_of_f} implies that $\cD$ has a node over $x_2$. Since $\deg(f)=6$ and each irreducible component which is a leaf on the dual graph of $\cC$ maps finitely to $\bP^1$, from (i) there is exactly one other leaf. This is case (3).
\end{enumerate}\end{proof}

\begin{prop}\label{prop:classification of GIT}
   Let $C_0:=\bV(x_0^4y_0-x_1^4y_1)$ and $C_1:=\bV(x_0x_1(x_0^2y_0-x_1^2y_1))$ be two curves on $\bP^1_{[x_0:x_1]}\times\bP^1_{[y_0:y_1]}$. Then under the $G:=\PGL(2)\times\PGL(2)$-action, a curve $C\in |\mtc{O}_{\bP^1\times\bP^1}(1,4)|$ is
   \begin{enumerate}
       \item GIT stable if and only if it is smooth and does not isotrivially degenerate to $C_0$ or $C_1$;
       \item strictly GIT polystable if and only if it is in the same $G$-orbit of $[C_0]$ or $[C_1]$.
   \end{enumerate}  
\end{prop}

\begin{proof}
    Observe that a $(1,4)$-curve on $\bP^1\times\bP^1$ is singular if and only if it is reducible. For any $(1,4)$-curve $C=\bV(f)$, consider the convex hull $\Delta(C)$ in the following diagram spanned by all monomials in $f$ with non-zero coefficients, where the red bullet point $P$ is the barycenter of the 10-vertex-diagram. 
    \[\begin{tikzcd}
	\begin{array}{c}  x_0^4y_0 \\ \bullet \end{array} & \begin{array}{c}  x_0^3x_1y_0 \\ \bullet \end{array} & \begin{array}{c}  x_0^2x_1^2y_0 \\ \bullet \end{array} & \begin{array}{c}  x_0x_1^3y_0 \\ \bullet \end{array} & \begin{array}{c}  x_1^4y_0 \\ \bullet \end{array} \\
	&& {\color{red}{\bullet}} \\
	\begin{array}{c} \bullet \\ x_0^4y_1 \end{array} & \begin{array}{c} \bullet \\ x_0^3x_1y_1 \end{array} & \begin{array}{c} \bullet \\ x_0^2x_1^2y_1 \end{array} & \begin{array}{c} \bullet \\ x_0x_1^3y_1 \end{array} & \begin{array}{c} \bullet \\ x_1^4y_1 \end{array}
\end{tikzcd}\]
Then a curve $C$ is 
    \begin{itemize}
        \item GIT semistable if and only if $P$ is contained in $\Delta(C')$ for any curve $C'$ in the same $G$-orbit of $C$; 
        \item GIT polystable if and only if it is GIT semistable and there exists a $C'$ in the same $G$-orbit of $C$ such that $\Delta(C')$ is a line segment through $P$;
        \item GIT stable if and only if $P$ is contained in the interior of $\Delta(C')$ for any curve $C'$ in the same $G$-orbit.
    \end{itemize}
    The only two strictly GIT polystable points up to the $G$-action are $[C_0]$ and $[C_1]$. 
\end{proof}

To understand the interpolation between the GIT moduli and the moduli of quasimaps, let us classify the smooth divisors on $\bP^1\times\bP^1$ of bidegree $(1,4)$. Let $C\subseteq \bP^1\times\bP^1$ be such a divisor. Then 
\begin{itemize}
    \item the projection $p_2:(\bP^1\times\bP^1,\frac{1}{2}C)\rightarrow \bP^1$ induces a morphism $\bP^1\rightarrow \sP_1^{\GIT}$, i.e. every point on $C$ has ramification index $\leq 2$ with respect to $p_2|_C:C\rightarrow \bP^1$; or
    \item $p_2|_C:C\rightarrow \bP^1$ has a branch point with ramification index either $3$ or $4$. Considering all possible combinations, one of the following cases occur:
    \begin{enumerate}
        \item there is a unique ramification index $3$ point;
        \item there are two ramification index $3$ points;
        \item there are three ramification index $3$ points;
        \item there is a unique ramification index $4$ point;
        \item there is a unique ramification index $4$ point and a unique ramification index $3$ point;
        \item there are two ramification index $4$ points.
    \end{enumerate}
\end{itemize}

Denote by $Z_i$ the sublocus parametrizing curves satisfying item (i), where $i=1,...,6$. Then these loci have the following specialization relation:
\[\begin{tikzcd}
	{Z_6=\{Q\}} & {\ove{Z}_5} & {\ove{Z}_4} & {\ove{Z}_1} \\
	& {Z_3=\{P\}} & {\ove{Z}_2}
	\arrow["\subseteq"{marking, allow upside down}, draw=none, from=1-1, to=1-2]
	\arrow["\subseteq"{marking, allow upside down}, draw=none, from=1-2, to=1-3]
	\arrow["\subseteq"{marking, allow upside down}, draw=none, from=1-2, to=2-3]
	\arrow["\subseteq"{marking, allow upside down}, draw=none, from=1-3, to=1-4]
	\arrow["\subseteq"{marking, allow upside down}, draw=none, from=2-2, to=2-3]
	\arrow["\subseteq"{marking, allow upside down}, draw=none, from=2-3, to=1-4]
\end{tikzcd},\] where $\ove{Z}_i$ is the closure of $Z_i$.

\subsection{From quasimap moduli space to GIT moduli space}
The goal of this subsection is to relate the quasimap moduli space with the GIT moduli space of (1,4)-curves in $\bP^1\times \bP^1$. More precisely, let $\sB$ be the normalization of a root stack of $\overline{\cQ}$, which is defined in Section \S \ref{sec:step 2}; we will construct a morphism from $\sB$ to the GIT moduli stack of $(1,4)$-curves in $\bP^1\times \bP^1$. 

\begin{prop}\label{prop:existence of a morphism}
   There exists a morphism $$\sB^n\ \longrightarrow \  \cM^{\GIT}_{(1,4)}.$$
\end{prop}

\noindent We divide the argument into three steps, by modifying the universal family of surface pairs over $\ove{\cQ}$, temporarily denoted by $(\sX,\sD)\to \sC\to \overline{\cQ}$.

\textbf{Step 1.} Using the classification of \Cref{prop:classfication_of_quasimaps_1_4} and the wall crossing of KSBA moduli spaces (ref. \cite{ascher2023wall,meng2023mmp}), we replace the family of surface pairs over the normalization of $\overline{\cQ}$ with a family obtained by a birational contraction which contracts the irreducible components of $\sX$ over irreducible components of $\sC$ that map to $\bP^1$, the coarse space of $\sP_1$, with a map of degree 2. These appear in \Cref{prop:classfication_of_quasimaps_1_4} as the cases (2), (4), (5) and (6). The end result will be a family of surface pairs over the normalization of $\overline{\cQ}$.

\textbf{Step 2.} The transformation of Step 1 leaves a certain locus on $\sX$ unchanged, as it only contracts irreducible components which map to tails of degree $2$. Specifically, it leaves a neighborhood of the nodal locus in each irreducible component mapping to a tail of degree $3$ unchanged. In Step 2, we study the local behavior of $\sC$ on this locus, and then perform a root stack followed by a specific blow-up to add a (stacky) rational bridge over the nodes of $\sC$, which have $\bmu_2$ or $\bmu_4$ as stabilizers. This modifies the family of surface pairs by adding an irreducible component which is isotrivial.

\textbf{Step 3.} Apply the wall-crossing results of \cite{ascher2023wall,meng2023mmp} to perform a birational contraction. The resulting family of surface pairs, which will be over the normalization of a root stack of $\overline{\cQ}$, will give the desired morphism to the GIT moduli stack $\cM^{\GIT}_{(1,4)}$ of $(1,4)$-curves in $\bP^1\times \bP^1$.

\subsubsection{Step 1}\label{subsub_step_1} We begin by simplifying the family of surface pairs $$\textstyle (\sX,\frac{1}{2}\sD)\ \longrightarrow \ \sC \ \longrightarrow \ \overline{\cQ}^n$$ over the normalization $\overline{\cQ}^n$ of $\overline{\cQ}$, contracting via a birational contraction $\sX \dashrightarrow \sX'$ the irreducible components of $\sX$ which map to an irreducible component of $\sC$ with a node with stabilizer $\bmu_3$. 

Let $B$ be a normal scheme of finite type over $\bC$, let $\psi\colon\cC\rightarrow B$ be a family of twisted curves, and $\phi\colon \cC\to \sP_1\times B$ be a family of stable quasimaps parametrized by $\overline{\cQ}$. Assume that the generic fiber of $\psi$ is smooth. Let $(\cX,\frac{1}{2}\cD)\to \cC$ be the universal family of surface pairs with coarse spaces $\pi\colon (X,\frac{1}{2}D)\to C$, and let $f\colon C\to \bP^1$ be the induced morphism between the coarse spaces of $\cC\rightarrow \sP_1$.

\[\begin{tikzcd}
	& {(\mathcal{X},\frac{1}{2}\mathcal{D})} &&& {(\mathcal{X}_{\univ},\frac{1}{2}\mathcal{D}_{\univ})} \\
	{(X,\frac{1}{2}D)} & {\mathcal{C}} && {\mathscr{P}_1\times B} & {\mathscr{P}_1} \\
	C && B && {\mathbb{P}^1}
	\arrow[from=1-2, to=1-5]
	\arrow[from=1-2, to=2-1]
	\arrow[from=1-2, to=2-2]
	\arrow[from=1-5, to=2-5]
	\arrow["\pi"{description}, from=2-1, to=3-1]
	\arrow["\phi", from=2-2, to=2-4]
	\arrow[from=2-2, to=3-1]
	\arrow["\psi"{description}, from=2-2, to=3-3]
	\arrow["{p_2}", from=2-4, to=2-5]
	\arrow["{p_2}"{description}, from=2-4, to=3-3]
	\arrow[from=2-5, to=3-5]
	\arrow["f"{description}, bend right = 18pt, from=3-1, to=3-5]
\end{tikzcd}\]

\begin{lemma}
     Let $F_{a}:=(f\circ \pi)^{-1}(a)$ be the fiber over $a\in \bP^1$. Then for three general points $a,b,c\in \bP^1$, the family 
    \[\textstyle 
    \xi\colon \left(X,(\frac{1}{2}+\epsilon)D +\frac{1}{6}(F_a + F_b + F_c) \right)\ \longrightarrow \ B
    \]
    is KSBA-stable. In particular, there is a morphism $\overline{\cQ}\to \cM^{\operatorname{KSBA}}$ to the moduli space of KSBA-stable pairs.
\end{lemma}
\begin{proof}
    As $a,b,c$ are general, and all the nodes of a family in $\overline{\cQ}$ are over either $x_2$ or $x_3$ by \Cref{prop:classfication_of_quasimaps_1_4}, then the pair is locally stable. Using the canonical bundle formula and the classification of \Cref{prop:classfication_of_quasimaps_1_4} one can check that each fiber of $\xi$ is KSBA-stable.
\end{proof}
When $B$ is normal, one can use the results in \Cref{sec:KSBA limit} together with the wall-crossing results to construct the relatively stable model of 
\[\textstyle
\xi\colon \left(X,(\frac{1}{2}+\epsilon)D +(\frac{1}{9} +\epsilon)(F_a + F_b + F_c) \right)\ \longrightarrow \ B,
\]
which we denote by 
\[\textstyle 
\xi'\colon \left(X',(\frac{1}{2}+\epsilon)D' +(\frac{1}{9} +\epsilon)F \right)\ \longrightarrow \ B.
\]
The next lemma follows from the results in \Cref{sec:KSBA limit}. In particular, to find the stable replacements, it suffices to run an MMP for the generalized pair given by the family of curves over $B$.

\begin{lemma}\label{lemma_objects_in_the_boundary_when_fibers_have_coef_1_over_18}
    Let $p\in B$ be a closed point. Then the relative stable model is given by the following \textup{(ref. Table \ref{table:Relative stable model})}:
    \begin{enumerate}
        \item if $\xi^{-1}(p)\cong \bP^1$, then $(\xi')^{-1}(p)\cong (\xi)^{-1}(p)$;
        \item if $\xi^{-1}(p)$ is as in \Cref{prop:classfication_of_quasimaps_1_4}(2), then $(\xi')^{-1}(p)\simeq (\bP^1\times \bP^1,\frac{1}{2}D),$ whose boundary part is supported at a single point with coefficient $\frac{1}{6}$;
        \item if $\xi^{-1}(p)$ is as in  \Cref{prop:classfication_of_quasimaps_1_4}(3), then $(\xi')^{-1}(p)\cong (\xi)^{-1}(p)$;
        \item if $\xi^{-1}(p)$ is as in \Cref{prop:classfication_of_quasimaps_1_4} (4), then $(\xi')^{-1}(p)\simeq (\bP^1\times \bP^1,\frac{1}{2}D)$, whose boundary part is supported at three points each with coefficient $\frac{1}{6}$;
        \item if $\xi^{-1}(p)$ is as in \Cref{prop:classfication_of_quasimaps_1_4}(5), then $(\xi')^{-1}(p)\simeq (\bP^1\times \bP^1,\frac{1}{2}D)$, whose boundary part is supported at two points each with coefficient $\frac{1}{6}$; and
        \item if $\xi^{-1}(p)$ is as in \Cref{prop:classfication_of_quasimaps_1_4}(6), then $(\xi')^{-1}(p)$ has two irreducible components, and is isomorphic to $(\xi)^{-1}(p)$ away from the component of $(\xi)^{-1}(p)$ which maps to the irreducible component of degree $2$.
    \end{enumerate}
    \end{lemma}
    
\begin{center}
\begin{longtable}{| p{.08\textwidth} | p{.50\textwidth}| p{.19\textwidth} |}
    \hline \textbf{$\xi^{-1}(p)$} & \textbf{$(\xi')^{-1}(p)$} & \textbf{boundary part}  \\ \hline
    \hline (1) &  $(\xi)^{-1}(p)\simeq (\bP^1\times\bP^1,\frac{1}{2}D)$ &  0 \\ \hline (2)   &  $(\bP^1\times\bP^1,\frac{1}{2}D)$  & $\frac{1}{6}p$  \\ \hline (3)   & $(\xi)^{-1}(p)$ & 0  \\ \hline (4)   &  $(\bP^1\times\bP^1,\frac{1}{2}D)$ &  $\frac{1}{6}(p_1+p_2+p_3)$ \\ \hline (5)   &  $(\bP^1\times\bP^1,\frac{1}{2}D)$  & $\frac{1}{6}(p_1+p_2)$  \\ \hline (6)   &  $(\xi)^{-1}(p)$ with degree $2$ component contracted  & $\frac{1}{6}p$  \\  \hline
    \caption{Relative stable model of $\xi$}
    \label{table:Relative stable model}  
\end{longtable}
    
\end{center}

\begin{proof}
It suffices to observe that this MMP contracts precisely all the \emph{degree $2$ tails} of $C\rightarrow B$, because of our choice of the coefficient $\frac{1}{9}+\epsilon$. Here, by \emph{degree two}, we mean those components $\Gamma$ of fibers of $C\rightarrow B$ satisfying that the morphism $\Gamma\rightarrow \bP^1$ induced by the $\cC\rightarrow \sP_1$ has degree $2$. In particular, the fibers of $\xi'$ have at most two irreducible components, and the map $X\dashrightarrow X'$ is an isomorphism along the non-normal locus of the fibers of $\xi'$.
\end{proof}

\begin{lemma}\label{lct}
    Let $(Y,\frac{1}{2}D)$ be a pair as in \Cref{lemma_objects_in_the_boundary_when_fibers_have_coef_1_over_18} cases (1), (2), (4) or (5). Then $[D]$ is GIT semistable.
\end{lemma}
\begin{proof}
    This follows from the classification given in Proposition \ref{prop:classification of GIT}. For any $(1,4)$-curve $C$ on $\bP^1\times\bP^1$ and any point $p\in \bP^1$, if $(\bP^1\times\bP^1,\frac{1}{2}C)$ is log canonical but $(\bP^1\times\bP^1,\frac{1}{2}C+p_1^{-1}(p))$ is not, then the local defining polynomial of $C$ along the fiber $p_1^{-1}(p)$ is one of the following: $$y-x^3,\ \ y-x^4, \ \ x(y-x^3), \ \ x(y-x^2), \ \ x^2y, \ \ x^2(y-x^2),$$ and the local log canonical threshold $\lct\big(\bP^1\times\bP^1,\frac{1}{2}C;p_1^{-1}(p)\big)$ is equal to $$\textstyle \frac{5}{6}, \ \ \frac{3}{4}, \ \ \frac{2}{3}, \ \ \frac{3}{4}, \ \ \frac{1}{2}, \ \ \frac{1}{2}$$ respectively. The only one which contributes $\frac{1}{6}$ to the boundary part is $y-x^3$, which gives a smooth point at $(0,0)$. Moreover, this point is of ramification index $3$ with respect to $p_1:C\rightarrow \bP^1$, and all the other ramification points of $p_1$ has index $2$. Therefore, $C$ is GIT semistable.
\end{proof}

\subsubsection{Step 2}\label{sec:step 2} Keep the notation in the last subsection. As $X$ and $X'$ are isomorphic over the nodal locus of $X'\to B$, then there is an open substack $\cU\subseteq \cC$ satisfying that
\begin{enumerate}
    \item the coarse space of $\cX_{\cU}:= \cU \times_\cC \cX$ is an open subscheme $X_{\cU}\subseteq X$, along which $X\dashrightarrow X'$ is an isomorphism, and
    \item the image of $X_{\cU}$ in $X'$ contains the nodal locus of the fibers of $X'\to B$.
\end{enumerate}
In other terms, in a neighborhood of the nodal locus of $X'\to B$, the family $X'\to B$ is the coarse space of a family of $\bP^1$s over $\cU$.
\[\begin{tikzcd}[ampersand replacement=\&]
	{\mathcal{X}_{\mathcal{U}}} \&\& {\mathcal{X}} \&\& X \&\& {X'} \\
	{\mathcal{U}} \&\& {\mathcal{C}} \&\& C \&\& {C'} \&\& B
	\arrow["{\operatorname{open}}"{description}, hook, from=1-1, to=1-3]
	\arrow[from=1-1, to=2-1]
	\arrow[from=1-3, to=1-5]
	\arrow[from=1-3, to=2-3]
	\arrow[dashed, from=1-5, to=1-7]
	\arrow[from=1-5, to=2-5]
	\arrow["{\zeta}"{description}, from=1-7, to=2-7]
	\arrow["{\xi'}"{description}, from=1-7, to=2-9]
	\arrow["{\operatorname{open}}"{description}, hook, from=2-1, to=2-3]
	\arrow["{\sigma_{\mathcal{U}}}"{description}, bend right =18pt, from=2-1, to=2-9]
	\arrow[from=2-3, to=2-5]
	\arrow["\sigma"{description}, bend right =12pt , from=2-3, to=2-9]
	\arrow[from=2-5, to=2-7]
	\arrow[from=2-7, to=2-9]
\end{tikzcd}\]

\begin{comment}
    We recall some of the notations below
\begin{enumerate}
    \item $\cC\to B$ is the original family of twisted curves, coming from $B\to \overline{\cQ}^n$,
    \item $C\to B$ is the coarse moduli space of $\cC\to B$,
    \item $C\to C'$ is the contraction we produced in \Cref{subsub_step_1}, and
    \item $\cU\hookrightarrow\cC$ is a neighbourhood of the nodes with a $\bmu_n$ stabilizer for $n=2$ or $n=4$.
\end{enumerate}
\end{comment}

\begin{defn}
    Define $\cC'\to B$ to be the family of twisted curves 
    \begin{enumerate}
        \item with $C'\to B$ as its coarse moduli space, and
        \item which is isomorphic to $\cC\to B$ over $\cU$.
    \end{enumerate}
Equivalently, one can define $\cC'\rightarrow B$ to be constructed by gluing $\cU\rightarrow B$ with $C'\setminus N$, where $N\subseteq C'$ is the singular locus of $C'\to B$. In other terms, we replace the nodal locus of $C'\to B$ with the corresponding twisted nodes of $\cC\to B$.

We can define a family of curves over $\cC'$, which we denote by $$\cX_{\cC'}\ \longrightarrow \ \cC',$$ by gluing $\cX_\cU\rightarrow \cU$ with $X'\setminus \zeta^{-1}(N)\rightarrow C'\setminus N$, where $\zeta\colon X'\to C'$ is the morphism above.
\end{defn}

We first study the local geometry of $\cU\to B$. By \cite[Proposition 2.2 (b)]{olsson2007log}, every geometric nodal point $\mtf{n}\in \cU$ admits two morphisms as below, which are the strict henselization inducing an isomorphism on residual gerbes at the closed points
\begin{equation}\label{equation_local_description_node}
\big[\spec\big(\cO_{B,\sigma_\cU(\mtf{n})}^{\sh}[x,y]/(xy-f)\big)/\bmu_n\big]\longleftarrow [\Spec\cO_{\cC,\mtf{n}}^{\sh}/\bmu_n] \longrightarrow \cC
\end{equation}
where $f$ is an element of $\cO_{B,\sigma(\mtf{n})}^{\sh}$, and a generator $\zeta \in \bmu_n$ acts trivially on $\cO_{B,\sigma(\mtf{n})}^{\sh}$ and acts by $$\zeta*x = \zeta \cdot x, \ \ \ \zeta*y=\zeta^{-1}\cdot y.$$
From the classification of the stabilizers of a point in $\cC'$ given in \Cref{prop:classfication_of_quasimaps_1_4}, we see that either $n=2$ or $n=4$. In particular, if we do not force that the maps in \Cref{equation_local_description_node} induce isomorphisms on the residual gerbes, we have the usual local description of a nodal singularity 

\begin{equation}\label{equation_local_description_node_2}
\begin{tikzcd}\spec\big(\cO_{B,\sigma_\cU(\mtf{n})}^{\sh}[x,y]/(xy-f))\longleftarrow \Spec\cO_{\cC,\mtf{n}}^{\sh} \longrightarrow \cC
\end{tikzcd}
\end{equation}
Although $f$ is not be unique, the following lemma asserts that the vanishing locus $\bV(f)$ is intrinsically defined. 

\begin{lemma}\label{lemma_nodal_locus_is_global}
    There exists a closed subscheme $Z\subseteq B$ such that for every local chart $\spec \cO_{B,p}^{\sh}$ near $p\in B$ as above, $Z$ agrees with $\bV(f)$.
\end{lemma}

\begin{proof} In fact, one can take $Z$ to be the scheme-theoretic image of the singular locus of the family of curves $\cC\rightarrow B$, which is defined by the first Fitting ideal of $\Omega_{\cC/B}$.
\end{proof}

Keeping the notation as above, we proceed as follows. 
\begin{notation}\textup{Let $\sB\to B$ be the $\bmu_2$-root stack of $B$ along $Z$.}\end{notation}This root stack replaces the local description around a node of \Cref{equation_local_description_node_2} as follows 

\begin{equation}
\big[\spec\big(\cO_{\sB,\sigma_\cU(\mtf{n})}^{\sh}[x,y]/(xy-r^2)\big)/\bmu_2\big]\longleftarrow \Spec\cO_{\sC,\mtf{n}}^{\sh}\longrightarrow \sC:=\cC\times_B \sB
\end{equation}
where $\cO_{\sB,\sigma_\cU(\mtf{n})}^{\sh}=\cO_{B,\sigma_\cU(\mtf{n})}^{\sh}[r]/(r^2-f)$, and $\bmu_2$ acts by $(-1)*r := -r$.

\begin{notation}
    \textup{We will denote by $\sC$ (resp. $\sC'$) the pull-back $\cC\times_B\sB$ (resp. $\cC\times_B\sB$). Let also $\sC''\rightarrow \sC'$ be the blow-up along the locus $\bV(x,y,r)$.}
\end{notation}
 Observe that $\sC''$ is intrinsically defined, as $(x,y,r)$ is the sum of the ideals $(x,y)$ and $(r)$, which are intrinsically defined as $(x,y)$ is the first fitting ideal of $\Omega^1_{\sC'/\sB}$ and $r$ is intrinsically defined from \Cref{lemma_nodal_locus_is_global}.
A quick computation shows that the blow-up along $(x,y,r)$ of $\spec(\cO_{\sB,\sigma_\cU(\mtf{n})}^{\sh}[x,y]/(xy-r^2))$ is covered by the following three charts:
    \begin{itemize}
        \item[(i)] $\spec(\cO_{\sB,\sigma_\cU(\mtf{n})}^{\sh}[u,v]/(uv-1))\cong \spec(\cO_{\sB,\sigma_\cU(\mtf{n})}^{\sh})\times \bG_m$ with $x=ur$ and $y=vr$;
        \item[(ii)]$\spec(\cO_{\sB,\sigma_\cU(\mtf{n})}^{\sh}[w,y]/(r-wy)$ with $x=w^2y$;
        \item [(iii)] $\spec(\cO_{\sB,\sigma_\cU(\mtf{n})}^{\sh}[z,x]/(r-zx))$ with $y=z^2x$.
    \end{itemize}
 
 It is straightforward to check that:
 \begin{enumerate}
     \item this blow-up replaces the singular locus consisting of a chain of two (stacky) rational curves with three (stacky) rational curves (see \Cref{fig:birational transformations of family}); and
     \item the relative coarse space $C''_{\sB}$ of $\sC''\rightarrow \sB$ is projective over $\sB$.
 \end{enumerate}
Observe that these procedures only change $\sC'$ over $\sU:=\cU\times_B\sB$. In particular, we can pull back $\cX_{\cC'}$ along these transformations. We set $\cX_{\sC'}:=\cX_{\cC'}\times_{\cC'}\sC'$ and $\cX_{\sC''}:=\cX_{\cC'}\times_{\cC'}\sC''$. The subscripts $\cC'$, $\sC'$ and $\sC''$ denote the family of curves, over which $\cX_{\cC'}$, $\cX_{\sC'}$ and $\cX_{\sC''}$ are fibered respectively. Denote by $Y\to \sB$ the relative coarse moduli space of $\cX_{\sC''}\to \sB$. See \Cref{diagram_root_stack} and \Cref{fig:birational transformations of family} for a diagram including all these transformations.

The same modification goes with the divisor $(\frac{1}{2}+\epsilon)D'$, and we denote the resulting divisor by $(\frac{1}{2}+\epsilon)D_Y$ on $Y$.
Observe that the surfaces on the fibers of $Y\to \sB$ are obtained by introducing an isotrivial family on the nodal locus.
\begin{equation}\label{diagram_root_stack}
    \begin{tikzcd}[ampersand replacement=\&]
	\&\&\&\& {\mathcal{X}_{\mathscr{C}''}} \&\&\& {X_{C''}=:Y} \\
	\&\&\& {\mathscr{C}''} \& {} \&\& {C''} \\
	\&\&\&\& {\mathcal{X}_{\mathscr{C}'}} \\
	{\mathscr{U}'} \&\&\& {\mathscr{C}'} \& {} \&\& {\mathscr{B}} \\
	\&\&\&\& {\mathcal{X}_{\cC'}} \\
	{\mathcal{U}'} \&\&\& {\mathcal{C}'} \&\&\& B
	\arrow["{\textup{relative cms}}", from=1-5, to=1-8]
	\arrow[from=1-5, to=2-4]
	\arrow[no head, from=1-5, to=2-5]
	\arrow[from=1-8, to=2-7]
	\arrow[from=1-8, to=4-7]
	\arrow["{\textup{relative cms}}"{description}, from=2-4, to=2-7]
	\arrow["{\textup{blow-up}}"{description}, from=2-4, to=4-4]
	\arrow[from=2-5, to=3-5]
	\arrow[from=2-7, to=4-7]
	\arrow[from=3-5, to=4-4]
	\arrow[no head, from=3-5, to=4-5]
	\arrow[hook, from=4-1, to=4-4]
	\arrow[from=4-1, to=6-1]
	\arrow[from=4-4, to=4-7]
	\arrow[from=4-4, to=6-4]
	\arrow[from=4-5, to=5-5]
	\arrow["{\mu_2-\textup{root}}"{description}, from=4-7, to=6-7]
	\arrow[from=5-5, to=6-4]
	\arrow["{\textup{open neighborhood}}", hook, from=6-1, to=6-4]
	\arrow[from=6-4, to=6-7]
\end{tikzcd}
\end{equation}

\begin{figure}
    \centering
    \includegraphics[width=.7\linewidth]{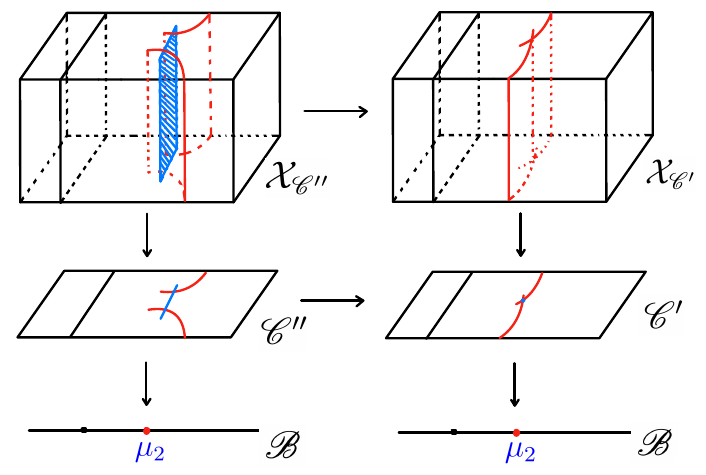}
    \caption{Birational modification performed in \Cref{diagram_root_stack}. The curve in blue is the exceptional divisor for $\sC''\to \sC$, whereas $\cX_{\sC''}$ is the pull-back of $\cX_{\sC'}\to \sC'$ via $\sC''\to \sC'$.}
    \label{fig:birational transformations of family}
\end{figure}

\begin{remark}
    \textup{We make a few comments on the diagram in  \Cref{diagram_root_stack}.
    \begin{enumerate}
        \item All squares are Cartesian except those with an arrow labeled by \emph{relative cms}, which stands for relative coarse moduli space.
        \item The most efficient way to read the diagram is to first ignore the family of surfaces, and focus on the family of curves, as all the action happens on $\cC'$.
        \item We summarize the modifications for the reader's convenience.
        \begin{itemize}
            \item We first perform a root stack $\sB\rightarrow B$ to get $\sC'$ by taking fiber product with $\cC'\rightarrow B$, and then we perform a blow-up $\sC''\rightarrow \sC'$.
          \item The family of surfaces $\cX_{\sC''}$ is the pull-back $\cX\times_{\cC'}\sC''$.
        \item Take the relative coarse moduli space of $\cX_{\sC''}\to \sB$, denoted by $Y\rightarrow \sB$. 
        \end{itemize}
    \item The step of root stack guarantees that the fiber of $\cX_{\sC''}\rightarrow \sB$ is reduced: indeed, blowing up a node on the fiber, which is a smooth point of the total space, leads to a non-reduced fiber on the new family.
     \end{enumerate}}
\end{remark}

\subsubsection{Step 3}
Finally, we perform some birational transformations for the family $$\textstyle \big(Y,(\frac{1}{2}+\epsilon)D_Y\big) \ \longrightarrow \ \sB$$ to construct the morphism $\sB^n\to \cM^{\GIT}$ from the normalization $\sB^n\to \sB$. As we will only work with the normalization of $\sB$, we will abuse notation and denote $\sB^n$ by $\sB$.

We will perform a birational transformation which (possibly after a sequence of flips) contracts the irreducible components of the singular fibers of $Y\to \sB$ which lie over rational tails (ref. red components in \Cref{fig:birational transformations of family}). To this end, we first replace $\sB$ with an \'etale cover $\cV\to \sB$, where we can add an appropriate ample divisor $H$ on $Y|_\cV$, and perform the desired birational contraction on $Y|_\cV$ via a relative MMP over $\cV$.
Then we show that the birational contraction above glues.

Let $C''\to \sB$ be the family of nodal curves as above. We denote by $C_\cV$ the pull-back of $C''$ to an \'etale cover $\cV\to \sB$ so that there are 9 disjoint sections, which we denote by $\{\sigma_i\colon \cV\to C_\cV\}_{i=1}^9$. Assume that each irreducible component of a singular fiber of $C_\cV\to \cV$ has three markings. Set $Y_\cV:= Y\times_{\sB}\cV$ and $\pi_\cV\colon Y_\cV\to C_\cV$ We will denote by $D_{Y_{\cV}}$ the pull-back of $D_Y$ to $Y_\cV$. Then \[H\ :=\ \pi_\cV^{-1}(\sigma_1(\cV) + \ldots +\sigma_9(\cV))+\epsilon D_{Y_{\cV}}\] is ample over $\cV$, and the family
$$\textstyle (Y_\cV,\frac{1}{2}D_{Y_{\cV}} + H)\ \longrightarrow \ \cV$$
is KSBA-stable. By \cite[Theorem 1.5]{meng2023mmp}, one can run an MMP for $\tau\colon (Y_\cV,\frac{1}{2}D_{Y_{\cV}})\ \longrightarrow \ \cV$ with a scaling by $H$. Recall that, if one denotes by $\cM^{\KSBA}_{I}$ the KSBA-moduli space of surface pairs with coefficients in a set $I$, one can obtain the end result of this MMP by composing the morphism $\cV\to \cM^{\KSBA}_{\frac{1}{2},1}$ induced by $\tau$, with the wall-crossing morphism  \[\cN^{\KSBA}_{\frac{1}{2},1} \ \longrightarrow \  \cM^{\KSBA}_{\frac{1}{2},t_0}\] for a certain $0<t_0<1$, where:
\begin{enumerate}
    \item we reduce the coefficients on $H$, and
    \item $\cN^{\KSBA}_{\frac{1}{2},1}$ is the normalization of the closure of the image of $\cV\to \cM^{\KSBA}_{\frac{1}{2},1}$.
\end{enumerate} Let
$(Y'_\cV, \frac{1}{2}D'_Y+t_0H')\to \cV$ be the family induced by $\cV\to \cM^{\KSBA}_{\frac{1}{2},t_0}$.

For every DVR $R$ and every morphism $\spec R\to \cV$ which sends the generic point to the locus parametrizing log canonical pairs, the family $$\textstyle (Y'_\cV, \frac{1}{2}D'_Y+t_0H')|_{\spec R} \ \longrightarrow \ \Spec R$$ is obtained by running an MMP on $(Y_\cV, \frac{1}{2}D_Y+t_0H)|_{\spec R}$.
Apply the results of \Cref{section_stable_reduction} to run this MMP: from how $H$ is defined, the MMP contracts the irreducible components of the singular fibers of $Y_\cV\to \cV$ which map to rational tails of $C_\cV$.

\begin{prop}
    The family $(Y_\cV',\frac{1}{2}D_Y')\to \cV$ descends to a family over $\sB$. 
\end{prop}

\begin{proof}
We use descent for polarized schemes. Observe that although $H'$ may not descend, the polarization $\cO_{Y'_\cV}(m(K_{Y'_\cV}+\frac{1}{2}D_Y'+t_0H'))$ does for $m\gg0$. In other terms, if we consider the two pull-backs of $Y'_\cV$ over the fiber product 
\[
\pi_1,\pi_2\colon \cV^{\times_\sB ^2}:=\cV\times_\sB\cV \ \longrightarrow \  \cV,
\]
we claim that there is an isomorphism \[Y'\times_{\cV,\pi_1} \cV^{\times_\sB ^2} \ \stackrel{\sim}{\longrightarrow} Y'\times_{\cV,\pi_2} \cV^{\times_\sB ^2}\]
which sends the polarization to itself. It is clear that there is such an isomorphism as both varieties $Y'\times_{\cV,\pi_i} \cV^{\times_\sB ^2}$ are obtained from $Y|_{ \cV^{\times_\sB ^2}}$ by contracting the same divisors, so they agree in codimension one. Moreover, they agree as pairs in codimension one, i.e. 
\[(Y'\times_{\cV,\pi_1} \cV^{\times_\sB ^2},D_\cV\times_{\cV,\pi_1} \cV^{\times_\sB ^2})\cong_{\text{codim. 1}}  (Y'\times_{\cV,\pi_2} \cV^{\times_\sB ^2},D_\cV\times_{\cV,\pi_2} \cV^{\times_\sB ^2}).\]
As $D_\cV\times_{\cV,\pi_i} \cV^{\times_\sB ^2}$ is ample over $\cV^{\times_\sB ^2}$ and $Y'\times_{\cV,\pi_1} \cV^{\times_\sB ^2}$ is $S_2$, the two pairs are isomorphic. Indeed, they are projectivizations of the same graded algebra over $\cV^{\times_\sB ^2}$.

This isomorphism preserves the polarization $H'$ as $H'$ is the sum of $\epsilon D'_Y$ and the pull-back of a $\bQ$-line bundle of fixed degree on a family of $\bP^1$s over $\cV$. Thus, by descent for polarized schemes
(ref. \cite[\href{https://stacks.math.columbia.edu/tag/0D1L}{Tag 0D1L}]{stacks-project}), there is a scheme $Y_{\sB}\to \sB$ as desired. Finally, descent for closed subschemes gives the family of divisors $D_{Y_{\sB}}$.
\end{proof}

Denote the resulting family over $\sB$ by $$\textstyle (Y_{\sB},\frac{1}{2}D_{Y_{\sB}}) \ \longrightarrow \ C_{\sB}  \ \longrightarrow \  \sB.$$ We conclude this subsection by proving \Cref{prop:existence of a morphism}.

\begin{proof}[Proof of \Cref{prop:existence of a morphism}]
   We first show that the fibers of $(Y_{\sB},D_{Y_{\sB}})\rightarrow \sB$ are GIT semistable. As $Y_{\sB}\rightarrow \sB$ is a family of $\bP^1\times \bP^1$, it suffices to check that every fiber of $D_{Y_{\sB}}\rightarrow \sB$ is a GIT semistable $(1,4)$-divisor. As the last step of MMP contracts all the degree $3$ tails of singular fibers of $C''\rightarrow \sB$, which contribute to the boundary part of the new fibers by $\frac{1}{4}$. Let $0\in \sB$ be a point such that $C''_0$ is such a fiber. Then $$\textstyle (Y_{\sB},\frac{1}{2}D_{Y_{\sB}})|_{0} \ \longrightarrow \ C_{\sB}|_{0}\simeq \bP^1$$ has exactly two fibers which contribute to the boundary part in the canonical bundle formula by $\frac{1}{4}$. By the classification of log canonical thresholds in the Proof of Lemma \ref{lct}, the local equations of the two fibers are both $y-x^4=0$, and hence one has $$(Y_{\sB},D_{Y_{\sB}})|_{0} \ \simeq \ (\bP^1\times\bP^1,\bV(x_0y_0^4-x_1y_1^4)),$$ which is GIT semistable.
\end{proof}

\subsection{Twisted stable reduction}\label{sec:twisted stable reduction}

In this section, we study the birational map $\mtc{M}^{\GIT}_{(1,4)}\dashrightarrow \cQ$. Although this is not a birational contraction, we understand the stable reduction: given a one-parameter family of GIT semistable, or equivalently K-semistable, pairs $$(\bP^1\times\bP^1\times T,\mts{C})\ \longrightarrow \ (0\in T)$$ such that all but the central fiber are twisted stable, one can perform an explicit birational modification to get a new twisted stable filling of the punctured family $$(\bP^1\times\bP^1\times T^{\circ},\mts{C}|_{T^{\circ}}) \ \longrightarrow  \ T^{\circ},$$ where $T^{\circ}= T\setminus\{0\}$. This is an analogue of DM stable reduction of one-parameter family of curves.

Before understanding the stable reduction of a one parameter family, we look at the geometry of the surface.

Let $X$ be the quadric surface $\bP^1_{[x,y]}\times\bP^1_{[u,v]}$, and $p=\bV(x,u)$, $q=\bV(y,v)$ be two points on $X$. Denote the ruling $\bV(x)$ (resp. $\bV(y)$) by $\sigma_p$ (resp. $\sigma_q$), and the ruling $\bV(u)$ (resp. $\bV(v)$) by $\ell_p$ (resp. $\ell_q$). Take the weighted blow-up $$\phi\ :\ \wt{X}\ \longrightarrow \ X$$ at $p$ and $q$ such that the weights of $(x,u)$ and $(y,v)$ are both $(1,4)$. Denote by $E_p$ (resp. $E_q$) the exceptional divisor over $p$ (resp. $q$), and by $\wt{\sigma}_p$ (resp. $\wt{\sigma}_q$, $\wt{\ell}_p$, $\wt{\ell}_q$) the strict transform of $\sigma_p$ (resp. $\sigma_q$, ${\ell}_p$, ${\ell}_q$). One has the intersection numbers $$\textstyle (E_p^2) \ =\ (E_q^2)\ =\ -\frac{1}{4}, \ \ \ \ \  (\wt{\sigma}_p^2) \ =\ (\wt{\sigma}_q^2)\ =\ -\frac{1}{4}, \ \ \ \ \ (\wt{\ell}_p^2) \ =\ (\wt{\ell}_q^2)\ =\ -4.$$ Contracting $\wt{\ell}_p$ and $\wt{\ell}_q$, one obtains a surface $\psi:\wt{X}\rightarrow \ove{X}$. 
\[\begin{tikzcd}
	{\bP^1\times\bP^1} && {\wt{X}} && {\ove{X}} \\
	&& {\bP^1}
	\arrow["\pi"{description}, from=1-1, to=2-3]
	\arrow["\phi"', from=1-3, to=1-1]
	\arrow["\psi", from=1-3, to=1-5]
	\arrow["{\wt{\pi}}"{description}, from=1-3, to=2-3]
	\arrow["{\ove{\pi}}"{description}, from=1-5, to=2-3]
\end{tikzcd}\]

Denote by $\ove{E_p}$ (resp. $\ove{E}_q$, $\ove{\sigma}_p$, $\ove{\sigma}_q$) the strict transform of $E_p$ (resp. $E_q$, $\wt{\sigma}_p$, $\wt{\sigma}_q$) on $\ove{X}$. Then $\ove{X}$ has two $\frac{1}{4}(1,1)$-singularities at $\ove{E}_p\cap \ove{\sigma}_q$ and $\ove{E}_q\cap \ove{\sigma}_p$; and has two $A_3$-singularities at $\ove{E}_p\cap \ove{\sigma}_p$ and $\ove{E}_q\cap \ove{\sigma}_q$. Moreover, the space $\NS(\ove{X})_{\bQ}$ is $2$-dimensional, with two generators $\ove{\sigma}_p$ and $\ove{E}_p$, which are both nef and effective $\bQ$-Cartier divisors with self-intersection numbers $0$. Moreover, the canonical divisor $-K_{\ove{X}}$ is ample, and hence a sufficiently large multiple of $\ove{E}_p$ is globally generated, which induces a contraction $\ove{X}\rightarrow \bP^1$ preserving all the $\ove{\pi}$-fibers.

\begin{figure}
    \centering
    \includegraphics[width=1.0\linewidth]{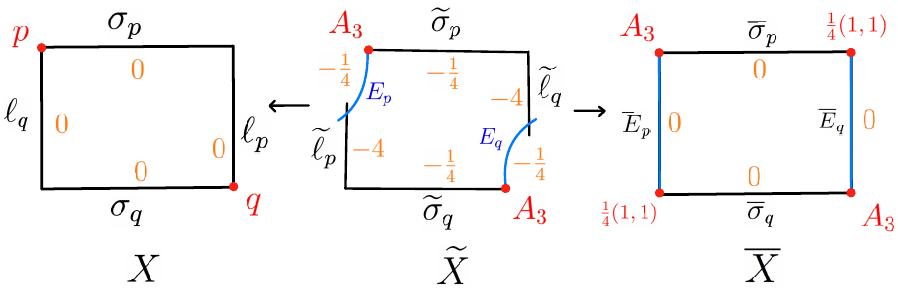}
    \caption{Birational modification of $\bP^1\times\bP^1$\\ orange = self-intersection number; red = singularity type}
    \label{fig:P1*P^1}
\end{figure}

Let $\bP(1,1,4)$ be the weighted projective plane with inhomogeneous coordinate $[s,w,z]$. Consider the linear series $\{az-bw^4\}_{[a,b]\in \bP^1}$, which has a base point at $p:=\bV(z,w)$. Take the weighted blow-up $\eta:Y\rightarrow \bP(1,1,4)$ at $p$ of $(z,w)$-weight $(1,4)$. Note that $Y$ has an $A_3$-singularity and a $\frac{1}{4}(1,1)$-singularity, and the strict transform of the linear series $\{az-bw^4\}_{[a,b]\in \bP^1}$ on $Y$ is base-point-free. Thus, there is a morphism $Y\rightarrow \bP^1$, such that the two singularities are contained in the same fiber.

\begin{figure}
    \centering
    \includegraphics[width=0.7\linewidth]{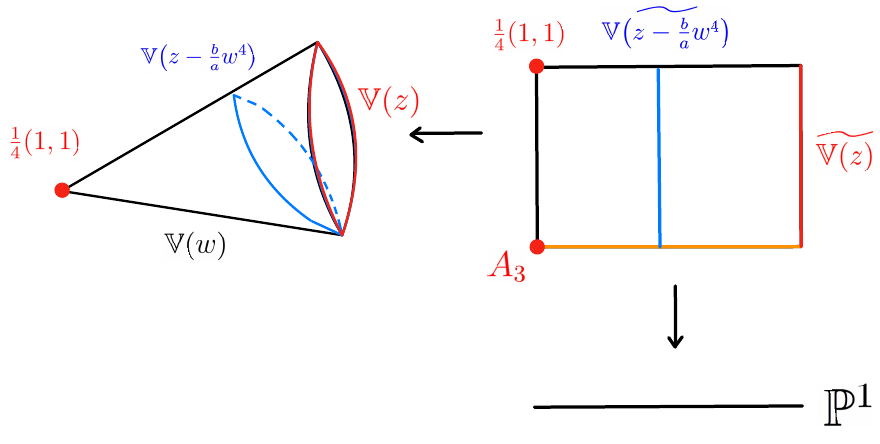}
    \caption{Birational modification of $\bP(1,1,4)$}
    \label{fig:P(1,1,4)}
\end{figure}

Now, let $C=\bV(uy^4-vx^4)$ be the curve on $\bP^1_{[x,y]}\times\bP^1_{[u,v]}$. Take a family $$\mts{Z}:=(\bP^1\times\bP^1\times T,\mts{C})\ \longrightarrow\ (0\in T)$$ such that $\mts{C}_0\simeq C$ and $\mts{C}_t$ is a curve such that $\mts{C}_t\rightarrow \bP^1$. Let $t$ be the local coordinate of $T$ near $0$, and take the weighted blow-up $\mts{X}\rightarrow \mts{Z}$ at $p:=\bV(u,x,t)$ and $q:=\bV(v,y,t)$ such that the weights of $(u,x,t)$ and $(v,y,t)$ are both $(4,1,1)$. Then the two exceptional divisors $S_p$ and $S_q$ are both isomorphic to $\bP(1,1,4)$. The central fiber $\mts{X}_0$ is the union $S_p\cup_{E_p} \wt{X} \cup_{E_q} S_q$. Then one can contract $\wt{\ell}_p$ (resp. $\wt{\ell}_q$) to the point $p'$ (resp. $q'$) in the threefold $\mts{X}$, and flip out two curves on $S_1$ and $S_2$. Finally, contracting the strict transform of $\wt{X}$ horizontally, one obtains a family $\ove{\mts{X}}\rightarrow T$. The covering stack $\mtc{X}$ of $\ove{\mts{X}}$ admits a $\bP^1$-fibration to the nodal twisted rational curve $\mtc{C}_1\cup_{\mtf{n}} \mtc{C}_2$ with a stabilizer $\mu_4$ at the node $\mtf{n}$, where $\mtc{C}_1$ and $\mtc{C}_2$ are two smooth stacky curves.

\subsection{K-moduli compactification}

In this section, we are aiming to show the following.

\begin{theorem}\label{thm:K=GIT}
    For any $0<c<\frac{1}{2}$, there is an isomorphism between stacks $$\mtc{M}^K_{(1,4)}(c) \ \simeq \ \bigg[\big|\mtc{O}_{\bP^1\times\bP^1}(1,4)\big|^{\sst}/ \PGL(2)\times\PGL(2)\bigg] \ =: \ \mtc{M}^{\GIT}_{(1,4)},$$ which descends to an isomorphism between their good moduli spaces $$\ove{M}^K_{(1,4)}(c) \ \simeq \ \big|\mtc{O}_{\bP^1\times\bP^1}(1,4)\big|^{\sst}\sslash \PGL(2)\times\PGL(2)\ =: \ \ove{M}^{\GIT}_{(1,4)}.$$
\end{theorem}

Let us start with lemmas proving that GIT (poly)stability implies K-(poly)stability.

\begin{lemma}\label{lem:polystable 1}
    Let $C_0:=\bV(x_0^4y_0-x_1^4y_1)$ be the GIT-polystable curve on $\bP^1_{[x_0:x_1]}\times\bP^1_{[y_0,y_1]}$. Then $(\bP^1\times\bP^1,cC_0)$ is K-polystable for any $0\leq c<\frac{1}{2}$.
\end{lemma}

\begin{proof}
    For simplicity, we denote by $(X,cD)$ the pair $(\bP^1\times\bP^1,cC_0)$. As $(X,cD)$ is a $\bT$-pair with complexity one, then we can apply criterion \cite[Theorem 1.31]{ACC23}. 

    Let $E$ be the exceptional divisor of the $(1,4)$-blowup $\pi:Y\rightarrow X$ at $p=([0,1],[1,0])$, where the weight of $(x_0,y_1)$ is $(1,4)$. Then one has $A_{X,cD}(E)=5-4c$. Let $f_1,f_2$ be the two rulings of $\bP^1\times\bP^2$ which pass through $p$ so that $D\sim f_1+4f_2$, and $\wt{f}_1,\wt{f}_2$ be their strict transforms respectively. Then $\wt{f}_1\sim \pi^*f_1-4E$ and $\wt{f}_2\sim \pi^*f_-E$ satisfy that $$(\wt{f}_1^2)=-4, \ \ \ \textup{and} \ \ \ (\wt{f}_2^2)=-\frac{1}{4}.$$ Write the Zariski decomposition of $-K_X-cD-tE$ as $$P(t)+N(t) \ = \  -K_X-cD-tE \ \sim_{\bR} \ (2-c)f_1+(2-4c)f_2-tE ,$$ where $P(t)$ and $N(t)$ are the positive and negative parts respectively. Then the effective threshold $T_{X,cD}(E)$ is $10-8c$, and  \begin{equation}\label{eq:zariski 1}
        P(t) \ = \ \begin{cases}
            -K_X-cD-tE & 0\leq t\leq 2-4c \\
            -K_X-cD-tE-\frac{1}{4}(t-2+4c)\wt{f}_1 & 2-4c\leq t\leq 8-4c \\
            -K_X-cD-tE-\frac{1}{4}(t-2+4c)\wt{f}_1-(t-8+4c)\wt{f}_2 & 8-4c\leq t\leq 10-8c           
        \end{cases}.
    \end{equation}
Setting $u:=t-(2-4c)$ and $v:=t-(8-4c)$, the Equation (\ref{eq:zariski 1}) simplifies as \begin{equation}\nonumber
        P(t) \ = \ \begin{cases}
            (2-c)f_1+(2-4c)f_2-tE & 0\leq t\leq 2-4c \\
            (2-c-\frac{1}{4}u)f_1+(2-4c)f_2-(2-4c)E  & 0\leq u\leq 6 \\
            (2-4c-v)(\frac{1}{4}f_1+f_2-E) & 0\leq v\leq 2-4c          
        \end{cases}.
    \end{equation}
It follows that $$S_{X,cD}(E)\ =\ \frac{1}{2(2-c)(2-4c)}\left( \int_0^{2-4c}P(t)dt+ \int_0^{6}P(t)du+\int_0^{2-4c}P(t)dv\right),$$ where 
\begin{equation}\nonumber
\begin{split}
        \int_0^{2-4c}P(t)dt & \ =\     \int_0^{2-4c}2(2-c)(2-4c)-\frac{1}{4}t^2dt \\
        & \ =\ 8(2-c)(1-2c)^2-\frac{2}{3}(1-2c)^3;  
\end{split}
\end{equation}
\begin{equation}\nonumber
\begin{split}
        \int_0^{6}P(t)du & \ =\     \int_0^{6}2(2-c-\frac{1}{4}u)(2-4c)-\frac{1}{4}(2-4c)^2du \\
        & \ =\ 24(2-c)(1-2c)-18(1-2c)-6(1-2c)^2;  
\end{split}
\end{equation}
\begin{equation}\nonumber
\begin{split}
        \int_0^{2-4c}P(t)dv & \ =\     \int_0^{2-4c}\frac{1}{4}(2-4c-v)^2dv \ = \ \frac{2}{3}(1-2c)^3.  
\end{split}
\end{equation}
Therefore, one has that 
\begin{equation}\nonumber
    \begin{split}
        S_{X,cD}(E)& \ =\ \frac{1}{2(2-c)}\big( 4(2-c)(1-2c)+12(2-c)-9-3(1-2c) \big)\\
        & \ = \ 2(1-2c)+6-3 \ = \ 5-4c, 
    \end{split}
\end{equation}
and hence $\beta_{X,cD}(E)=0$. Now it suffices to check that $\beta_{X,cD}(F)>0$ for any vertical divisor $F$ (ref. \cite[Definition 1.26]{ACC23}) on $\bP^1\times\bP^1$ for any $0\leq c<\frac{1}{2}$. To do so, we only need to show that $\beta_{X,cD}(f_1),\beta_{X,cD}(f_2)$ and $\beta_{X,cD}(D)$ are all positive. One has that $A_{X,cD}(D)=1-c$ and \begin{equation}\nonumber
    \begin{split}
        S_{X,cD}(D)& \ =\ \frac{1}{2(2-c)(2-4c)}\int_0^{\frac{1}{2}-c}2(2-c-t)(2-4c-4t)dt\\
        & \ = \ \frac{(1-2c)(13-4c)}{12(2-c)},
    \end{split}
\end{equation}
and hence $\beta_{X,cD}(D)>0$ if and only if $4c^2-6c+11>0$, which holds for any $c>0$. Similarly, one has $A_{X,cD}(f_1)=1$ and 
\begin{equation}\nonumber
    \begin{split}
        S_{X,cD}(f_1)& \ =\ \frac{1}{2(2-c)(2-4c)}\int_0^{2-c}2(2-c-t)(2-4c)dt\\
        & \ = \ 1-\frac{c}{2},
    \end{split}
\end{equation}
and hence $\beta_{X,cD}(f_1)>0$ for any $c>0$; one has $A_{X,cD}(f_2)=1$ and 
\begin{equation}\nonumber
    \begin{split}
        S_{X,cD}(f_2)& \ =\ \frac{1}{2(2-c)(2-4c)}\int_0^{2-4c}2(2-c)(2-4c-t)dt\\
        & \ = \ 1-2c,
    \end{split}
\end{equation}
and hence $\beta_{X,cD}(f_2)>0$ for any $c>0$.
\end{proof}

\begin{remark}
    \textup{Alternatively, one can apply \cite{Zhu21} to prove the K-polystability of $(X,cD)$. Notice that there is an action of $\mu_2 \rtimes \bG_m$, where $\mu_2$ acts by swapping $x_0$ with $x_1$ and swapping $y_0$ with $y_1$. This action has no fixed point, and the only fixed curve is $D$.} 
\end{remark}

The same argument proves the following.

\begin{lemma}\label{lem:polystable 2}
     Let $C_1:=\bV(x_0x_1(x_0^2y_0-x_1^2y_1))$ be the GIT-polystable curve on $\bP^1_{[x_0:x_1]}\times\bP^1_{[y_0,y_1]}$. Then $(\bP^1\times\bP^1,cC_1)$ is K-polystable for any $0\leq c<\frac{1}{2}$.
\end{lemma}

\begin{lemma}\label{lem:polystable 3}
    Let $C$ be a GIT stable $(1,4)$-curve on $\bP^1\times\bP^1$. Then $(\bP^1\times\bP^1,cC)$ is K-stable for any $0< c<\frac{1}{2}$.
\end{lemma}

We will apply \cite[Theorem~3.4]{AZ22} to verify K-(semi/poly)stability via admissible flags. We do not reproduce the necessary preliminaries here, as they are presented in full detail in \cite[Section~1.7]{ACC23} and \cite{Xu25}, to which we refer the reader.
 
\begin{proof}
   For any point $x\in \bP^1\times\bP^1$, denote by $f_1$ and $f_2$ the two rulings of $\bP^1\times\bP^1$ so that $D\sim f_1+4f_2$. Choose the flag $$x\ \subseteq \ f_2 \ \subseteq \ \bP^1\times\bP^1.$$ It is easy to see that $A_{X,cC}(f_2)=1$ and $S_{X,cC}(f_2)=1-2c$, and thus $\beta_{X,cC}(f_2)>0$. We may assume that $x\in C$. Then $A_{f_2,cC|_{f_2}}=1-c$ and $$S_{f_2,cx}(W_{\bullet,\bullet},x)=\frac{1}{2(2-c)(2-4c)}\int_0^{2-4c}dt\int_0^{2-c}(2-c-u)du \ = \ \frac{2-c}{4}.$$ Therefore, one has $\delta_x(X,cC)>1$ and $(X,cC)$ is K-stable for any $0<c<\frac{1}{2}$.
\end{proof}

\begin{proof}[Proof of Theorem \ref{thm:K=GIT}]
    The proof proceeds in two steps: to prove the isomorphism for $0<c=\epsilon\ll1$ (cf. \cite[Theorem 1.11]{LZ24}) and to show that there are no wall crossings when varying the coefficient $c$. For the wall-crossing of K-moduli spaces where the boundary divisors and canonical divisors are not proportional, see \cite{LZ24b,LZ24}. 
    
    We may assume that $\epsilon$ is a rational number. By exactly the same argument as in \cite[Theorem 5.2]{ADL21}, one can show that for any $(X,\epsilon D)\in \mtc{M}^K_{(1,4)}(\epsilon)$, $X$ is isomorphic to $\bP^1\times\bP^1$, and $D$ is a $(1,4)$-curve; moreover, the K-(semi/poly)stability of $(X,\epsilon D)$ is equivalent to the GIT (semi/poly)stability of $[D]$. In particular, by the universality of the K-moduli stack, there exists a morphism $$f:\mtc{M}^{\GIT}_{(1,4)}\ \longrightarrow \ \mtc{M}^K_{(1,4)}(\epsilon),$$ which induces a bijection on their good moduli spaces $$h:\ove{M}^{\GIT}_{(1,4)}\ \longrightarrow \ \ove{M}^K_{(1,4)}(\epsilon).$$ To show that $f$ is an isomorphism, let us construct its inverse. Let $(\mts{X},\epsilon\mts{D})\rightarrow \mtc{M}^K_{(1,4)}(\epsilon)$ be the universal family, and $\mtc{M}^K_{(1,4)}(\epsilon)\rightarrow \cB G$ be the morphism given by $$[(\mts{X}_S,\mts{D}_S)\rightarrow S]\ \mapsto\  [\mts{X}_S \rightarrow S],$$ which is representable since any automorphism of $\bP^1\times \bP^1$ preserving $D$ is contained in $\PGL(2)\times\PGL(2)$. Let $Z:=\mtc{M}^K_{(1,4)}(\epsilon)\times_{\cB G}\Spec \bC$ be the algebraic space, and the pull-back family $\mts{X}_Z\rightarrow Z$ is a trivial fibration by $\bP^1\times\bP^1$. Then the family $(\mts{X}_Z,\mts{D}_Z)\rightarrow Z$ induces a morphism $g:Z\rightarrow |\mtc{O}_{\bP^1\times\bP^1}(1,4)|^{\sst}$, which is $G$-equivariant. Therefore, $g$ descends to a morphism $$\mtc{M}^K_{(1,4)}(\epsilon)\ \longrightarrow \ \mtc{M}^{\GIT}_{(1,4)},$$ which is an inverse of $f$.
    
    The second step follows from Lemma \ref{lem:polystable 1}, Lemma \ref{lem:polystable 2} and Lemma \ref{lem:polystable 3} immediately: indeed, if there were a wall $c=c_0$ (may assume to be the first wall), then there exists a pair $(\bP^1\times\bP^1,cD)$ which is K-semistable for $c<c_0$ and K-unstable for $c>c_0$; however, every K-semistable pair $(X,\epsilon D)$ satisfies that $(X,cD)$ is K-semistable for any $0<c<\frac{1}{2}$.
\end{proof}

\begin{comment}
    As by \cite[Theorem 7.36]{Xu25}, the K-moduli stack $\mtc{M}^K_{(1,4)}(\epsilon)$ is isomorphic to the quotient stack $\big[Z^{\circ}_{\epsilon}/\PGL(N_m+1)\big]$ for some $m\gg0$ sufficiently divisible, where $Z^{\circ}_{\epsilon}$ is the K-semistable locus in the Hilbert scheme of embedded by the $m$-log-anticanonical divisors, and $N_m:=h^0(-m(K_{\bP^1\times\bP^1}+\epsilon D))-1$. Let $Z:=|\mtc{O}_{\bP^1\times\bP^1}(1,4)|^{\sst}$ be the open subscheme, and $$g:(\mts{X},\mts{D})\ =\ (\bP^1\times\bP^1\times Z,\mts{D})\ \longrightarrow \ Z$$ be the universal family, and $P\rightarrow Z$ be the $\PGL(N_m+1)$-torsor induced by the vector bundle $g_{*}\mtc{O}_{\mts{X}}(-m(K_{\mts{X}/Z}+\epsilon\mts{D}))$. This induces the following diagram \[\begin{tikzcd}
	P && {Z^{\circ}_{\epsilon}} \\
	Z && {[Z^{\circ}_{\epsilon}/\PGL(N_m+1)]}
	\arrow["\phi", from=1-1, to=1-3]
	\arrow[from=1-1, to=2-1]
	\arrow[from=1-3, to=2-3]
	\arrow["\psi", from=2-1, to=2-3]
\end{tikzcd},\] where $\phi$ and $\psi$ are both $G$-equivariant, and $\psi$ descends to $f$. Notice that the universal family $f_{\epsilon}:\mts{X}_{\epsilon}\rightarrow Z_{\epsilon}^{\circ}$ is isotrivial with fibers isomorphic to $\bP^1\times\bP^1$. Take an \'{e}tale cover $V\rightarrow Z_{\epsilon}^{\circ}$ such that $f_{\epsilon}$ trivializes, i.e. $\mts{X}_{\epsilon}\times_{Z_{\epsilon}^{\circ}}V\simeq (\bP^1\times\bP^1)\times V$
    
\end{comment}

For curves of larger bidegree $(d_1,d_2)$, the GIT moduli is not always isomorphic to the K-moduli spaces, and there do exist wall crossings for the K-moduli spaces, where the walls are usually irrational (but algebraic) numbers. Moreover, there are K-polystable log Fano pairs $(X,cD)$ parametrized by the K-moduli space $\ove{M}^K_{(d_1,d_2)}(c)$ such that $X$ itself is not Fano. We illustrate this by the following simple example.

Let $C_0:=4Q+\ell_1+\ell_2$ be a curve on $\bP^1\times\bP^1$ of bidegree $(4,6)$, where $Q$ is a smooth conic of bidegree $(1,1)$, and $\ell_1,\ell_2$ are two distinct lines. Then $(\bP^1\times\bP^1,cC_0)$ is a $\bT$-pair of complexity one.

\begin{lemma}
    The log pair $(\bP^1\times\bP^1,cC_0)$ is K-semistable for $0<c\ll 1$.
\end{lemma}

\begin{proof}
    There are two ways to see it. Either use the criterion for complexity one pairs, or use the GIT stability.
\end{proof}

\begin{lemma}\label{lem: necessary condition for K-ss}
    The log pair $(\bP^1\times\bP^1,cC_0)$ is K-unstable for any $c>c_0:=\frac{9-\sqrt{21}}{30}$.
\end{lemma}

\begin{proof}
    We compute the $\beta$-invariant with respect to the divisor $Q$. It is easy to see that $A_{\bP^1\times\bP^1,cC_0}(Q)=1-4c$ and \begin{equation}\nonumber
   \begin{split}
    S_{\bP^1\times\bP^1,cC_0}(Q)&=\frac{1}{\mtc{O}_{\bP^1\times\bP^1}(2-4c,2-6c)^2}\int_0^{2-6c}\mtc{O}_{\bP^1\times\bP^1}(2-4c-t,2-6c-t)^2dt\\
    &=\frac{1}{2-4c}\left(\frac{1}{3}(2-6c)^2-\frac{1}{2}(4-10c)(2-6c)+(2-4c)(2-6c)\right).
    \end{split}
    \end{equation}
    Therefore, $\beta_{\bP^1\times\bP^1,cC_0}(Q)\geq0$ is equivalent to $15c^2-9c+1\geq0$, which implies that $$c\ \geq\ c_0\ :=\ \frac{9-\sqrt{21}}{30}\ \approx\  0.147.$$
\end{proof}

\begin{lemma}
    There is an isotrivial degeneration from $(\bP^1\times\bP^1,C_0)$ to $(\bF_2,D_0)$ with $$D_0 \ := \ 4\mtf{e}_{\infty}+2\mtf{e}_0+\mtf{f}_1+\mtf{f}_2,$$ where $\mtf{e}_0$ is the negative section, $\mtf{e}_{\infty}$ is the infinity section, and $\mtf{f}_i$ are two distinct fibers.
\end{lemma}

\begin{proof}
    This is standard: taking the trivial family of $(\bP^1\times\bP^1,C_0)$ over $\bA^1$ and performing blow-up and blow-down on the central fiber.
\end{proof}

\begin{lemma}
    The pair $(\bF_2,cD_0)$ is K-semistable if and only if $c=c_0$. Moreover, $(\bF_2,c_0D_0)$ is a K-polystable toric log Fano pair.
\end{lemma}

\begin{proof}
    As $(\bF_2,cD_0)$ is a toric pair, it suffices to show that $$\beta_{\bF_2,cD_0}(\mtf{e}_{\infty})\ =\ \beta_{\bF_2,cD_0}(\mtf{f}_1)\ =\ 0$$ if and only if $c=c_0$.
    The computation of $\beta_{\bF_2,c_0D_0}(\mtf{e}_{\infty})=0$ is almost the same as Lemma \ref{lem: necessary condition for K-ss}. On the other hand, we have that $A_{\bF_2,cD_0}(\mtf{f}_1)=1-c$ and \begin{equation}\nonumber
   \begin{split}
    S_{\bF_2,cD_0}(\mtf{f}_1)&=\frac{1}{\big((2-6c)\mtf{e}+(4-10c)\mtf{f}\big)^2}\bigg(\int_0^{2c}\big((2-6c)\mtf{e}+(4-10c-t)\mtf{f}\big)^2dt\\
    & \ \ \ \ \  + \int_0^{4-12c}\big((2-6c-{\textstyle \frac{t}{2}})(\mtf{e}+2\mtf{f})\big)^2dt\bigg) \\
    &=\frac{1}{1-2c}\left(\frac{1}{3}(2-6c)^2-c(2-6c)+c(4-10c)-c^2 \right).
    \end{split}
    \end{equation}
    Therefore, $\beta_{\bF_2,cD_0}(\mtf{f}_1)=0$ is equivalent to $15c^2-9c+1=0$.
\end{proof}

\begin{corollary}
    The number $c_0=\frac{9-\sqrt{21}}{30}$ is a wall for the K-moduli space $\ove{M}^K_{(4,6)}(c)$, and $[(\bF_2,c_0D_0)]\in \ove{M}^K_{(4,6)}(c_0)$. 
\end{corollary}

\bibliographystyle{alpha}
\bibliography{v10}

\end{document}